\documentclass[11pt,a4paper]{article}
\usepackage[utf8]{inputenc}
\usepackage{amsfonts,amssymb,amsmath,amsthm,indentfirst}
\usepackage{epsfig}
\usepackage{enumerate}
\usepackage{amsfonts}
\usepackage{latexsym}
\usepackage{amssymb}
\usepackage{amsmath}
\usepackage{textcomp}
\usepackage{mathrsfs}
\usepackage{graphicx}

\usepackage{appendix}
\usepackage{enumitem}
\usepackage{mathtools} 

\usepackage{xcolor}
\usepackage{geometry}
\usepackage{hyperref}
\hypersetup{colorlinks=true}
\hypersetup{
    colorlinks=true,
    linkcolor=blue,
    filecolor=blue,      
    urlcolor=blue,
    anchorcolor=blue,
    citecolor={red},
}

\numberwithin{equation}{section}

\newtheorem{theorem}{Theorem}[section]
\newtheorem{definition}[theorem]{Definition}
\newtheorem{lemma}[theorem]{Lemma}
\newtheorem{proposition}[theorem]{Proposition}
\newtheorem{remark}[theorem]{Remark}

\newcommand{\norm}[1]{\left\Vert#1\right\Vert}
\newcommand{\abs}[1]{\left\vert#1\right\vert}
\newcommand{\fint}{\,-\mspace{-19.4mu}\int}
\title{\textbf{Existence of Vortex Patch Equilibria  for Active Scalar Equations}}

\begin{document}

\author{} 
\maketitle

\centerline{
         {\large Edison Cuba}\footnote{Departamento Académico de Matemáticas y Estadística, Universidad Nacional de San
Antonio Abad del Cusco, Perú. (\texttt{edison.cuba@unsaac.edu.pe}) }
     }
\medskip \medskip

\textbf{Abstract.}
In this paper, we investigate the existence of a finite number of vortex patches for the generalized surface quasi-geostrophic (gSQG) equations with \(\alpha \in [1,2)\), focusing on configurations that may rotate uniformly, translate, or remain stationary. Using a desingularization technique, we reformulate the problem to resolve singularities arising in the point vortex limit. Assuming a nondegenerate equilibrium of the point vortices, we apply the implicit function theorem to construct time-periodic solutions to the gSQG equations, offering asymptotic descriptions of the vortex patch boundaries.

\vskip 5mm

\vskip 5mm
\textbf{Math. Class. No.\,:}  35Q35, 76B03, 76B47.
\vspace{3mm}


\section{Introduction}\label{section1}

This study explores the dynamics of $N$-vortex patches for the generalized surface quasi-geostrophic (gSQG) equations. The gSQG equations represent a partial differential equation defined on the plane, frequently used to model atmospheric flows in areas outside the tropics. It is formulated as follows
\begin{align}\label{1-1}
	\begin{cases}
		\partial_t\theta+v\cdot \nabla \theta =0,&\text{in}\ \mathbb{R}^2\times (0,T)\\
	    v=\nabla^\perp\psi,&\\
    \psi=-(-\Delta)^{-1+\frac{\alpha}{2}}\theta,&\\
		\theta\big|_{t=0}=\theta_0, &\text{in}\ \mathbb{R}^2,
	\end{cases}
\end{align}
where \( 0 \leq \alpha < 2 \) and \(\perp\) denotes a counterclockwise rotation by an angle of \(\pi/2\) in the plane. In this system, \(\theta: \mathbb{R}^2\times(0,T)  \to \mathbb{R}\) represents the potential temperature of the fluid, referred to throughout this paper as the active scalar, and \(v: \mathbb{R}^2\times(0,T)  \to \mathbb{R}^2\) is the velocity field of the fluid. First, we define the nonlocal operator for \( 0 \leq \alpha < 2 \) as follows
\begin{equation*}
	(-\Delta)^{-1+\frac{\alpha}{2}}\theta(\boldsymbol  x)=\int_{\mathbb{R}^2}K_\alpha(\boldsymbol x- \boldsymbol y)\theta(\boldsymbol y)d\boldsymbol y,
\end{equation*}
where \( K_\alpha \) is the fundamental solution corresponding to the Green's function of the operator \( (-\Delta)^{-1 + \frac{\alpha}{2}} \) in \( \mathbb{R}^2 \), and is explicitly given by
\begin{equation}\label{eq:kernel}
K_\alpha(\boldsymbol  x )=\left\{
	\begin{array}{lll}
		\frac{1}{2\pi}\ln \frac{1}{|\boldsymbol x|}, \ \ \ \ \ \ \ \ \ \ \ \ \ \ \ \ \ \ \ \ \ & \text{if} \ \ \alpha=0,\\
		\frac{C_\alpha}{2\pi}\frac{1}{|\boldsymbol x |^\alpha}, \ \ \ C_\alpha=\frac{\Gamma(\alpha/2)}{2^{1-\alpha}\Gamma(\frac{2-\alpha}{2})}, & \text{if} \ \ \ 0<\alpha<2.
	\end{array}
	\right.
\end{equation}
here \(\Gamma\) denotes the Euler gamma function. The system simplifies to the 2D incompressible Euler equations when \(\alpha = 0\), and for \(\alpha = 1\), it represents the surface quasi-geostrophic (SQG) equations. The complete range \( \alpha \in [0,2) \) was introduced through a series of works. The scenario where \( \alpha \in [0,1] \) was investigated by C\'{o}rdoba et al. \cite{cordova}, while the more singular velocity fields corresponding to \( \alpha \in [1,2) \) were examined by Chae et al. \cite{chae}. This work specifically addresses the latter scenario, focusing on \( \alpha \in [1,2) \).

In recent years, there has been a growing interest in the mathematical study of active scalar equations. While the question of global existence remains largely unsolved, the case of the  Euler equations is an exception, where the global well-posedness of classical solutions in both the entire plane \(\mathbb{R}^2\) and any smooth bounded domain \(\Omega\) is well-established, see \cite{MP94}.
 Yudovich \cite{Yud} provided a proof of global well-posedness for initial data in \(L^1 \cap L^\infty\). However, extending this theory across the full range of \(\alpha \in (0,2)\) proves to be challenging due to the increasingly singular nature of the velocity, which falls below the Lipschitz class. Consequently, the question of global well-posedness for the entire interval \(\alpha \in (0,2)\) remains unresolved. The issue of local well-posedness within Sobolev spaces has been investigated in \cite{chae} for the whole plane and in \cite{ConN18a} for smooth bounded domains. It is also known that \(L^2\)-weak solutions to the gSQG equation exist globally in time, as demonstrated in \cite{LX19,Mar08,Resnick} for the full space and \cite{ConN18b,NHQ18} for smooth bounded domains. More specifically, considerable progress has been made regarding a particular class of vortices referred to as \textit{$\alpha$--patch solutions}. These solutions arise from initial data represented by the characteristic function of a bounded domain \(D\), specifically \(\theta_0(\mathbf{x}) = \chi_D(\mathbf{x})\), and their boundary motion \(D_t\) is described through the contour dynamics formulation. In fact, for the case \(0 < \alpha < 2\), the velocity field can be reconstructed using the Biot-Savart law and the Green-Stokes theorem as follows
\begin{equation*}
	v(\boldsymbol{x}, t) = 
	\frac{ C_\alpha}{2\pi} \int_{\partial D_t} \frac{1}{|\boldsymbol{x} - \boldsymbol{y}|^\alpha} d\boldsymbol{y}, \quad 0 < \alpha < 2.
\end{equation*}
Assuming the patch boundary \(\partial D_t\) is parameterized as \(z(\xi,t)\) with \(\xi \in [0, 2\pi)\), it follows that \(z(\xi,t)\) satisfies
\begin{equation*}
	\partial_t z(\xi,t) = 
	\frac{ C_\alpha}{2\pi} \int_0^{2\pi} \frac{\partial_\tau z( \tau,t)}{|z(\xi,t) - z(\tau, t)|^\alpha} d\tau, \quad 0 < \alpha < 1,
\end{equation*}
commonly referred to as the contour dynamics equation.  For the case \(1 \leq \alpha < 2\), the integral diverges. To address this singularity, we can adjust the velocity at the boundary by subtracting a tangential vector, leading to the definition
\begin{equation*}
	\partial_t z(\xi,t ) = \frac{ C_\alpha}{2\pi} \int_0^{2\pi} \frac{\partial_\tau z(\tau, t) - \partial_\xi z(\xi,t)}{|z(\xi,t) - z(\tau, t)|^\alpha} d\tau, \quad 1 \leq \alpha < 2.
\end{equation*}

 
 For the Euler equations, global persistence of boundary regularity has been demonstrated in \cite{B-C,C}. On the other hand, the situation for the gSQG equation with \(\alpha \in (0,2)\) is more complex. For this range of \(\alpha\), only local-in-time persistence of regularity in Sobolev spaces has been proven, as discussed in \cite{chae,Gan08,Rod05}. Furthermore, results concerning finite-time singularities involving multi-signed patches in the half-plane have been achieved across various ranges of \(\alpha\), as noted in \cite{Gancedo-1,KRYZ,KYZ17}. 
 

The investigation of the existence of time-periodic solutions for active scalar equations \eqref{1-1}, commonly referred to as relative equilibria or V-states has a rich history and continues to be a dynamic area of research,  see, e.g. \cite{edison3,cao,Cas1,edison2,de1,Serr,Gom,HH2,multipole,HH15,Hmidi-Mateu,hmidi2}. Although these systems may appear to have straightforward flow configurations characterized by either rotating or traveling motions, they possess remarkable complexity and exhibit intricate dynamics. The earliest known example of rotating patches in the Euler equations was provided by Kirchhoff \cite{Kirch}, who demonstrated that an ellipse with semi-axes \(a\) and \(b\) rotates uniformly with an angular velocity of \(\Omega = \frac{ab}{a^2 + b^2}\). Almost a century later, Deem and Zabusky \cite{DZ78} performed numerical simulations that illustrated the existence of  V-states with \(m\)-fold symmetry. Burbea \cite{Bur} later provided an analytical confirmation of this finding using bifurcation theory, revealing that bifurcation from the Rankine vortices (in the radial case) occurs at angular velocities of \(\Omega = \frac{m-1}{2m}\) for \(m \geq 2\). Subsequently, Hmidi, Mateu and Verdera \cite{hmidi2} established the \(C^\infty\) regularity and the convexity of the boundaries of the bifurcated V-states near the Rankine vortex for the Euler equations. Additionally, the analyticity of the boundary was explored further by Castro et al. in \cite{CCG16b} across the entire range of \(\alpha \in [0,2)\). The study of V-states for the gSQG model in the entire plane was first conducted by Hassainia and Hmidi \cite{HH15}, who confirmed results analogous to those of Burbea for all \(\alpha \in (0,1)\). Later, Castro et al. \cite{Cas1} extended this framework to the range \(\alpha \in [1,2)\) and established the \(C^\infty\) boundary regularity, with further confirmation of the real analyticity of the V-states boundary presented in \cite{CCG16b}. The examination of V-states in radial domains with rigid boundaries was obtained by De la Hoz et al. for the 2D Euler equation in \cite{DHHM} and subsequently expanded by Hmidi et al. \cite{HXX23} for the gSQG equations with \(\alpha \in (0,1)\). Results concerning with the existence of solutions to doubly connected V-states  were obtained in \cite{de1,HMV}.

In addition to the previously mentioned results concerning vortex patches for active scalar equations, another notable category is that of asymmetric vortex patches, which holds considerable physical significance. These configurations have been numerically explored by Dritschel \cite{Dritschel}. Examining asymmetric patches presents challenges associated with the nonlocal velocity field and the contour dynamics equation related to \eqref{1-1}. In the context of the 2D Euler equations, Hassainia and Hmidi \cite{HH2} investigated the asymmetric scenario and established the existence of co-rotating and traveling asymmetric vortex pairs. For the gSQG equations with \(\alpha \in [0,1)\), the work presented in \cite{multipole} established the existence of a finite set of multipole vortex patches, assuming the point vortex equilibrium is non-degenerate. This study provided instances of asymmetric vortex patch pairs that rotate and travel, along with asymmetric stationary tripoles. The authors in \cite{Gomez-Jaemin-Jia2021} investigates stationary solutions that feature multiple multi-layered patches in the Euler equations. Results regarding the asymmetric configuration for the case   $\alpha\in[1,2)$ were presented in \cite{edison2}, where the authors established the existence of solutions for co-rotating and traveling asymmetric vortex pairs for \eqref{1-1}. Additional results concerning the existence of stationary  multiple vortex patches have been established in bounded domains, with symmetric vortex patches for the SQG equation (\(\alpha = 1\)) reported in \cite{cao}, while the existence of stationary multiple symmetric vortex patches for \(\alpha \in (1,2)\) were found in \cite{edison3}.

The goal of this paper is to obtain the existence of rigidly rotating, translating and stationary vortex patch equilibria for the gSQG equations \eqref{1-1} with \(\alpha\in[1,2)\), by desingularizing arbitrary non-degenerate equilibria of finitely many point vortices.
For a general configuration the problem can be reduced, through a Lyapunov--Schmidt argument, to a finite-dimensional nonlinear equation. Under a natural non-degeneracy condition on the point vortex configuration one can instead directly apply a modified version of the implicit function theorem, as in \cite{multipole}. Our results extend those of \cite{multipole}, where \(\alpha\in[0,1)\), to the more singular range \(\alpha\in[1,2)\), for which the functional framework of \cite{Cas1} is needed. As applications, we recover within a unified framework the asymmetric co-rotating and traveling pairs obtained in \cite{edison2} (see also \cite{cao2} for related co-rotating and traveling solutions), and we obtain asymmetric stationary tripoles of vortex patches, which are new for the SQG and gSQG equations with \(\alpha\in[1,2)\). For comparison, refer to \cite{Gomez-Jaemin-Jia2021} for stationary solutions with multiple multi-layered patches in the Euler equations and \cite{Gom} for stationary doubly connected solutions to the gSQG equations with \(\alpha \in (0,2)\).

To present our main results, we first introduce some preliminaries that will be used frequently throughout this paper. For \(\boldsymbol x=(x_1,x_2)\in\mathbb R^2\) we write \(\boldsymbol x^\perp:=(-x_2,x_1)\), and we denote \(\mathbf e_1=(1,0)\), \(\mathbf e_2=(0,1)\). We consider the gSQG point vortex model, which describes \(N\) interacting vortices in the entire plane \(\mathbb{R}^2\). This model is governed by the following Hamiltonian system
\begin{align}\label{ode-sys0}
\frac{d}{dt}w_i(t)=\frac{\widehat{C}_\alpha }{2}\sum_{\substack{j=1, j\neq i}}^{N}\gamma_j   \frac{\big(w_i(t)-w_j(t)\big)^\perp}{|w_i(t)-w_j(t)|^{\alpha+2}},
\quad i=1,\ldots,N,
\end{align}
where \(w_1(t), \ldots, w_N(t)\) denote the positions of the vortices, and \(\pi\gamma_1, \ldots, \pi\gamma_N \in \mathbb{R} \setminus \{0\}\) represent their corresponding circulations. The constant \(\widehat{C}_\alpha\) is defined as
\begin{equation}\label{eqn:kalpha2}
\widehat{C}_\alpha := \alpha C_\alpha = \frac{2^\alpha \Gamma(1 + \alpha/2)}{\Gamma(1 - \alpha/2)}.
\end{equation}
In the case where \(\alpha = 0\), this system becomes the classical model of Eulerian point vortex interactions (with \(\widehat C_0=1\)). For an extensive discussion on the \(N\)-vortex problem and vortex equilibria, refer to \cite{Aref}  and for the gSQG equations  \cite{Ros}.

We concentrate on periodic solutions where the vortex configuration moves like a rigid body, namely
\begin{align}\label{eqn:ode}
\frac{d}{dt} w_i(t) = U\mathbf{e}_2 + \Omega\, w_i(t)^\perp,
\end{align}
where \(U \in \mathbb{R}\) denotes a constant (vertical) linear velocity and \(\Omega \in \mathbb{R}\) a constant angular velocity. These solutions are referred to as relative equilibria. By solving \eqref{eqn:ode}, we can set \(U = 0\) by shifting coordinates when \(\Omega \neq 0\). Consequently, we can focus on three types of equilibria. Rotating equilibria with \(\Omega \neq 0\) and \(U = 0\), traveling equilibria with \(U \neq 0\) and \(\Omega = 0\), or stationary equilibria where \(U = \Omega = 0\).

\noindent Defining \(w_i = (w_{i1}, w_{i2}) := w_i(0)\), inserting \eqref{eqn:ode} into \eqref{ode-sys0} and rotating the resulting identity by the angle \(-\pi/2\) (which maps \(\boldsymbol x^\perp\) to \(\boldsymbol x\) and \(\mathbf e_2\) to \(\mathbf e_1\)), the system \eqref{ode-sys0} becomes
\begin{align}\label{alg-sysP}
\mathcal{P}_i^\alpha(\lambda) &= \Omega w_i + U\mathbf{e}_1 - \frac{\widehat{C}_\alpha}{2} \sum_{\substack{j=1,\, j \neq i}}^{N} \gamma_j \frac{w_i - w_j}{|w_i - w_j|^{\alpha + 2}} = 0, \quad i = 1, \dots, N,
\end{align}
where \(\lambda = (w_{11}, \dots, w_{N1}, w_{12}, \dots, w_{N2}, \gamma_1, \dots, \gamma_N, \Omega, U)\). This defines a mapping \(\mathcal{P}^\alpha(\lambda)\) with values in \(\mathbb{R}^{2N}\).

\begin{definition}\label{def:non-deg}
\begin{enumerate}[label=\rm(\roman*)]
\item We refer to a rigidly rotating solution \(\lambda^*\) of \eqref{alg-sysP} (that is, \(\Omega^*\neq 0=U^*\)) or to a rigidly traveling solution (that is, \(U^*\neq0=\Omega^*\)) as non-degenerate if, after an appropriate reordering of the components of \(\lambda\), it can be written as
\begin{align}
    \label{non-deg-codim1}
    \lambda = (\lambda_1, \lambda_2), \quad \text{where} \quad \lambda_1 \in \mathbb{R}^{2N-1}, \quad \text{and} \quad \operatorname{codim} \operatorname{ran} D_{\lambda_1} \mathcal{P}^\alpha(\lambda^*) = 1,
\end{align}
and \(U\) (in the rotating case), respectively \(\Omega\) (in the traveling case), is one of the components of \(\lambda_2\).
\item Similarly, a stationary solution \(\lambda^*\) of \eqref{alg-sysP} (with \(\Omega^* = U^* = 0\)) is called non-degenerate if, after reordering the components of \(\lambda\), it can be expressed as
\begin{align}
    \label{non-deg-codim3}
    \lambda = (\lambda_1, \lambda_2), \quad \text{where} \quad \lambda_1 \in \mathbb{R}^{2N-3}, \quad \text{and} \quad \operatorname{codim} \operatorname{ran} D_{\lambda_1} \mathcal{P}^\alpha(\lambda^*) = 3,
\end{align}
and \(\Omega\) and \(U\) are components of \(\lambda_2\).
\end{enumerate}
\end{definition}
Since \(\mathcal P^\alpha\) takes values in \(\mathbb R^{2N}\), the conditions \eqref{non-deg-codim1} and \eqref{non-deg-codim3} imply that \(D_{\lambda_1}\mathcal P^\alpha(\lambda^*)\) is injective.

In summary, our first result can be informally expressed as follows. For a more thorough and detailed explanation, we refer the reader to Theorems~\ref{existence}, \ref{theorem-Phi-stationary}, \ref{existenceb} and \ref{theorem-Phi-stationaryb}, where a precise formulation is provided.

\begin{theorem}\label{thm:general}
Assuming \(\alpha \in [1,2)\), every non-degenerate solution \(\lambda\) of \eqref{alg-sysP} can be transformed into a family of vortex patch equilibria through a desingularization process. The equilibria are parameterized by a small parameter \(\varepsilon > 0\), which controls the size of the vortex patches.
\end{theorem}

As a direct application of the above theorem, we present our second main result.
\begin{theorem}\label{thm:informal-pair}
  Let \(\alpha \in [1,2)\), \(b_1, b_2, b_3 \in (0, \infty)\), \(\gamma\in\mathbb R\setminus\{0\}\), \(d>0\), \(0<|\mathtt{c}|<1\) and \(\mathtt{a}\in(0,1)\). The following results are established:

\begin{enumerate}[label=\rm(\roman*)]
\item For any sufficiently small \(\varepsilon > 0\), there exist two strictly convex domains \(\mathcal{O}^\varepsilon_1\) and \(\mathcal{O}^\varepsilon_2\) of class \(C^2\), symmetric with respect to the horizontal axis, and real numbers \(w_{11}(\varepsilon) = d + o(\varepsilon)\) and \(w_{21}(\varepsilon) = -\mathtt{c}d + o(\varepsilon)\), such that
\begin{equation*}
    \theta_{0}^\varepsilon = \frac{\mathtt{c}\gamma}{\varepsilon^2 b_1^2} \chi_{\mathcal{D}_1^\varepsilon} + \frac{\gamma}{\varepsilon^2 b_2^2} \chi_{\mathcal{D}_2^\varepsilon},
\end{equation*}
where
\begin{equation*}
    \mathcal{D}_1^\varepsilon := \varepsilon b_1 \mathcal{O}_1^\varepsilon + w_{11}(\varepsilon)\mathbf e_1, \quad \mathcal{D}_2^\varepsilon := \varepsilon b_2 \mathcal{O}_2^\varepsilon + w_{21}(\varepsilon)\mathbf e_1,
\end{equation*}
generates a co-rotating vortex pair for the equation \eqref{1-1} with angular velocity
\[\Omega^* = \frac{1}{2} \gamma \widehat{C}_\alpha d^{-\alpha-2} (1 + \mathtt{c})^{-\alpha-1} .\]

\item For any sufficiently small \(\varepsilon > 0\), there exist two strictly convex domains \(\mathcal{O}^\varepsilon_1\) and \(\mathcal{O}^\varepsilon_2\) of class \(C^2\), symmetric with respect to the horizontal axis, and real numbers \(w_{21}(\varepsilon) = -d + o(\varepsilon)\) and \(\gamma_1(\varepsilon) = -\gamma + o(\varepsilon)\) such that
\begin{equation*}
    \theta_{0}^\varepsilon = \frac{\gamma_1(\varepsilon)}{\varepsilon^2 b_1^2} \chi_{\mathcal{D}_1^\varepsilon} + \frac{\gamma}{\varepsilon^2 b_2^2} \chi_{\mathcal{D}_2^\varepsilon},
\end{equation*}
where
\begin{equation*}
    \mathcal{D}_1^\varepsilon := \varepsilon b_1 \mathcal{O}_1^\varepsilon + d \mathbf e_1, \quad \mathcal{D}_2^\varepsilon := \varepsilon b_2 \mathcal{O}_2^\varepsilon + w_{21}(\varepsilon) \mathbf e_1,
\end{equation*}
generates a vortex pair traveling in the vertical direction for the equation \eqref{1-1} with speed
\[U^* = \gamma \widehat{C}_\alpha 2^{-\alpha-2} d^{-\alpha-1} .\]

\item For any sufficiently small \(\varepsilon > 0\), there exist three strictly convex domains \(\mathcal{O}^\varepsilon_1\), \(\mathcal{O}^\varepsilon_2\), and \(\mathcal{O}^\varepsilon_3\) of class \(C^2\), symmetric with respect to the horizontal axis, and real numbers \(w_{31}(\varepsilon) = -\mathtt{a} + o(\varepsilon)\) and
\[
\gamma_2(\varepsilon) = -\gamma \left(\frac{\mathtt{a}}{\mathtt{a} + 1}\right)^{\alpha+1} + o(\varepsilon)
\]
such that
\begin{align*}
    \theta_{0}^\varepsilon &= \frac{\gamma}{\varepsilon^2 b_1^2} \chi_{\mathcal{D}_1^\varepsilon} + \frac{\gamma_2(\varepsilon)}{\varepsilon^2 b_2^2} \chi_{\mathcal{D}_2^\varepsilon} + \frac{\gamma \mathtt{a}^{\alpha+1}}{\varepsilon^2 b_3^2} \chi_{\mathcal{D}_3^\varepsilon}, \\
    \text{where} \quad \mathcal{D}_1^\varepsilon &= \varepsilon b_1 \mathcal{O}_1^\varepsilon + \mathbf e_1, \quad \mathcal{D}_2^\varepsilon = \varepsilon b_2 \mathcal{O}_2^\varepsilon, \quad \mathcal{D}_3^\varepsilon = \varepsilon b_3 \mathcal{O}_3^\varepsilon + w_{31}(\varepsilon)\mathbf e_1,
\end{align*}
generates a stationary vortex tripole for the equation \eqref{1-1}.
  \end{enumerate}
  \end{theorem}
The domains \(\mathcal O_i^\varepsilon\) are small perturbations of the unit disc, with explicit leading-order deformations given in Theorems~\ref{existence} and \ref{existenceb}. Note that we arrive at vortex point patches when $\varepsilon = 0$. Additionally, by setting some of $b_1$, $b_2$, or $b_3$ to zero, we can recover configurations that involve a combination of vortex patches and point vortices.

The paper is organized as follows. In Section~\ref{section2} we derive the boundary equations that describe the vortex equilibria and introduce the function spaces. Section~\ref{section3} is devoted to the regularity of the associated functional and to its linearization at the point vortex configuration. In Section~\ref{section4} we establish integral identities satisfied by the functional and prove the main existence results by means of a modified implicit function theore. We also prove the convexity of the patches. Finally, Section~\ref{section5} contains the applications to asymmetric pairs and tripoles.

\section{Time-periodic vortex patch models}\label{section2}

In this section, we examine a configuration of a finite number of point vortices that exhibit co-rotating or traveling motion. Utilizing methods established in \cite{Cas1,edison2,multipole}, we begin by deriving the contour dynamics equations that govern the steady states of \(N\) vortex patches associated to the gSQG equations \eqref{1-1} with $ 1\leq \alpha<2$. Next, we specify the suitable function spaces required to ensure that the problem is well-posed.

\smallskip

\subsection{Co-rotating vortex patches}

In this section, we begin by examining \(N\) co-rotating patches, represented by simply connected domains \(\mathcal{O}_i^\varepsilon\) for \(i = 1, \ldots, N\), within the context of the gSQG equations \eqref{1-1},
 with $\alpha\in[1,2)$. To facilitate our discussion, we will establish the following conventions and notation. We define \(N\) simply connected domains, denoted as \(\mathcal{O}_i^\varepsilon\) for \(i = 1, \ldots, N\), which are close to the unit disk and are contained within a disk of radius 2, with both disks centered at the origin.
Let \( b_i \in (0, \infty) \), \( w_i \in \mathbb{R}^2 \), and \( \varepsilon \in (0, \varepsilon_0) \). We introduce the domains
\begin{equation}\label{Dj} 
\mathcal{D}_{i}^\varepsilon := \varepsilon b_i \mathcal{O}_i^\varepsilon + w_i, 
\end{equation}
where $\varepsilon_0 > 0$ is chosen to ensure that the sets $\mathcal{D}_{i}^\varepsilon$ are mutually disjoint, \begin{equation}\label{intersection} 
\overline{\mathcal{D}_{i}^\varepsilon} \cap \overline{\mathcal{D}_{j}^\varepsilon} = \emptyset, \quad i \neq j. 
\end{equation} 
The initial vorticity is given by
\begin{equation}\label{intial-vort} 
\theta_{0}^\varepsilon(x) = \frac{1}{\varepsilon^2} \sum_{i=1}^N \frac{\gamma_i}{b_i^2} \chi_{\mathcal{D}_i^\varepsilon}(x). \end{equation} 
If  $|\mathcal{O}_i^\varepsilon| \to |\mathbb{D}|$ and $\varepsilon \to 0$, then the vorticity in \eqref{intial-vort} converges to a point vortex distribution 
\begin{equation*} \theta_{0}^0(x) = \pi \sum_{i=1}^N \gamma_i \delta_{w_i}(x), 
\end{equation*}
with its evolution governed by \eqref{ode-sys0}.
First, assume that $\theta_{0}^\varepsilon$ corresponds to $N$ co-rotating patches in the model \eqref{1-1}, about the origin, with an angular velocity $\Omega$, which rotates  counterclockwise.  Specifically, we seek a solution $\theta^\varepsilon(t)$ of \eqref{1-1} of the form

\[
\theta^\varepsilon(x,t) = \theta_{0}^\varepsilon\big(Q_{-\Omega t}(x)\big).
\]
where \(Q_\theta\) denotes the rotation matrix
\[
Q_\theta=
\begin{pmatrix}
\cos\theta & -\sin\theta\\
\sin\theta & \cos\theta
\end{pmatrix}.
\]
Substituting this ansatz into \eqref{1-1} yields
\[
\big(v^\varepsilon(\boldsymbol x) - \Omega \boldsymbol x^\perp\big) \cdot \nabla \theta_{0}^\varepsilon(\boldsymbol x) = 0,
\]
where $v^\varepsilon$ denotes the velocity field associated with $\theta_{0}^\varepsilon$. Since $\theta_0^\varepsilon$ is a sum of vortex patches, this is equivalent to
\begin{equation}\label{kin-rot}
 \big(v^\varepsilon(\boldsymbol x) - \Omega \boldsymbol x^\perp\big) \cdot \mathbf{n}(\boldsymbol x)=0, \quad \text{for every} \quad \boldsymbol x \in \partial \mathcal{D}_i^\varepsilon, \quad i = 1, \ldots, N,
\end{equation}
where $\mathbf{n}$ denotes the outward unit normal vector to $\partial\mathcal{D}_i^\varepsilon$.
 We consider vortex patches which are close to discs of radius \(\varepsilon b_i\), and parametrize their boundaries as
\begin{equation}\label{param}
z_{i}(x) = w_i + \varepsilon b_{i} R_{i}(x)\, (\cos x,\sin x), 
\end{equation}
where
\begin{equation}\label{solutions}
R_{i}(x) = 1 + \varepsilon |\varepsilon|^{\alpha} b_{i}^{1+\alpha} f_{i}(x), \quad \text{for } i = 1, \ldots, N,
\end{equation}
with \(x \in [0, 2\pi)\). Since $z_i'(x)=\varepsilon b_i\big(R_i'(x)\mathbf (\cos x,\sin x)+R_i(x)(-\sin x,\cos x)\big)$, the vector $-(z_i'(x))^\perp$ is an outward normal vector to $\partial\mathcal D_i^\varepsilon$ at $z_i(x)$, so that \eqref{kin-rot} is equivalent to
\[
\big(v^\varepsilon(z_i(x))-\Omega\, z_i(x)^\perp\big)\cdot (z_i'(x))^\perp=0 .
\]
 Therefore, the process of constructing \(N\) co-rotating patches for the gSQG equations with $1\leq\alpha<2$ is equivalent to finding a zero of the following equations
\begin{equation}
\begin{split}
&\Omega \left[ \varepsilon b_{i}R_{i}(x)R_i^{\prime }(x)+w_i\cdot\big(R_{i}^{\prime}(x)(\cos x,\sin x)+R_{i}(x)(-\sin x,\cos x)\big)\right]   \\
& {\resizebox{.98\hsize}{!}{$+\frac{\gamma_i C_\alpha}{\varepsilon^{1+\alpha}b_i^{1+\alpha} }\displaystyle\fint
\frac{\big(R_i(x)R_i(y)+R_i'(x)R_i'(y)\big)\sin(x-y)+\big(R_i(x)R_i'(y)-R_i'(x)R_i(y)\big)\cos(x-y)}{\left|\left(R_i(x)-R_i(y)\right)^2+4R_i(x)R_i(y)\sin^2\left(\frac{x-y}{2}\right)\right|^{\frac{\alpha}{2}}}\,dy$}} \\
&
{\resizebox{.98\hsize}{!}{$+\displaystyle\sum_{j=1, j\neq i}^N\frac{\gamma_j C_\alpha}{\varepsilon  b_j} 
\displaystyle\fint
\frac{\left((R_i(x)R_j(y)+R_i'(x)R_j'(y))%
\sin(x-y)+(R_i(x)R_j'(y)-R_i'(x)R_j(y))\cos(x-y)\right)dy}{|\varepsilon b_{i}R_{i}(x)(\cos(x),\sin(x))+w_i-\varepsilon b_{j}R_{j}(x)(\cos(x),\sin(x))-w_j|^\alpha}$}} \\
 &=0 .
\end{split}
\label{1-3}
\end{equation}

\subsection{Traveling vortex patches}

The second class of solutions we aim to construct consists of $N$ traveling patches corresponding to the gSQG equations. Similar to the co-rotating case, we consider $N$ bounded, simply connected domains $\mathcal{O}_i^\varepsilon$, $i = 1, \ldots, N$, which are small perturbations of the unit disc. These domains are contained within a disc of radius 2 and centered at the origin. Given $b_i \in (0,\infty)$, $w_i \in \mathbb{R}^2$, and $\varepsilon \in (0,\varepsilon_0)$, we define the domains 
\begin{equation*}
\mathcal{D}_{i}^\varepsilon := \varepsilon b_i \mathcal{O}_i^\varepsilon + w_i, 
\end{equation*}
where $\varepsilon_0 > 0$ is chosen to ensure that the sets $\mathcal{D}_{i}^\varepsilon$ are mutually disjoint, 
\begin{equation*} 
\overline{\mathcal{D}_{i}^\varepsilon} \cap \overline{\mathcal{D}_{j}^\varepsilon} = \emptyset, \quad i \neq j. 
\end{equation*} 
In this context, for any given real numbers \(\gamma_1\) and \(\gamma_2\), the initial vorticity data is expressed as follows
\begin{equation}
\theta_{0}^\varepsilon(x) = \frac{1}{\varepsilon^2} \sum_{i=1}^N \frac{\gamma_i}{b_i^2} \chi_{\mathcal{D}_i^\varepsilon}(x),
\label{initial}
\end{equation}
which consists of \(N\) simply connected patches, each characterized by distinct vorticity magnitudes \(\gamma_i\). It is important to note that the initial vorticity data \(\theta_{0}^\varepsilon\) is formulated for traveling patches. Therefore, in the case where $\theta_{\varepsilon}(x,t)$ is translating in the vertical direction with a uniform velocity $U$, it follows that
\begin{equation*}
\theta_{\varepsilon}(x,t) = \theta^\varepsilon_{0}(x - U\mathbf{e}_2 t),
\end{equation*}
where \(U\) represents a constant uniform speed. Based on equation \eqref{1-1}, we obtain
\begin{equation*}
(v^{\varepsilon}(x)- U \mathbf{e}_2)\cdot \nabla \theta^\varepsilon_{0}(x)=0, \quad \text{for every} \quad x \in \partial \mathcal{D}_i^\varepsilon, \quad i = 1, \ldots, N.
\end{equation*}%
In other words, the system outlined above can be expressed as
\begin{equation*}
(v^{\varepsilon}(x)-U\mathbf{e}_2)\cdot\mathbf{n}( x)=0,\quad \text{for every} \quad x \in \partial \mathcal{D}_i^\varepsilon, \quad i = 1, \ldots, N, 
\end{equation*}%

Once more, our investigation centers on identifying \(N\) vortex patches related to the gSQG equations \eqref{1-1} within the interval \([1,2)\). To put it differently, our goal is to find a zero of the following equations
\begin{equation}\label{t3}
\begin{split}
&-U \mathbf{e}_1\cdot\left(R_{i}(x)(-\sin (x),\cos(x))+R_{i}^{\prime
}(x)(\cos (x),\sin(x))\right)   \\
& {\resizebox{.98\hsize}{!}{$+\frac{\gamma_i
C_\alpha}{\varepsilon^{1+\alpha}b_i^{1+\alpha} }\displaystyle \fint
\frac{\left((R_i(x)R_i(y)+R_i'(x)R_i'(y))%
\sin(x-y)+(R_i(x)R_i'(y)-R_i'(x)R_i(y))\cos(x-y)\right)dy}{\left|
\left(R_i(x)-R_i(y)\right)^2+4R_i(x)R_i  (y)\sin^2\left(\frac{x-y}{2}\right)%
\right|^{\frac{\alpha}{2}}}$}} \\
& {\resizebox{.98\hsize}{!}{$+\displaystyle\sum_{j=1, j\neq i}^N\frac{\gamma_j C_\alpha}{\varepsilon  b_j} 
\displaystyle \fint
\frac{\left((R_i(x)R_j(y)+R_i'(x)R_j'(y))%
\sin(x-y)+(R_i(x)R_j'(y)-R_i'(x)R_j(y))\cos(x-y)\right)dy}{|\varepsilon b_{i}R_{i}(x)(\cos(x),\sin(x))+w_i-\varepsilon b_{j}R_{j}(x)(\cos(x),\sin(x))-w_j|^\alpha}$}} \\
& \qquad =0,
\end{split}
\end{equation}%
To simplify our approach, we define a single condition that encompasses both the co-rotating vortex patches $\eqref{1-3}$ and the traveling vortex patches $\eqref{t3}$.
\begin{equation}\label{eq:funct}
\begin{split}
&U\left(R_{i}(x)\sin(x)-R_{i}^{\prime
}(x)\cos(x)\right) \\
&\qquad +\Omega \left( \varepsilon b_{i}R_{i}(x)R_i^{\prime }(x)+w_i\cdot(R_{i}^{\prime
}(x)(\cos (x),\sin(x))+w_i\cdot R_{i}(x)(-\sin (x),\cos(x))\right)   \\
& 
{\resizebox{.98\hsize}{!}{$+\frac{\gamma_i
C_\alpha}{\varepsilon^{1+\alpha}b_i^{1+\alpha} }\displaystyle \fint
\frac{\left((R_i(x)R_i(y)+R_i'(x)R_i'(y))%
\sin(x-y)+(R_i(x)R_i'(y)-R_i'(x)R_i(y))\cos(x-y)\right)dy}{\left|
\left(R_i(x)-R_i(y)\right)^2+4R_i(x)R_i  (y)\sin^2\left(\frac{x-y}{2}\right)%
\right|^{\frac{\alpha}{2}}}$}} \\
&
{\resizebox{.98\hsize}{!}{$+\displaystyle\sum_{j=1, j\neq i}^N\frac{\gamma_j C_\alpha}{\varepsilon  b_j} 
\displaystyle \fint
\frac{\left((R_i(x)R_j(y)+R_i'(x)R_j'(y))%
\sin(x-y)+(R_i(x)R_j'(y)-R_i'(x)R_j(y))\cos(x-y)\right)dy}{|\varepsilon b_{i}R_{i}(x)(\cos(x),\sin(x))+w_i-\varepsilon b_{j}R_{j}(x)(\cos(x),\sin(x))-w_j|^\alpha}$}} \\
& \qquad =0,
\end{split}
\end{equation}
with the convention that either \(U=0\) (rotating case) or \(\Omega=0\)
(translating case).


\subsection{Notation and Functional spaces}\label{sec:spaces}

We need to establish some notations that will be utilized consistently throughout this paper. Let \(C\) represent any positive constant that may change from one instance to another. We denote the unit disk by \(\mathbb{D}\), while its boundary, the unit circle, is represented by \(\partial(\mathbb{D})=\mathbb{T}\). For simplicity, let us define
$$
\fint_{0}^{2\pi} f(\tau)\, d\tau \equiv \frac{1}{2\pi}\int_{0}^{2\pi} f(\tau)\, d\tau.
$$
denote the mean value of the integral over the unit circle.
 The primary strategy for demonstrating the existence of \(N\)  vortex patches for the gSQG equations \eqref{1-1} with $\alpha\in[1,2)$, involves applying the implicit function theorem to the system of equations represented by \eqref{eq:funct}. To accomplish this, it is crucial to verify that the functional \(\mathcal{F}^\alpha\) meets the necessary regularity conditions. The relevant functional spaces utilized in this analysis are outlined below
\begin{equation*}
	X^k=\left\{ f_i\in H^k, \ f_i(x)= \sum\limits_{n=2}^{\infty}\left(a^i_n\cos(n x)+d^i_n \sin (n x)\right)\right\},
\end{equation*}

\begin{equation*}
	Y^{k}=\left\{ g_i\in H^{k}, \ g_i(x)= \sum\limits_{n=1}^{\infty}\left(b^i_n\sin(n x)+c^i_n \cos (n x)\right),
\right\},
\end{equation*}
and
\begin{equation*}
 Y_{0}^{k}=\left\{ g_i\in Y^{k},\
b^i_1=c^i_1=0,\right\},
\end{equation*}%

For the SQG equation $(\alpha=1)$, we also need the following space
\begin{equation}\label{Valpha0}
\mathcal{X}^{k+\log} :=\underbrace{X^{k+\log}\times\cdots \times X^{k+\log}}_{\text{$N$ times}}, \quad
\mathcal{Y}^k:=\underbrace{Y^{k}\times\cdots \times Y^{k}}_{\text{$N$ times}}, \quad\textnormal{and}\quad  
\mathcal{Y}^k_0:=\underbrace{Y_{0}^{k}\times\cdots \times Y_{0}^{k}}_{\text{$N$ times}}
\end{equation}
with
\begin{equation*}
	\begin{split}
 {\resizebox{.98\hsize}{!}{$ X^{k+\log}=\left\{ f_i\in H^k, \ f_i(x)= \sum\limits_{n=2}^{\infty}\left(a^i_n\cos(n x)+d^i_n \sin (n x)\right), \ \left\|\displaystyle\int_0^{2\pi}\frac{\partial^kg(x-y)-\partial^kg(x)}{|\sin(\frac{y}{2})|}dy\right\|_{L^2}<\infty \right\},$}}
	\end{split}
\end{equation*}
For the generalized SQG equations  ($1<\alpha<2$), we will need the following fractional spaces
\begin{equation}\label{alpha}
 {\resizebox{.98\hsize}{!}{$\mathcal{X}^{k+\alpha-1} :=\underbrace{X^{k+\alpha-1}\times\cdots \times X^{k+\alpha-1}}_{\text{$N$ times}}, \quad
\mathcal{Y}^k:=\underbrace{Y^{k}\times\cdots \times Y^{k}}_{\text{$N$ times}}, \quad\textnormal{and}\quad  
\mathcal{Y}^k_0:=\underbrace{Y_{0}^{k}\times\cdots \times Y_{0}^{k}}_{\text{$N$ times}}$}} ,
\end{equation}
where $X^{k+\alpha-1}$ is defined by
\begin{equation*}
	\begin{split}
 {\resizebox{.98\hsize}{!}{$		X^{k+\alpha-1}=\left\{ f_i\in H^k, \ f_i(x)= \sum\limits_{n=2}^{\infty}\left(a^i_n\cos(n x)+d^i_n \sin (n x)\right), \ \left\|\displaystyle\int_0^{2\pi}\frac{\partial^kg(x-y)-\partial^kg(x)}{|\sin(\frac{y}{2})|^\alpha}dy\right\|_{L^2}<\infty \right\},$}}
	\end{split}
\end{equation*}

The norms for the spaces \(X^k\) and \(Y^k\) are defined using the \(H^k\)-norm. In the cases of \(X^{k+\log}\) and \(X^{k+\alpha-1}\), their norms are formulated as the sum of the \(H^k\)-norm and the integral as specified in their definitions, respectively. It is noteworthy that for all \(\nu > 0\), the following embeddings are valid: \(X^{k+\nu} \hookrightarrow  X^{k+\log} \hookrightarrow  X^k\). In this paper, we consistently assume that \(k \geq 3\).

We denote by $\mathcal{B}_X$ the unit ball
\begin{equation*}
    \mathcal{B}_X:= \begin{cases}
       \big\{f\in \mathcal{X}^{k+\log}: \|f\|_{\mathcal{X}^{k+\log}}< 1\big\},\qquad &\alpha=1, \\
       \big\{f\in \mathcal{X}^{k+\alpha-1}: \|f\|_{\mathcal{X}^{k+\alpha-1}}< 1\big\} ,\qquad &\alpha\in(1,2) ,
    \end{cases}
\end{equation*}

\section{Extension, Regularity and linearization of the functional}\label{section3}

We seek domains $\mathcal{O}_i^\varepsilon$ that are perturbations of the unit disc, with perturbation amplitudes on the order of $\varepsilon b_i$. Specifically, we consider the following expansion
\begin{align}\label{conf0}
R_i(x) = 1 + \varepsilon\abs{\varepsilon}^{\alpha} b_i^{1+\alpha} f_i(x). 
\end{align}
The coefficient $|\varepsilon|^{\alpha}$ in \eqref{conf0} is a result of the singularity in the gSQG kernel and thus,  to desingularize this system in terms of $\varepsilon$,  we adopt the approach from \cite{Hmidi-Mateu} and transform the equation \eqref{eq:funct} into
\begin{equation}\label{funct}
    \mathcal{F}^\alpha_i(\varepsilon, f, \lambda,x)= \mathcal{F}^\alpha_{i, 1}\left(\varepsilon, f_i, x\right) +\mathcal{F}^\alpha_{i, 2}(\varepsilon, f, \lambda,x)+\mathcal{F}^\alpha_{i, 3}(\varepsilon, f, \lambda,x), \quad \forall i=1, \ldots, N,
\end{equation}
and the terms $\mathcal{F}_{i, j}\left(\varepsilon, f, \lambda,x \right)$,  for $j=1,2,3$ and $i=1,\cdots,N$ are given by

\begin{equation}\label{f1}
   \begin{aligned}
 \mathcal{F}^\alpha_{i, 1}(\varepsilon,f_i,\lambda,x) &=  \Omega \Big( \varepsilon b_{i}R_i(x)R_i^{\prime }(x)+w_i\cdot\big(R_{i}^{\prime}(x)\mathbf e(x)+R_{i}(x)\mathbf e^\perp(x)\big)\Big)\\
 &\quad + U \big(R_{i}^{\prime}(x)\cos x-R_i(x)\sin x\big)  ,
   \end{aligned}
\end{equation}
\begin{equation}\label{f2}
    \begin{aligned}
        &  \mathcal{F}^\alpha_{i, 2}\left(\varepsilon, f_i, x\right)=\frac{C_\alpha \gamma_i}{2 \pi \varepsilon|\varepsilon|^{\alpha} b_i^{1+\alpha}}  \times \\
        & \int_0^{2 \pi} \frac{\left(\left(R_i( x) R_i( y)+R_i^{\prime}( x) R_i^{\prime}( y)\right)\sin( x- y)+\left(R_i( x) R_i^{\prime}( y)-R_i^{\prime}( x) R_i( y)\right) \cos( x- y)\right) d  y}{\left( A( x,  y)+\varepsilon|\varepsilon|^\alpha b_i^{1+\alpha}  B\left(  f_i,  x, y\right)\right)^{\frac{\alpha}{2}}},
    \end{aligned}
\end{equation}
and
\begin{equation}\label{f3}
    \begin{aligned}
        &  \mathcal{F}^\alpha_{i, 3}(\varepsilon, f, \lambda,x)=\frac{C_\alpha }{2 \pi \varepsilon} \sum_{j \neq i} \\
        & {\resizebox{.96\hsize}{!}{$\frac{\gamma_j}{b_j}\displaystyle\int_0^{2 \pi} \frac{\left(\left(R_i( x) R_j( y)+R_i^{\prime}( x) R_j^{\prime}( y)\right)\sin( x- y)+\left(R_i( x) R_j^{\prime}( y)-R_i^{\prime}( x) R_j( y)\right) \cos( x- y)\right) d  y}{\left(A_{i, j}+\varepsilon B_{i, j}( x,  y)\right)^{\frac{\alpha}{2}}},$}}
    \end{aligned}
\end{equation}
 where, for the sake of convenience, we adopted the following notations
\begin{equation}\label{comput}
	\begin{split}
		&A( x, y):=4\sin^2\left(\frac{ x-y}{2}\right),\,\,\,A_{ij}=|w_i-w_j|^2,\\
				&
    {\resizebox{.95\hsize}{!}{$B(f, x,  y):=4(f( x)+f( y))\sin^2\left(\frac{ x- y}{2}\right)+\varepsilon|\varepsilon|^{\alpha}b^{1+\alpha}_i\left[(f( x)-f( y))^2+4 f( x)f( y)\sin^2\left(\frac{ x- y}{2}\right)\right],$}}\\
		&
  {\resizebox{.95\hsize}{!}{$B_{ij}( x,  y,\lambda):=2(w_i-w_j)\cdot(b_i(\cos x, \sin x)-b_j(\cos y, \sin y))+2\varepsilon|\varepsilon|^{\alpha}(w_i-w_j)\cdot\left(b^{2+\alpha}_if_i( x)(\cos x, \sin x)\right.$}}\\
		&
    {\resizebox{.95\hsize}{!}{$\quad\,\left.-b^{2+\alpha}_jf_j( y)(\cos y, \sin y)\right)+\varepsilon\left|b_i(1+\varepsilon|\varepsilon|^{\alpha}b_i^{1+\alpha} f_i( x))(\cos x, \sin x)-b_j(1+\varepsilon|\varepsilon|^{\alpha} b_j^{1+\alpha}  f_j( y))(\cos y, \sin y)\right|^2.$}}
	\end{split}
\end{equation}

We then define the nonlinear operator
\begin{equation*}
\mathcal{F}^\alpha(\varepsilon,f,\lambda):=\big(\mathcal{F}^\alpha_1(\varepsilon,f,\lambda),\ldots,\mathcal{F}^\alpha_N(\varepsilon,f,\lambda)\big).
\end{equation*}
To apply the implicit function theorem at \(\varepsilon = 0\), it is essential to extend the functional $\mathcal{F}^\alpha$ to a domain that includes \(\varepsilon \leq 0\), while ensuring that it maintains \(C^1\) regularity. The first step in this process involves verifying the continuity of \(\mathcal{F}^\alpha\).

\begin{proposition}\label{p3-1}
Let $\alpha\in(1,2)$ and let $\lambda^*$ solve \eqref{alg-sysP}. There exist $\varepsilon_0>0$ and a small neighborhood $\Lambda$ of $\lambda^*$ such that $\mathcal{F}^\alpha$ can be extended to a continuous mapping $\left(-\varepsilon_0, \varepsilon_0\right) \times \mathcal{B}_X\times \Lambda\to \mathcal Y^{k-1}$. 
\end{proposition}

\begin{proof}
We choose $\Lambda$ bounded and so small that $w_i\neq w_j$ for $i\neq j$ and $\lambda\in\overline\Lambda$, and $\varepsilon_0>0$ small.   By using \eqref{f1}, we can rewrite $\mathcal{F}^\alpha_{i, 1}:\left(-\varepsilon_0, \varepsilon_0\right) \times \mathcal{B}_X\times \Lambda \rightarrow \mathcal Y^{k-1}$ in terms of $f_i$ by using  the following expression
\begin{equation*}
R_{i}(x) = 1 + \varepsilon |\varepsilon|^{\alpha} b_{i}^{1+\alpha} f_{i}(x), \quad \text{for } i = 1, \ldots, N,
\end{equation*}
Then, one gets
\begin{equation}\label{eq:f1}
   \begin{aligned}
 \mathcal{F}^\alpha_{i, 1} &= \Omega \Big( \varepsilon^2 |\varepsilon|^{\alpha} b_{i}^{2+\alpha}\big(1 + \varepsilon |\varepsilon|^{\alpha} b_{i}^{1+\alpha} f_{i}(x)\big) f'_{i}(x)+\varepsilon |\varepsilon|^{\alpha} b_{i}^{1+\alpha} f'_{i}(x)\,w_i\cdot (\cos x,\sin x)\\
&\qquad\quad+\big(1 + \varepsilon |\varepsilon|^{\alpha} b_{i}^{1+\alpha} f_{i}(x)\big)\,w_i\cdot (-\sin x,\cos x)\Big)\\
&\quad +U \Big(\varepsilon |\varepsilon|^{\alpha} b_{i}^{1+\alpha} f'_{i}(x)\cos x-\big(1 + \varepsilon |\varepsilon|^{\alpha} b_{i}^{1+\alpha} f_{i}(x)\big)\sin x\Big).
   \end{aligned}
\end{equation}
Splitting $\mathcal{F}^\alpha_{i, 1}$ into the components that are independent of $\varepsilon$ and those that are $O(\varepsilon|\varepsilon|^{\alpha})$, and using $-\sin x=\mathbf e_1\cdot(-\sin x,\cos x)$, we obtain
\begin{equation}
\mathcal{F}^\alpha_{i, 1}=\big(\Omega w_i+U\mathbf e_1\big)\cdot (-\sin x,\cos x)+\varepsilon|\varepsilon|^\alpha\mathcal{R}_{i,1}(\varepsilon,f_i,\lambda) ,
\label{2-1}
\end{equation}
where
\[
\begin{split}
\mathcal{R}_{i,1}=\;&\Omega\Big(\varepsilon b_i^{2+\alpha}\big(1 + \varepsilon |\varepsilon|^{\alpha} b_{i}^{1+\alpha} f_{i}(x)\big) f'_{i}(x)+b_i^{1+\alpha}f_i'(x)\,w_i\cdot(\cos x,\sin x)\\
&\qquad+b_i^{1+\alpha}f_i(x)\,w_i\cdot(-\sin x,\cos x)\Big)\\
&\qquad \qquad +Ub_i^{1+\alpha}\big(f_i'(x)\cos x-f_i(x)\sin x\big).
\end{split}
\]
where  $\mathcal R_{i,1}$ is continuous. Consequently, we can infer that \(\mathcal{F}^\alpha_{i, 1}\) is also continuous. It is important to observe that we can decompose the functional $\mathcal{F}^\alpha_{i, 2}$ as follows
                   \begin{equation}\label{2-2}
	\begin{split}
    		\mathcal{F}^\alpha_{i, 2}=&{\resizebox{.94\hsize}{!}{$\frac{C_\alpha \gamma_i}{\varepsilon|\varepsilon|^{\alpha}b_i^{1+\alpha}}\displaystyle\fint\frac{(  1+\varepsilon|\varepsilon|^\alpha b_i^{1+\alpha} f_i( y))\sin( x- y)d y}{\left(  A( x,  y)+\varepsilon|\varepsilon|^\alpha b_i^{1+\alpha} B\left( f_i,  x,  y\right)\right)^{\frac{\alpha}{2}}}
  +C_\alpha \gamma_i\displaystyle\fint \frac{(f'_i( y)-f'_i( x))\cos( x- y)d y}{\left(  A( x,  y)+\varepsilon|\varepsilon|^\alpha b_i^{1+\alpha} B\left(f_i,  x,  y\right)\right)^{\frac{\alpha}{2}}}$}}\\		
		&
  +\frac{C_\alpha \gamma_i\varepsilon |\varepsilon|^\alpha b_i^{1+\alpha} f'_i( x)}{  1+\varepsilon|\varepsilon|^\alpha b_i^{1+\alpha} f_i( x)}\int\!\!\!\!\!\!\!\!\!\; {}-{} \frac{(f_i( x)-f_i( y))\cos( x- y)d y}{\left(  A( x,  y)+\varepsilon|\varepsilon|^\alpha b_i^{1+\alpha} B\left( f_i,  x,  y\right)\right)^{\frac{\alpha}{2}}}\\	
		&
  \qquad+\frac{C_\alpha \gamma_i\varepsilon|\varepsilon|^\alpha b_i^{1+\alpha} }{  1+\varepsilon|\varepsilon|^\alpha b_i^{1+\alpha} f_i( x)}\int\!\!\!\!\!\!\!\!\!\; {}-{} \frac{f'_i( x)f'_i( y)\sin( x- y)d y}{\left(  A( x,  y)+\varepsilon|\varepsilon|^\alpha b_i^{1+\alpha} B\left( f_i,  x,  y\right)\right)^{\frac{\alpha}{2}}}\\	
		=&
  \mathcal{F}^\alpha_{i21}+\mathcal{F}^\alpha_{i22}+\mathcal{F}^\alpha_{i23}+\mathcal{F}^\alpha_{i24} .
	\end{split}
\end{equation} 

Specifically, the most singular component in $\mathcal{F}^\alpha_{i, 2}$  is given by $\mathcal{F}^\alpha_{i21}$. To deal with this term we proceed as follows
\begin{equation*}
    \mathcal{F}^\alpha_{i21}:=\frac{C_\alpha \gamma_i}{\varepsilon|\varepsilon|^{\alpha}b_i^{1+\alpha}}\int\!\!\!\!\!\!\!\!\!\; {}-{} \frac{(  1+\varepsilon|\varepsilon|^{\alpha} b_i^{1+\alpha} f_i( y))\sin( x- y)}{\left(  A( x,  y)+\varepsilon|\varepsilon|^{\alpha} b_i^{1+\alpha}B\left( f_i,  x,  y\right)\right)^{\frac{\alpha}{2}}} d y,
\end{equation*}
the potential singularity arising from \(\varepsilon=0\) may only manifest when we compute the zeroth-order derivative of \(\mathcal{F}^\alpha_{i21}\). Throughout the proof, we will often use the following Taylor's formula
\begin{equation}\label{taylor}
	\frac{1}{(A+B)^s}=\frac{1}{A^s}-s\int_0^1\frac{B}{(A+tB)^{1+s}}dt.
\end{equation}
Then, we decompose the kernel into two parts
\begin{equation}\label{2-6}
	\begin{split}
		\mathcal{F}^\alpha_{i21}&{\resizebox{.95\hsize}{!}{$	=\frac{C_{\alpha}\gamma_i}{\varepsilon|\varepsilon|^{\alpha}b_i^{1+\alpha}}\displaystyle\fint\frac{ \sin{( x- y)}d y}{\left(  A( x, y)+\varepsilon|\varepsilon|^{\alpha}b_i^{1+\alpha}B\left( f_{i}, x, y\right)\right)^{\frac{\alpha}{2}}}+C_\alpha\gamma_i\displaystyle\fint \frac{f_{i}( y)\sin{( x- y)}d y}{\left(   A( x, y)+\varepsilon|\varepsilon|^{\alpha}b_i^{1+\alpha} B\left(   f_{i},  x,  y\right)\right)^{\frac{\alpha}{2}}}$}}\\
		&=\frac{C_\alpha \gamma_i}{\varepsilon|\varepsilon|^{\alpha}b_i^{1+\alpha}}\int\!\!\!\!\!\!\!\!\!\; {}-{} \frac{\sin( x- y)d y}{ A(x,y)^{\frac{\alpha}{2}}}-\frac{\alpha C_\alpha \gamma_i}{2}\int\!\!\!\!\!\!\!\!\!\; {}-{} \int_0^1 \frac{B\sin( x- y)dt d y}{\left(  A( x,  y)+t\varepsilon|\varepsilon|^\alpha b_i^{1+\alpha}B\left( f_i,  x,  y\right)\right)^{\frac{\alpha+2}{2}}}\\
  &\qquad \qquad+C_\alpha\gamma_i\int\!\!\!\!\!\!\!\!\!\;{}-{} \frac{f_{i}( y)\sin{( x- y)}d y}{\left(   A( x, y)+\varepsilon|\varepsilon|^{\alpha} b_i^{1+\alpha}B\left(   f_{i},  x,  y\right)\right)^{\frac{\alpha}{2}}}\\
		&=-\frac{\alpha C_\alpha \gamma_i}{2}\int\!\!\!\!\!\!\!\!\!\; {}-{} \int_0^1 \frac{B\sin( x- y)dt d y}{(A(x,y))^{\frac{\alpha}{2}+1}}\\
  &
\ \ \ \ + \frac{C_\alpha\gamma_i\varepsilon|\varepsilon|^\alpha \alpha(\alpha+2)}{4}\int\!\!\!\!\!\!\!\!\!\; {}-{}\int_0^1 \int_0^1 \frac{t B^{2}\sin( x- y)d\tau dt d y}{\left(  A( x,  y)+t\tau\varepsilon|\varepsilon|^\alpha b_i^{1+\alpha}B\left( f_i,  x,  y\right)\right)^{\frac{\alpha+4}{2}}}\\
		&\ \ \ \ {\resizebox{.92\hsize}{!}{$	+C_\alpha\gamma_i\displaystyle\fint\frac{f_i( y)\sin( x- y)d y}{A(x,y)^{\frac{\alpha}{2}}}-\frac{C_\alpha \alpha \gamma_i\varepsilon|\varepsilon|^\alpha}{2}\displaystyle
  \fint\int_0^1 \frac{B f_i( y)\sin( x- y)dt d y}{\left(  A( x,  y)+\varepsilon|\varepsilon|^\alpha b_i^{1+\alpha} B\left( f_i,  x,  y\right)\right)^{\frac{\alpha+2}{2}}}$}}\\
        &= -\frac{\alpha C_\alpha \gamma_i}{2}\int\!\!\!\!\!\!\!\!\!\; {}-{}  \frac{f_i( y)A(x,y)\sin( x- y)d y}{A(x,y)^{\frac{\alpha+2}{2}}}+C_\alpha \gamma_i\int\!\!\!\!\!\!\!\!\!\; {}-{} \frac{f_i( y)\sin( x- y)d y}{A(x,y)^{\frac{\alpha}{2}}}+\varepsilon|\varepsilon|^\alpha\mathcal{R}_{i11}(\varepsilon,f_i)\\
		&=C_\alpha \gamma_i\left(1-\frac{\alpha}{2}\right)\int\!\!\!\!\!\!\!\!\!\; {}-{} \frac{f_i( y)\sin( x- y)d y}{\left|4\sin(\frac{ x- y}{2})\right|^\alpha}+\varepsilon|\varepsilon|^\alpha\mathcal{R}_{i21}(\varepsilon,f_i),
	\end{split}
\end{equation}
where $\mathcal{R}_{i21}$ is not singular with respect to $\varepsilon$.

Next, we proceed to differentiate $\mathcal{F}^\alpha_i$ with respect to $ x$ up to $\partial^{k-1}$ times. Our initial focus will be on the most singular term, namely, $ \partial ^{k-1}\mathcal{F}^\alpha_{i22}$
\begin{equation*}
\begin{split}
& {\resizebox{.96\hsize}{!}{$\partial ^{k-1}\mathcal{F}^\alpha_{i22}=C_\alpha\gamma _{i}\displaystyle\fint\frac{%
(\partial ^{k}f_{i}(y)-\partial ^{k}f_{i}(x))\cos (x-y)dy}{\left( A( x, y)+\varepsilon|\varepsilon|^{\alpha}b_i^{1+\alpha} B\left( f_{i}, x, y\right)\right)^{\frac{\alpha}{2}}}  -C_\alpha\gamma _{i}\varepsilon |\varepsilon |b_{i}^{2}\displaystyle\fint\frac{\cos (x-y)}{\left( A( x, y)+\varepsilon|\varepsilon|^{\alpha}b_i^{1+\alpha} B\left( f_{i}, x, y\right)\right)^{\frac{\alpha+2}{2}}} $}}\\
& \ \ \ \ {\resizebox{.96\hsize}{!}{$\times \left(\varepsilon |\varepsilon|^\alpha
b_i^{1+\alpha}(f_i(x)-f_i(y))(f'_i(x)-f'_i(y))+2((1+\varepsilon|\varepsilon|^\alpha
b_i^{1+\alpha}f_i(x))f'_i(y)+(1+\varepsilon|\varepsilon|^\alpha
b_i^{1+\alpha}f_i(y))f'_i(x))\sin^2(\frac{x-y}{2})\right)$}} \\
& \ \ \ \ \times (\partial ^{k-1}f_{i}(y)-\partial ^{k-1}f_{i}(x))dy+l.o.t,
\end{split}%
\end{equation*}%
where \(l.o.t\) refers to lower order terms. We now utilize the fact that \(\| \partial^m f_i \|_{L^\infty} \leq C \| f_i \|_{X^{k+\alpha-1}} < \infty\) for \(m=0, 1, 2\), given that \(f_i(x) \in X^{k+\alpha-1}\) for \(k \geq 3\). By applying the Hölder inequality in conjunction with the mean value theorem, we obtain
\begin{equation*}
\begin{split}
\left\Vert \partial ^{k-1}\mathcal{F}^\alpha_{i22}\right\Vert _{L^{2}}& \leq C\left\Vert \int
\!\!\!\!\!\!\!\!\!\;{}-{}\frac{\partial ^{k}f_{i}(x)-\partial ^{k}f_{i}(y)}{%
|4\sin (\frac{x-y}{2})|^\alpha}dy\right\Vert _{L^{2}}+C\left\Vert \int
\!\!\!\!\!\!\!\!\!\;{}-{}\frac{\partial ^{k-1}f_{i}(x)-\partial
^{k-1}f_{i}(y)}{|4\sin (\frac{x-y}{2})|^\alpha}dy\right\Vert _{L^{2}} \\
& \leq C\Vert f_{i}\Vert _{X^{k+\alpha-1 }}+C\Vert f_{i}\Vert _{X^{k+\alpha-2}}<\infty .
\end{split}%
\end{equation*}%
Now, observe that \(\mathcal{F}^\alpha_{i23}\) is less singular than \(\mathcal{F}^\alpha_{i22}\). Therefore, it is straightforward to establish an upper bound for \(\left\Vert \partial^{k-1}\mathcal{F}^\alpha_{i23}\right\Vert_{L^{2}}\). Next, we turn our attention to the final term \(\mathcal{F}^\alpha_{i24}\) and differentiate it \(k-1\) times with respect to \(x\) to derive
\begin{equation*}
\begin{split}
&
\partial^{k-1}\mathcal{F}^\alpha_{i24}=\frac{-C_\alpha\gamma_i\varepsilon^{2+2\alpha}
b_i^{2+2\alpha}\,\partial^{k-1}f_i(x)}{(1+\varepsilon|\varepsilon|^{\alpha}b_i^{1+\alpha}
f_i(x))^2}\displaystyle\fint 
\frac{f'_i(x)f'_i(y)\sin(x-y)dy}{\left( A( x, y)+\varepsilon|\varepsilon|^{\alpha}b_i^{1+\alpha} B\left( f_{i}, x, y\right)\right)^{\frac{\alpha}{2}}} \\
&
\quad +\frac{C_\alpha\gamma _{i}\varepsilon |\varepsilon |^\alpha b_{i}^{1+\alpha}}{1+\varepsilon
|\varepsilon |^\alpha b_{i}^{1+\alpha}f_{i}(x)}\displaystyle\fint\frac{(f_{i}^{\prime }(x)\partial ^{k}f_{i}(y)+\partial
^{k}f_{i}(x)f_{i}^{\prime }(y))\sin
(x-y)dy}{\left( A( x, y)+\varepsilon|\varepsilon|^{\alpha}b_i^{1+\alpha} B\left( f_{i}, x, y\right)\right)^{\frac{\alpha}{2}}} \\
& \quad -\frac{2C_\alpha\gamma _{i}\varepsilon^{2+2\alpha}
b_i^{2+2\alpha}}{1+\varepsilon
|\varepsilon |^\alpha b_{i}^{1+\alpha}f_{i}(x)}\displaystyle\fint\frac{\sin
(x-y)}{\left( A( x, y)+\varepsilon|\varepsilon|^{\alpha}b_i^{1+\alpha} B\left( f_{i}, x, y\right)\right)^{\frac{\alpha+2}{2}}} \\
&
\ \ \ \ {\resizebox{.96\hsize}{!}{$\times \left(\varepsilon |\varepsilon|^\alpha
b_i^{1+\alpha}(f_i(x)-f_i(y))(f'_i(x)-f'_i(y))+2((1+\varepsilon|\varepsilon|^\alpha
b_i^{1+\alpha}f_i(x))f'_i(y)+(1+\varepsilon|\varepsilon|^\alpha
b_i^{1+\alpha}f_i(y))f'_i(x))\sin^2(\frac{x-y}{2})\right)$}} \\
&
\qquad \qquad \times (f_{i}^{\prime }(x)\partial ^{k-1}f_{i}(y)+\partial
^{k-1}f_{i}(x)f_{i}^{\prime }(y))+l.o.t.
\end{split}%
\end{equation*}%
Using the definition of the space \(X^{k+\alpha-1}\), we can establish the following estimate
\begin{equation*}
\begin{split}
&{\resizebox{.96\hsize}{!}{$\left\Vert \partial ^{k-1}\mathcal{F}^\alpha_{i24}\right\Vert _{L^{2}} \leq C\varepsilon
|\varepsilon |^\alpha\left( \varepsilon |\varepsilon |^\alpha\Vert f_{i}^{\prime }\Vert
_{L^{\infty }}^{2}\Vert \partial ^{k-1}f_{i}\Vert _{L^{2}}+\Vert
f_{i}^{\prime }\Vert _{L^{\infty }}\Vert \partial ^{k}f_{i}\Vert
_{L^{2}}+\varepsilon |\varepsilon |^\alpha\Vert f_{i}\Vert _{L^{\infty }}\Vert
f_{i}^{\prime }\Vert _{L^{\infty }}^{2}\Vert \partial ^{k-1}f_{i}\Vert
_{L^{2}}\right)$}} \\
&\qquad \qquad\quad \leq C\varepsilon |\varepsilon |^\alpha\Vert f_{i}\Vert _{X^{k+\alpha-1 }}<\infty .
\end{split}%
\end{equation*}%
Based on the preceding computations, we can conclude that the nonlinear functional \(\mathcal{F}^\alpha_{i2}\) is an element of $\mathcal Y^{k-1}$.

Next, we prove the continuity of \(\mathcal{F}^\alpha_{i2}\), particularly focusing on the most singular term, \(\mathcal{F}^\alpha_{i22}\). For this purpose, we introduce the following notation. For any general function \(g_{i}(x)\), we define
\begin{equation*}
\Delta f_{i}=f_{i}(x)-f_{i}(y),\ \mbox{where}\ f_{i}=f_{i}(x),\ \mbox{and}\quad
\tilde{f}_{i}=f_{i}(y),\,\mbox{for}\,i=1,2 ,
\end{equation*}%
and
\begin{equation*}
D_{\alpha }(f_{i})=b_{i}^{2+2\alpha }\varepsilon ^{2+2\alpha }\Delta
f_{i}^{2}+4(1+\varepsilon |\varepsilon |^{\alpha }b_{i}^{1+\alpha
}f_{i})(1+\varepsilon |\varepsilon |^{\alpha }b_{i}^{1+\alpha }\tilde{f}%
_{i})\sin ^{2}\left( \frac{x-y}{2}\right) .
\end{equation*}%
Consequently, for \(f_{i1}\) and \(f_{i2}\) belonging to \(X^{k+\alpha-1}\) with \(i=1,2\), we have
\begin{equation*}
\begin{split}
\mathcal{F}^\alpha_{i22}(\varepsilon ,f_{i1})& -\mathcal{F}^\alpha_{i22}(\varepsilon ,f_{i2})=C_\alpha\gamma _{i}\int
\!\!\!\!\!\!\!\!\!\;{}-{}\frac{(\Delta f_{i1}^{\prime }-\Delta
f_{i2}^{\prime })\cos (x-y)dy}{D_{1}(f_{i1})^{\frac{1}{2}}} \\
& +\left(C_\alpha \gamma _{i}\displaystyle\fint\frac{\Delta
f_{i2}^{\prime }\cos (x-y)dy}{D_{1}(f_{i1})^{\frac{\alpha}{2}}}-C_\alpha\gamma _{i}\int
\!\!\!\!\!\!\!\!\!\;{}-{}\frac{\Delta f_{i2}^{\prime }\cos (x-y)dy}{%
D_{1}(f_{i2})^{\frac{\alpha}{2}}}\right) \\
& =I_{1}+I_{2}.
\end{split}%
\end{equation*}%
It is clear that \(I_{1}\) can be bounded as follows
\[
\| I_{1} \|_{Y^{k-1}} \leq C \| f_{i1} - f_{i2} \|_{X^{k+\alpha-1}}.
\]
To estimate \(I_{2}\), we can apply the mean value theorem in the following manner
\begin{equation}
\begin{split}
& \frac{1}{D_{\alpha }(f_{i1})^{\frac{\alpha }{2}}}-\frac{1}{D_{\alpha
}(f_{i2})^{\frac{\alpha }{2}}}=\frac{\alpha }{2}\frac{D_{\alpha
}(f_{i2})-D_{\alpha }(f_{i1})}{D_{\alpha }(\delta _{x,y}f_{i1}+(1-\delta
_{x,y})f_{i2})^{1-\frac{\alpha }{2}}D_{\alpha }(f_{i1})^{\frac{\alpha }{2}%
}D_{\alpha }(f_{i2})^{\frac{\alpha }{2}}} \\
& {\resizebox{.98\hsize}{!}{$=\frac{\alpha}{2}\frac{b_i^{2+2\alpha}|%
\varepsilon|^{2+2\alpha}(\Delta f_{i2}^2-\Delta
f_{i1}^2)+4\varepsilon|\varepsilon|^\alpha b^{1+\alpha}_i
((f_{i2}-f_{i1})(1+\varepsilon|\varepsilon|^\alpha b^{1+\alpha}_i \tilde
f_{i2})+(\tilde f_{i2}-\tilde f_{i1})(1+\varepsilon|\varepsilon|^\alpha
b^{1+\alpha}_i
f_{i1}))\sin^2(\frac{x-y}{2})}{D_\alpha(\delta_{x,y}f_{i1}+(1-%
\delta_{x,y})f_{i2})^{1-\frac{\alpha}{2}}D_\alpha(f_{i1})^\frac{\alpha}{2}D_%
\alpha(f_{i2})^\frac{\alpha}{2}}$}},
\end{split}
\label{2-7}
\end{equation}%
for some \(\delta_{x,y} \in (0,1)\), it holds that \(D_{\alpha}(g) \sim \sin^2\left(\frac{x-y}{2}\right) \sim \frac{|x-y|^2}{4}\) as \(|x-y|\) approaches \(0\) and
\begin{equation*}
\begin{split}
& {\resizebox{.9\hsize}{!}{$ \partial^{k-1}I_2\sim C\displaystyle\fint 
{}-{}\frac{\partial^{k-1}f_{i2}(x)-\partial^{k-1}f_{i2}(y)}{|\sin(%
\frac{x-y}{2})|^\alpha}\times \left(\frac{|\varepsilon|^{2+2\alpha}b_i^{2+2\alpha}(\Delta
f_{i2}^2-\Delta
f_{i1}^2)}{|x-y|^{1+\alpha}}+4\varepsilon|\varepsilon|^\alpha b_i^{1+\alpha}(f_{i2}-f_{i1}+\tilde
f_{i2}-\tilde f_{i1})\right)dy$}} \\
& \qquad \qquad \qquad +l.o.t.
\end{split}%
\end{equation*}%
Consequently, we can easily establish that 
\[
\| I_{2} \|_{Y^{k-1}} \leq C \| f_{i1} - f_{i2} \|_{X^{k+\alpha-1}},
\]
which leads to the conclusion that \(\mathcal{F}^\alpha_{i2}\) is continuous. We can now apply Taylor's formula to \(\mathcal{F}^\alpha_{i22}\) and \(\mathcal{F}^\alpha_{i23}\) to derive
\begin{equation}
{\resizebox{.96\hsize}{!}{$	\mathcal{F}^\alpha_{i,2}=C_\alpha\gamma _{i}\left(1-\frac{\alpha}{2}\right)\displaystyle\fint\frac{f_{i}(x-y)\sin
(y)dy}{|\sin (\frac{y}{2})|^\alpha}-C_\alpha \gamma _{i}\int
\!\!\!\!\!\!\!\!\!\;{}-{}\frac{(f_{i}^{\prime }(x)-f_{i}^{\prime }(x-y))\cos
(y)dy}{|\sin (\frac{y}{2})|^\alpha}+\varepsilon |\varepsilon |^\alpha\mathcal{R}%
_{i2}(\varepsilon ,f_i),$}}
\label{2-8}
\end{equation}%
where \(\mathcal{R}_{i2}\) remains continuous as well.

Next, we turn our attention to the final component of the functional $\mathcal{F}^\alpha_{i}$, namely $\mathcal{F}^\alpha_{i3}$. Once again to manage the singularities at $\varepsilon =0$, we use \eqref{comput}
and the apply the Taylor formula (\ref{taylor}) on the element $\mathcal{F}^\alpha_{i3}$
\begin{equation}\label{333}
    \begin{aligned}
        &{\resizebox{.99\hsize}{!}{$\mathcal{F}^\alpha_{i3}=\frac{C_\alpha }{ \varepsilon   R_i( x)} \displaystyle\sum_{j \neq i}  \frac{\gamma_j}{b_j}\displaystyle\fint \frac{\left(\left(  R_i( x) R_j( y)+  R_i^{\prime}( x) R_j^{\prime}( y)\right) \sin ( x- y)+\left(  R_i( x) R_j^{\prime}( y)-  R_i^{\prime}( x) R_j( y)\right) \cos ( x- y)\right) d  y}{\left(A_{i, j}+\varepsilon B_{i, j}( x,  y)\right)^{\frac{\alpha}{2}}}$}} \\
        &
        {\resizebox{.99\hsize}{!}{$=\frac{C_\alpha }{\varepsilon }\displaystyle\sum_{j=1,\,j \neq i}^{N}  \frac{\gamma_j}{b_j}\displaystyle\fint \frac{(1+\varepsilon|\varepsilon|^\alpha b_j^{1+\alpha} f_j( y))\sin( x- y)d y}{\left(A_{i, j}+\varepsilon B_{i, j}( x,  y)\right)^{\frac{\alpha}{2}}}+ \frac{C_\alpha \varepsilon|\varepsilon|^{2\alpha}b_i^{1+\alpha}}{  1+\varepsilon|\varepsilon|^\alpha b_i^{1+\alpha}f_i( x)}\displaystyle\sum_{j=1,\,j \neq i}^{N}  \gamma_j b_j^\alpha\displaystyle\fint \frac{g'_i( x)g'_j( y)\sin( x- y)d y}{\left(A_{i, j}+\varepsilon B_{i, j}( x,  y)\right)^{\frac{\alpha}{2}}}$}}\\
		&
  \quad {\resizebox{.97\hsize}{!}{$+C_\alpha|\varepsilon|^\alpha \displaystyle\sum_{j=1,\,j \neq i}^{N}   \gamma_j b_j^\alpha\displaystyle\fint \frac{f'_j( y)\cos( x- y)d y}{\left(A_{i, j}+\varepsilon B_{i, j}( x,  y)\right)^{\frac{\alpha}{2}}}-\frac{C_\alpha  |\varepsilon|^\alpha b_i^{1+\alpha}f'_i( x)}{  1+\varepsilon|\varepsilon|^{\alpha}b_i^{1+\alpha} f_i( x)}\displaystyle\sum_{j=1,\,j \neq i}^{N} \frac{\gamma_j}{b_j}\displaystyle\fint \frac{(1+\varepsilon|\varepsilon|^\alpha b_j^{1+\alpha} f_j( y))\cos( x- y)d y}{\left(A_{i, j}+\varepsilon B_{i, j}( x,  y)\right)^{\frac{\alpha}{2}}}$}}\\			
        & 
        =\frac{C_\alpha}{\varepsilon} \displaystyle\sum_{j=1,\,j \neq i}^{N}\frac{\gamma_j}{b_j}\int\!\!\!\!\!\!\!\!\!\; {}-{} \frac{\sin( x- y)d y}{\left(A_{i, j}+\varepsilon B_{i, j}( x,  y)\right)^{\frac{\alpha}{2}}}+\varepsilon \mathcal{R}_{i, 3}(\varepsilon ,f,\lambda)\\
        &
          {\resizebox{.99\hsize}{!}{$=\frac{C_\alpha }{\varepsilon }\displaystyle\sum_{j=1,\,j \neq i}^{N}\frac{\gamma_j}{b_j}\displaystyle
        \fint \frac{\sin( x- y)d y}{\left(A_{i, j}\right)^{\frac{\alpha}{2}}}-\frac{\alpha C_\alpha }{2\varepsilon }\displaystyle\sum_{j=1,\,j \neq i}^{N}\frac{\gamma_j}{b_j}\displaystyle
        \fint\int_0^1\frac{\varepsilon  B_{i,j}\sin( x- y)d y d\tau}{\left(A_{i, j}+\tau\varepsilon B_{i, j}( x,  y)\right)^{\frac{\alpha}{2}+1}}+\varepsilon \mathcal{R}_{i, 3}(\varepsilon ,f,\lambda) $}}\\
        &
        =-\frac{\alpha C_\alpha }{2\varepsilon }\displaystyle\sum_{j=1,\,j \neq i}^{N} \frac{\gamma_j}{b_j}\int\!\!\!\!\!\!\!\!\!\; {}-{}\frac{\varepsilon  B_{i,j}\sin( x- y)d y }{\left(A_{i, j}\right)^{\frac{\alpha}{2}+1}}\\
&
        \qquad-\frac{\alpha (\alpha+2)C_\alpha}{4\varepsilon}\sum_{j \neq i} \frac{\gamma_j}{b_j} \int\!\!\!\!\!\!\!\!\!\; {}-{}\int_0^1\int_0^1\frac{\varepsilon^2  \tau B^2_{i,j}\sin( x- y)d y d\tau dt}{\left(A_{i, j}+t\tau\varepsilon B_{i, j}( x,  y)\right)^{\frac{\alpha}{2}+2}}+\varepsilon \mathcal{R}_{i, 3}(\varepsilon ,f,\lambda)  \\
        &
        = {\resizebox{.96\hsize}{!}{$-\frac{\alpha C_\alpha}{2\varepsilon }\displaystyle\sum_{j=1,\,j \neq i}^{N}\frac{\gamma_j}{b_j}\displaystyle\fint\frac{ 2\varepsilon(w_i-w_j)\cdot(b_i  (\cos (x), \sin (x))-b_j(\cos (y), \sin (y)))\sin( x- y)d y }{\left|w_i-w_j\right|^{\alpha+2}}+\varepsilon \mathcal{R}_{i, 3}(\varepsilon ,f,\lambda)$}}\\
        &
        =\frac{\alpha C_\alpha  }{2\pi } \displaystyle\sum_{j=1,\,j \neq i}^{N} \gamma_j \int_0^{2 \pi} \frac{\sin ( x- y)\left(w_i-w_j\right) \cdot(\cos  (y), \sin  (y))}{\left|w_i-w_j\right|^{\alpha+2}} d  y+\varepsilon \mathcal{R}_{i, 3}(\varepsilon ,f,\lambda) \\
        &
        = \frac{\alpha C_\alpha}{2}\displaystyle\sum_{j=1,\,j \neq i}^{N} \gamma_j \frac{\left(w_i-w_j\right) \cdot(\sin  (x),-\cos  (x))}{\left|w_i-w_j\right|^{\alpha+2}}+\varepsilon \mathcal{R}_{i3}(\varepsilon ,f,\lambda) .
    \end{aligned}
\end{equation}
Throughout the computations, we used $\displaystyle\fint\sin(x-y)\,dy=0$, $\displaystyle\fint\sin(x-y)\,(\cos y,\sin y)\,dy=-\frac12(-\sin x,\cos x)$, $\widehat C_\alpha=\alpha C_\alpha$, and the fact that, by \eqref{comput},
\[
B_{ij}(x,y)=2(w_i-w_j)\cdot\big(b_i(\cos x,\sin x)-b_j(\cos y,\sin y)\big)+O(\varepsilon).
\] The function $\mathcal R_{i3}$  is bounded and continuous. Notice that, for $\varepsilon_0$ small enough and $\lambda\in\Lambda$,
\begin{equation*}
\begin{split}
|z_i(x)-z_j(y)|&=\big|\varepsilon b_{i}R_{i}(x)(\cos x,\sin x)+w_i-\varepsilon b_{j}R_{j}(y)(\cos y,\sin y)-w_j\big|\\
&\geq \tfrac12|w_i-w_j|>0,\qquad j\neq i,
\end{split}
\end{equation*}%
so that the kernels appearing in \eqref{333} are smooth. As a result, we conclude that $\mathcal{F}^\alpha_{i3}$ is less singular than $\mathcal{F}^\alpha_{i2}$. Moreover, it follows that $\mathcal{F}^\alpha_{i3}(\varepsilon, f, \lambda)$ remains continuous.
Thus, we conclude the proof of the continuity of the functional $\mathcal{F}^\alpha_{i}$.
This completes the proof.
\end{proof}

By combining \eqref{2-1}, \eqref{2-8}, and \eqref{333}, we obtain the following expressions
\begin{equation}\label{3-7}
\left\{
    \begin{aligned}
&\mathcal{F}^\alpha_{i, 1}=\big(\Omega w_i+U\mathbf e_1\big)\cdot (-\sin x,\cos x)+\varepsilon|\varepsilon|^\alpha \mathcal{R}_{i 1}(\varepsilon ,f,\lambda),\\
&\mathcal{F}^\alpha_{i,2}=C_\alpha\gamma _{i}\Big(1-\frac{\alpha}{2}\Big)\fint   \frac{f_i( x- y)\sin y\,d y}{\big(4\sin^2(\frac{ y}{2})\big)^{\frac{\alpha}{2}}}-C_\alpha \gamma_i\fint \frac{\big(f'_i( x)-f'_i( x- y)\big)\cos y\,d y}{\big(4\sin^2(\frac{ y}{2})\big)^{\frac{\alpha}{2}}}\\
&\qquad\quad+\varepsilon|\varepsilon|^\alpha\mathcal{R}_{i2}(\varepsilon ,f,\lambda),\\
& \mathcal{F}^\alpha_{i,3}= -\sum_{j=1,\,j \neq i}^{N}  \frac{\widehat C_\alpha \gamma_j\, (w_i-w_j)\cdot (-\sin x,\cos x)}{2|w_i-w_j|^{\alpha+2}}+\varepsilon \mathcal{R}_{i3}(\varepsilon ,f,\lambda),
    \end{aligned}
\right.
\end{equation}
where $\mathcal{R}_{ik}$, for $k=1,\cdots, 3$, are bounded and smooth.

The subsequent proposition focuses on the $C^1$ smoothness of the functional $\mathcal{F}^\alpha$ with respect to $f$.
\begin{proposition}\label{lem2-3}
Let $\alpha\in(1,2)$ and let $\lambda^*$ solve \eqref{alg-sysP}. With $\varepsilon_0$ and $\Lambda$ as in Proposition~\ref{p3-1}, $\mathcal F^\alpha$ can be extended to a $C^1$ mapping $(-\varepsilon_0,\varepsilon_0)\times\mathcal B_X\times\Lambda\to\mathcal Y^{k-1}$. More precisely, the Gateaux derivatives \( \partial_{f_i} \mathcal{F}^\alpha_i(\varepsilon, f, \lambda) h_i\) and \( \partial_{f_j} \mathcal{F}^\alpha_i(\varepsilon, f, \lambda) h_j\), $j\neq i$, exist for all $(\varepsilon,f,\lambda)\in(-\varepsilon_0, \varepsilon_0) \times \mathcal{B}_X \times \Lambda$ and depend continuously on $(\varepsilon,f,\lambda)$ as operators from $X^{k+\alpha-1}$ to $Y^{k-1}$, and $\mathcal F^\alpha$ is continuously differentiable with respect to $\varepsilon$ and $\lambda$. Moreover, for all $(f,\lambda)\in\mathcal B_X\times\Lambda$ and $h\in\mathcal X^{k+\alpha-1}$,
\begin{equation}\label{DfF0}
\partial_{f_i}\mathcal F^\alpha_i(0,f,\lambda)h_i=\mathcal F^\alpha_{i,2}(0,h_i),\qquad \partial_{f_j}\mathcal F^\alpha_i(0,f,\lambda)h_j=0,\quad j\neq i .
\end{equation}
\end{proposition}

\begin{proof}
We claim $\partial_{ f_i} \mathcal{F}^\alpha_{i,2}(\varepsilon, f_i)h_i=\partial_{ f_i} \mathcal{F}^\alpha_{i11}+\partial_{ f_i} \mathcal{F}^\alpha_{i22}+\partial_{ f_i} \mathcal{F}^\alpha_{i23}+\partial_{ f_i}\mathcal{F}^\alpha_{i24}$ is continuous, where
	\begin{equation}\label{2-11}
		\begin{split}		
    			&{\resizebox{.99\hsize}{!}{$\partial_{ f_i} \mathcal{F}^\alpha_{i21}=C_\alpha\gamma _{i}\displaystyle\fint \frac{h_i( y)\sin( x- y)d y}{\left( |\varepsilon|^{2+2\alpha}b_i^{2+2\alpha}\left(f_i( x)-f_i( y)\right)^2+4(1+\varepsilon|\varepsilon|^\alpha b_i^{1+\alpha} f_i( x))(1+\varepsilon|\varepsilon|^\alpha b_i^{1+\alpha} f_i( y))\sin^2\left(\frac{ x- y}{2}\right)\right)^{\frac{\alpha}{2}}}$}}\\
			& \ \ \ \ {\resizebox{.9\hsize}{!}{$-\frac{\alpha C_\alpha \gamma_i}{2}\displaystyle \fint \frac{(1+\varepsilon|\varepsilon|^\alpha b_i^{1+\alpha}  f_i( y))\sin( x- y)}{\left( |\varepsilon|^{2+2\alpha}b_i^{2+2\alpha}\left(f_i( x)-f_i( y)\right)^2+4(1+\varepsilon|\varepsilon|^\alpha b_i^{1+\alpha} f_i( x))(1+\varepsilon|\varepsilon|^\alpha b_i^{1+\alpha}  f_i( y))\sin^2\left(\frac{ x- y}{2}\right)\right)^{\frac{\alpha+2}{2}}}$}}\\
			& \ \ \ \ \times\bigg(2\varepsilon|\varepsilon|^\alpha b_i^{1+\alpha}(f_i( x)-f_i( y))(h_i( x)-h_i( y))\\
			& \ \ \ \ \ \ \ \ +4(h_i( x)(1+\varepsilon|\varepsilon|^\alpha b_i^{1+\alpha}  f_i( y))+h_i( y)(1+\varepsilon|\varepsilon|^\alpha b_i^{1+\alpha} f_i( x)))\sin^2\left(\frac{ x- y}{2}\right)\bigg)d y,
		\end{split}
	\end{equation}
    \begin{equation}\label{2-12}
    	\begin{split}
    		&\partial_{ f_i} \mathcal{F}^\alpha_{i22}=C_\alpha\gamma_i{\resizebox{.85\hsize}{!}{$\displaystyle\fint \frac{(h'_i( y)-h'_i( x))\cos( x- y)d y}{\left( |\varepsilon|^{2+2\alpha}\left(f_i( x)-f_i( y)\right)^2+4(1+\varepsilon|\varepsilon|^\alpha b_i^{1+\alpha} f_i( x))(1+\varepsilon|\varepsilon|^\alpha b_i^{1+\alpha} f_i( y))\sin^2\left(\frac{ x- y}{2}\right)\right)^{\frac{\alpha}{2}}}$}}\\
    		&  \ \ \ {\resizebox{.99\hsize}{!}{$-\frac{\alpha C_\alpha\gamma_i\varepsilon|\varepsilon|^\alpha b_i^{1+\alpha}}{2}\displaystyle \fint \frac{(g'_i( y)-g'_i( x))\cos( x- y)}{\left( |\varepsilon|^{2+2\alpha}\left(f_i( x)-f_i( y)\right)^2+4(1+\varepsilon|\varepsilon|^\alpha b_i^{1+\alpha} f_i( x))(1+\varepsilon|\varepsilon|^\alpha b_i^{1+\alpha} f_i( y))\sin^2\left(\frac{ x- y}{2}\right)\right)^{\frac{\alpha+2}{2}}}$}}\\
    			&  \ \ \ \times\bigg(\varepsilon|\varepsilon|^\alpha b_i^{1+\alpha}(2(f_i( x)-f_i( y))(h_i( x)-h_i( y))\\
			& \ \ \ \ \ \ \ \ +4(h_i( x)(1+\varepsilon|\varepsilon|^\alpha b_i^{1+\alpha}  f_i( y))+h_i( y)(1+\varepsilon|\varepsilon|^\alpha b_i^{1+\alpha} f_i( x)))\sin^2\left(\frac{ x- y}{2}\right)\bigg)d y,
    	\end{split}
    \end{equation}
    \begin{equation}\label{2-13}
    	\begin{split}
    		&\partial_{ f_i} \mathcal{F}^\alpha_{i23}={\resizebox{.9\hsize}{!}{$\frac{C_\alpha \gamma_i \varepsilon|\varepsilon|^\alpha b_i^{1+\alpha} h'_i( x)}{1+\varepsilon|\varepsilon|^\alpha b_i^{1+\alpha} f_i( x)}\displaystyle\fint \frac{(f_i( x)-f_i( y))\cos( x- y)d y}{\left( |\varepsilon|^{2+2\alpha}b_i^{2+2\alpha}\left(f_i( x)-f_i( y)\right)^2+4(1+\varepsilon|\varepsilon|^\alpha b_i^{1+\alpha} f_i( x))(1+\varepsilon|\varepsilon|^\alpha b_i^{1+\alpha} f_i( y))\sin^2\left(\frac{ x- y}{2}\right)\right)^{\frac{\alpha}{2}}}$}}\\
    		&
            \ \ \ \ +{\resizebox{.93\hsize}{!}{$\frac{C_\alpha  \gamma_i\varepsilon|\varepsilon|^\alpha b_i^{1+\alpha} 
            f'_i( x)}{1+\varepsilon|\varepsilon|^\alpha b_i^{1+\alpha} f_i( x)}\displaystyle\fint\frac{(h_i( x)-h_i( y))\cos( x- y)d y}{\left( |\varepsilon|^{2+2\alpha}b_i^{2+2\alpha}\left(f_i( x)-f_i( y)\right)^2+4(1+\varepsilon|\varepsilon|^\alpha b_i^{1+\alpha} f_i( x))(1+\varepsilon|\varepsilon|^\alpha b_i^{1+\alpha} f_i( y))\sin^2\left(\frac{ x- y}{2}\right)\right)^{\frac{\alpha}{2}}}$}}\\
            &
            \ \ \ \ -{\resizebox{.9\hsize}{!}{$\frac{C_\alpha  \gamma_i|\varepsilon|^{2+2\alpha}b_i^{2+2\alpha} f'_i( x)h_i( x)}{(1+\varepsilon|\varepsilon|^\alpha b_i^{1+\alpha}
            f_i( x))^2}\displaystyle\fint \frac{(f_i( x)-f_i( y))\cos( x- y)d y}{\left( |\varepsilon|^{2+2\alpha}b_i^{2+2\alpha}\left(f_i( x)-f_i( y)\right)^2+4(1+\varepsilon|\varepsilon|^\alpha b_i^{1+\alpha} f_i( x))(1+\varepsilon|\varepsilon|^\alpha b_i^{1+\alpha} f_i( y))\sin^2\left(\frac{ x- y}{2}\right)\right)^{\frac{\alpha}{2}}}$}}\\
    		&
            \ \ \ \ -{\resizebox{.9\hsize}{!}{$\frac{\alpha C_\alpha \gamma_i |\varepsilon|^{2+2\alpha}b_i^{2+2\alpha} f'_i( x)}{2(1+\varepsilon|\varepsilon|^\alpha b_i^{1+\alpha} f_i( x))}\displaystyle\fint \frac{(f_i( x)-f_i( y))\cos( x- y)}{\left( |\varepsilon|^{2+2\alpha}b_i^{2+2\alpha}\left(f_i( x)-f_i( y)\right)^2+4(1+\varepsilon|\varepsilon|^\alpha b_i^{1+\alpha} f_i( x))(1+\varepsilon|\varepsilon|^\alpha b_i^{1+\alpha} f_i( y))\sin^2\left(\frac{ x- y}{2}\right)\right)^{\frac{\alpha+2}{2}}}$}}\\
    		&
            \ \ \ \ \times\bigg(\varepsilon|\varepsilon|^\alpha b_i^{1+\alpha}(2(f_i( x)-f_i( y))(h_i( x)-h_i( y))\\
			& \ \ \ \ \ \ \ \ +4(h_i( x)(1+\varepsilon|\varepsilon|^\alpha b_i^{1+\alpha}  f_i( y))+h_i( y)(1+\varepsilon|\varepsilon|^\alpha b_i^{1+\alpha} f_i( x)))\sin^2\left(\frac{ x- y}{2}\right)\bigg)d y
    	\end{split}
    \end{equation}
    and
    \begin{equation}\label{2-14}
    	\begin{split}
    		&\partial_{ f_i} \mathcal{F}^\alpha_{i24}={\resizebox{.9\hsize}{!}{$\frac{C_\alpha \gamma_i |\varepsilon|^{1+\alpha}b_i^{1+\alpha}}{1+\varepsilon|\varepsilon|^\alpha b_i^{1+\alpha} f_i( x)}\displaystyle\fint \frac{(h'_i( y)f'_i( x)+h'_i( x)f'_i( y))\sin( x- y)d y}{\left( |\varepsilon|^{2+2\alpha}b_i^{2+2\alpha} \left(f_i( x)-f_i( y)\right)^2+4(1+\varepsilon|\varepsilon|^\alpha b_i^{1+\alpha} f_i( x))(1+\varepsilon|\varepsilon|^\alpha b_i^{1+\alpha} f_i( y))\sin^2\left(\frac{ x- y}{2}\right)\right)^{\frac{\alpha}{2}}}$}}\\
    		& \ \ \ \ -{\resizebox{.93\hsize}{!}{$\frac{\alpha C_\alpha \gamma_i|\varepsilon|^{2+2\alpha}b_i^{2+2\alpha}}{2(1+\varepsilon|\varepsilon|^\alpha b_i^{1+\alpha} f_i( x))}\displaystyle\fint  \frac{f'_i( x)f'_i( y)\sin( x- y)}{\left( |\varepsilon|^{2+2\alpha}b_i^{2+2\alpha} \left(f_i( x)-f_i( y)\right)^2+4(1+\varepsilon|\varepsilon|^\alpha b_i^{1+\alpha} f_i( x))(1+\varepsilon|\varepsilon|^\alpha b_i^{1+\alpha} f_i( y))\sin^2\left(\frac{ x- y}{2}\right)\right)^{\frac{\alpha+2}{2}}}$}}\\
    		& \ \ \ \ \times\bigg(\varepsilon|\varepsilon|^\alpha b_i^{1+\alpha}(2(f_i( x)-f_i( y))(h_i( x)-h_i( y))\\
			& \ \ \ \ \ \ \ \ +4(h_i( x)(1+\varepsilon|\varepsilon|^\alpha b_i^{1+\alpha}  f_i( y))+h_i( y)(1+\varepsilon|\varepsilon|^\alpha b_i^{1+\alpha} f_i( x)))\sin^2\left(\frac{ x- y}{2}\right)\bigg)d y.\\
			&\ \ \ \ \ \ \ \ \ {\resizebox{.93\hsize}{!}{$-\frac{C_{\alpha}\gamma_i|\varepsilon|^{1+\alpha}b_i^{1+\alpha}f'_{i}( x)}{(1+\varepsilon|\varepsilon|^{\alpha}b_i^{1+\alpha}f_{i}( x))^{2}}\displaystyle\fint  \frac{f'_{i}( y)h_i( x)\sin{( x- y)}}{\left( |\varepsilon|^{2+2\alpha}b_i^{2+2\alpha} \left(f_i( x)-f_i( y)\right)^2+4(1+\varepsilon|\varepsilon|^\alpha b_i^{1+\alpha} f_i( x))(1+\varepsilon|\varepsilon|^\alpha b_i^{1+\alpha} f_i( y))\sin^2\left(\frac{ x- y}{2}\right)\right)^{\frac{\alpha}{2}}}\, d y.$}}
    	\end{split}
    \end{equation}
    \begin{equation}\label{partial_i2}
  \partial_{ f_i} \mathcal{F}^\alpha_{i,1}(\varepsilon ,f,\lambda) = O(\varepsilon) .
\end{equation}
\begin{equation}\label{partial_i3}
  \partial_{ f_i}  \mathcal{F}^\alpha_{i,3}(\varepsilon ,f,\lambda) = O(\varepsilon) .
\end{equation}
and
\begin{equation}\label{partial_j2}
  \partial_{ f_j} \mathcal{F}^\alpha_{i}(\varepsilon ,f,\lambda) = O(\varepsilon) .
\end{equation}

The first step is to demonstrate
    \begin{equation*}
        \lim\limits_{t\to0}\left\|\frac{\mathcal{F}^\alpha_{i2l}(\varepsilon, f_i+th_i)-\mathcal{F}^\alpha_{i2l}(\varepsilon, f_i)}{t}-\partial_{f_i}\mathcal{F}^\alpha_{i2l}(\varepsilon, g,h_i)\right\|_{Y^{k-1}}\to 0
    \end{equation*}
for $l = 1, \dots, 4$, we will focus on the most singular case, namely when $l = 2$, and adopt the notation from Proposition \ref{2-2}. In this scenario, we have the following result
\begin{equation*}
		\begin{split}
			&\frac{\mathcal{F}^\alpha_{i22}(\varepsilon, f_i+th_i)-\mathcal{F}^\alpha_{i22}(\varepsilon, f_i)}{t}-\partial_{f_i}\mathcal{F}^\alpha_{i22}(\varepsilon, f_i,h_i)\\
			&=\frac{1}{t}\int\!\!\!\!\!\!\!\!\!\; {}-{}{\resizebox{.9\hsize}{!}{$(f'_i( x)-f'_i( y))\cos( x- y)\bigg(\frac{1}{D_\alpha(f_i+th_i)^{\alpha/2}}-\frac{1}{D_\alpha(f_i)}+t\frac{\Delta f_i\Delta h_i+2(h_i\tilde{R}_i+\tilde{h}_iR_i)\sin^2(\frac{ x- y}{2})}{D_\alpha(f_i)}\bigg)d y$}}\\
			&\ \ \ \ +\int\!\!\!\!\!\!\!\!\!\; {}-{}(h'_i( x)-h'_i( y))\cos( x- y)\bigg(\frac{1}{D_\alpha(f_i+th_i)}-\frac{1}{D_\alpha(f_i)}\bigg)d y\\
			&=\mathcal{G}^\alpha_{i21}+\mathcal{G}^\alpha_{i22}.
		\end{split}
	\end{equation*}
By taking $\partial^{k-1}$ derivatives of $\mathcal{G}^\alpha_{i21}$, we can infer the following
    \begin{equation*}
    	\begin{split}
    	  \partial^{k-1}\mathcal{G}^\alpha_{i21}&=\frac{1}{t}\int\!\!\!\!\!\!\!\!\!\; {}-{}
            \bigg(\frac{1}  {D_\alpha(f_i+th_i)^{\alpha/2}}-\frac{1}{D_\alpha(f_i)^{\alpha/2}}+t\frac{\Delta f_i\Delta     h+2(\tilde{R}_i h_i+\tilde{h}_i R_i)\sin^2(\frac{ x- y}{2})}{D_\alpha(f_i)^{\alpha/2}}\bigg)\\
    	  & \ \ \ \ \times(\partial^kf_i( x)-\partial^kf_i( y))\cos( x- y)d y+l.o.t.
    	\end{split}
    \end{equation*}
Using the mean value theorem, we derive the following result
    \begin{equation*}
    	\frac{1}{D_\alpha(f_i+th_i)^{\alpha/2}}-\frac{1}{D_\alpha(f_i)^{\alpha/2}}+t\frac{\Delta f_i\Delta h_i+2(\tilde{R}_i h_i+\tilde{h}_i R_i)\sin^2(\frac{ x- y}{2})}{D_\alpha(f_i)^{\alpha/2}}\sim \frac{Ct^2}{|\sin(\frac{ x- y}{2})|^{\alpha/2}}\varphi(\varepsilon,f_i,h_i),
    \end{equation*}
    with $\|\varphi(\varepsilon,f_i,h_i)\|_{L^\infty}<\infty$. Hence, we can conclude that
    \begin{equation*}
    	\|\mathcal{G}^\alpha_{i21}\|_{Y^{k-1}}\le Ct\left\|\int\!\!\!\!\!\!\!\!\!\; {}-{}\frac{\partial^kf_i( x)-\partial^kf_i( y)}{|\sin(\frac{ x- y}{2})|^{\alpha/2}}d y\right\|_{L^2}\le Ct\|f_i\|_{X^{k+\alpha-1}}.
    \end{equation*}
   Following the same reasoning as in \eqref{2-7}, we can similarly show that $\norm{\mathcal{G}^\alpha_{i22}}_{Y^{k-1}}$ is bounded by $Ct\norm{f_i}_{X^{k+\alpha-1}}$. Therefore, by letting $t \rightarrow 0$, we conclude that
   \begin{equation*}
\left\|\frac{\mathcal{F}^\alpha_{i22}(\varepsilon, f_i+th_i)-\mathcal{F}^\alpha_{i22}(\varepsilon, f_i)}{t}-\partial_{f_i}\mathcal{F}^\alpha_{i22}(\varepsilon, f_i)h_i\right\|_{Y^{k-1}}\leq C|t|\,\|h_i\|^2_{X^{k+\alpha-1}}\xrightarrow[t\to0]{}0 .
\end{equation*}
This concludes the first step. The second step consists of proving the continuity of $\partial_{f_i} \mathcal{F}^\alpha_{i22}(\varepsilon, f)$, which also relies on \eqref{2-7}.

By applying the same approach as described earlier, we infer that
    \begin{equation}\label{2-15}
    	\partial_{f_i} \mathcal{F}^\alpha_{i,1}(\varepsilon ,f,\lambda)h_i=|\varepsilon|\partial_{f_i}\mathcal{R}_{i,2}(\varepsilon ,f,\lambda) ,
    \end{equation}
          \begin{equation}\label{deriv_fi3}
    	\partial_{f_i}\mathcal{F}^\alpha_{i,3}(\varepsilon ,f,\lambda)h_i=|\varepsilon|\partial_{f_i}\mathcal{R}_{i,3}(\varepsilon ,f,\lambda)
    \end{equation}
    and
       \begin{equation}\label{deriv_fj}
    	\partial_{f_j}\mathcal{F}^\alpha_{i}(\varepsilon ,f,\lambda)h_j=|\varepsilon|\partial_{f_i}\mathcal{R}_{j}(\varepsilon ,f,\lambda) ,
    \end{equation}
 these functions are continuous, where $\mathcal{R}_{i,2}(\varepsilon ,f,\lambda)$, $\mathcal{R}_{i,3}(\varepsilon ,f,\lambda)$, and $\mathcal{R}_{j}(\varepsilon ,f,\lambda)$ are bounded and $C^1$ functions. 

The continuity with respect to $\varepsilon$ follows, as in Proposition~\ref{p3-1}. By applying the same approach to $\mathcal{F}^\alpha_{i21}$, $\mathcal{F}^\alpha_{i23}$ and $\mathcal{F}^\alpha_{i24}$, which are less singular, together with \eqref{2-15} and the discussion of $\mathcal{F}^\alpha_{i,3}$ above, we conclude that the Gateaux derivatives exist and are continuous.

Lemma~\ref{abstract-lemma} requires the functional to be $C^1$ with respect to all variables, including $\varepsilon$. In $\mathcal F^\alpha_{i,1}$ and $\mathcal F^\alpha_{i,2}$, the parameter $\varepsilon$ only enters through $\varepsilon b_i$ and $\varepsilon|\varepsilon|^\alpha b_i^{1+\alpha}$, which are $C^1$ functions of $\varepsilon$. Since \eqref{eq:f1}, \eqref{2-2} and \eqref{2-6} contain no negative powers of $\varepsilon$, differentiating them with respect to $\varepsilon|\varepsilon|^\alpha b_i^{1+\alpha}$, and estimating the result as above, shows that these terms are $C^1$ in $\varepsilon$. In $\mathcal F^\alpha_{i,3}$ the kernel is smooth, and $\varepsilon$ enters either smoothly or through the factors $\varepsilon|\varepsilon|^{\alpha}$ and $|\varepsilon|^\alpha$, see \eqref{333}. Since $\alpha>1$, the function $|\varepsilon|^\alpha$ is $C^1$. For $\alpha=1$, which is considered in Lemma~\ref{lem2-7}, each term of \eqref{333} carrying the factor $|\varepsilon|$ has a coefficient which vanishes at $\varepsilon=0$, because
\[
\fint f_j(y)\sin(x-y)\,dy=\fint f_j'(y)\cos(x-y)\,dy=\fint \sin(x-y)\,dy=\fint \cos(x-y)\,dy=0 ,
\]
since the functions of $X^{k+\log}$ have no Fourier modes of order $0$ and $1$, as $\varepsilon\mapsto|\varepsilon|g(\varepsilon)$ is $C^1$ whenever $g$ is $C^1$ and $g(0)=0$, these terms are also $C^1$ in $\varepsilon$. Finally, $\mathcal F^\alpha$ depends smoothly on $\lambda$. In fact $\mathcal F^\alpha_{i,1}$ is affine in $(\Omega,U,w_i)$, $\mathcal F^\alpha_{i,2}$ does not depend on $\lambda$, and the kernel of $\mathcal F^\alpha_{i,3}$ is a smooth function of $w$ and $\gamma$. Since all the partial derivatives are continuous in $(\varepsilon,f,\lambda)$, $\mathcal F^\alpha$ is of class $C^1$ with values in $Y^{k-1}$. Therefore $\mathcal F^\alpha\colon(-\varepsilon_0,\varepsilon_0)\times\mathcal B_X\times\Lambda\to\mathcal Y^{k-1}$ is of class $C^1$. This completes the proof of Proposition~\ref{lem2-3}.
\end{proof}
From \eqref{2-11}-\eqref{2-14}, by setting $\varepsilon = 0$ and $f_i \equiv 0$, we obtain
\begin{equation}\label{gateaux}
	\partial_{f_i} \mathcal{F}^\alpha_{i}(0,0,\lambda)h_i=C_\alpha\gamma_i\left(1-\frac{\alpha}{2}\right)\int\!\!\!\!\!\!\!\!\!\; {}-{} \frac{h_i( x- y)\sin( y)d y}{\left(4\sin^2(\frac{ y}{2})\right)^{\frac{\alpha}{2}}}-C_\alpha \gamma_i\int\!\!\!\!\!\!\!\!\!\; {}-{} \frac{(h'_i( x)-h'_i( x- y))\cos( y)d y}{\left(4\sin^2(\frac{ y}{2})\right)^{\frac{\alpha}{2}}} .
\end{equation}
and
\begin{equation}\label{gateaux2}
	\partial_{f_j} \mathcal{F}^\alpha_{i}(0,0,\lambda)h_j=0 .
\end{equation}
  Having established the $C^1$ regularity of the functional.  We then define the following nonlinear operator 
\begin{equation*}
\mathcal{F}^\alpha(\varepsilon,f,\lambda):=\big(\mathcal{F}^\alpha_1(\varepsilon,f,\lambda),\ldots,\mathcal{F}^\alpha_N(\varepsilon,f,\lambda)\big).\end{equation*}
Note that identify roots of the nonlinear functional $\mathcal{F}^\alpha=0$ are equivalent to identifying roots $\lambda$ of   the $\mathcal{P}^\alpha(\lambda)=0$, defined in \eqref{alg-sysP}, here $\lambda$ denotes the point vortex parameters. As a result, we can produce a collection of trivial solutions for \eqref{1-1} with $\alpha \in (1,2)$.
\begin{proposition}\label{equivalence}
For all $\lambda\in\Lambda$ and $i=1,\dots,N$,
\begin{equation}\label{F00}
\mathcal F^\alpha_i(0,0,\lambda)(x)=\mathcal P_i^\alpha(\lambda)\cdot (-\sin x,\cos x),\qquad x\in[0,2\pi),
\end{equation}
where $\mathcal P_i^\alpha$ is defined in \eqref{alg-sysP}. In particular, $\mathcal{F}^\alpha(0, 0, \lambda)=0$ if and only if $\mathcal{P}^\alpha(\lambda)=0$.
\end{proposition}

\begin{proof}
This follows directly from the expansions \eqref{3-7}, since $\mathcal F^\alpha_{i,2}(0,0,\lambda)=0$, together with the definition \eqref{alg-sysP} of $\mathcal{P}_i^\alpha(\lambda)$ and $\widehat C_\alpha=\alpha C_\alpha$.
\end{proof}

\subsection{Linearization of the functional}

In this subsection, we concentrate on the linearization of the functional introduced in Section~\ref{section2}. By Proposition~\ref{equivalence}, $(0,0,\lambda)$ is a zero of $\mathcal{F}^\alpha$ if and only if $\mathcal{P}^\alpha(\lambda)=0$. We now analyze the linearization of $\mathcal{F}^\alpha$ at $(0,0,\lambda)$.

By \eqref{3-7} and \eqref{DfF0}, taking \(\varepsilon = 0\) and \(f_i \equiv 0\) for all \(i = 1, \ldots, N\), the derivatives can be expressed as follows
\begin{equation}\label{4-1}
    \left\{
    \begin{aligned}
        &\partial_{f_i} \mathcal{F}^\alpha_i(0, 0, \lambda) h_i = C_\alpha\gamma_i\Big(1-\frac{\alpha}{2}\Big) \fint\frac{h_i( x-  y)\sin y\,d y}{\big(4\sin^2(\frac{ y}{2})\big)^{\frac{\alpha}{2}}}\\
        &\qquad\qquad\qquad\qquad\ \qquad\ -C_\alpha \gamma_i\fint \frac{\big(h'_i( x)-h'_i( x- y)\big)\cos y\,d y}{\big(4\sin^2(\frac{ y}{2})\big)^{\frac{\alpha}{2}}},\\
        &\partial_{f_j}\mathcal{F}^\alpha_i(0, 0, \lambda) h_j = 0, \qquad j \neq i.
    \end{aligned}
    \right.
\end{equation}

\begin{proposition}\label{iso}
    Let $\alpha\in(1,2)$ and $h=(h_1,\ldots,h_N)\in \mathcal{X}^{k+\alpha-1} $, where
    \[    h_i( x)=\sum_{n=2}^{\infty}\left(a^i_n \cos (n  x)+d^i_n \sin (n  x)\right).\]
    Then the following holds
    \begin{equation*}
        \begin{aligned}
            &\partial_{f_i}  \mathcal{F}^\alpha_i(0, 0, \lambda) h_i = \gamma_i\sum_{n=2}^{\infty} n \sigma_n  \left(a^i_n \sin (n  x)-d^i_n \cos (n  x)\right), \\
            &\partial_{f_j}  \mathcal{F}^\alpha_i(0, 0, \lambda) h_j = 0, \quad j \neq i,
        \end{aligned}
    \end{equation*}
    where
    \begin{equation}\label{sigma}
        \sigma_n=2^{\alpha-1}\frac{\Gamma(1-\alpha)}{\Gamma(1-\frac{\alpha}{2})^2} \left(\frac{\Gamma(1+\frac{\alpha}{2})}{\Gamma(2-\frac{\alpha}{2})}-\frac{\Gamma(n+\frac{\alpha}{2})}{\Gamma(1+n-\frac{\alpha}{2})}\right), \quad n \geq 1 .
    \end{equation}
The sequence $(\sigma_n)_{n\geq1}$ is increasing, $\sigma_1=0$, $\sigma_n>0$ for $n\geq2$, and $c_1n^{\alpha-1}\leq\sigma_n\leq c_2n^{\alpha-1}$ for $n\geq2$ and some constants $0<c_1\leq c_2$. Consequently, $\partial_{f_{i}} \mathcal{F}^\alpha_i(0, 0, \lambda)$ is an isomorphism from $X^{k+\alpha-1}$ onto $Y^{k-1}_0$. Moreover, the Gateaux derivative of $\mathcal{F}^\alpha$ with respect to $f$ at $(0,0,\lambda)$ is given by
      \begin{align}\label{eq:linearization}
    D_f \mathcal{F}^\alpha(0, 0, \lambda)h(x)=&\sum_{n=2}^{\infty}n \sigma_n \begin{pmatrix}
      \gamma_1   \left(a^1_n \sin (n  x)-d^1_n \cos (n  x)\right) \\
      \vdots
      \\
       \gamma_N   \left(a^N_n \sin (n  x)-d^N_n \cos (n  x)\right)
    \end{pmatrix} ,
  \end{align}
and, for any $\lambda \in \Lambda$, the linear operator
  $D_{f}\mathcal{F}^\alpha (0,0,\lambda)\colon \mathcal{X}^{k+\alpha-1}\rightarrow \mathcal{Y}^{k-1}_0$ is an isomorphism.
\end{proposition}

\begin{proof}
The direct computations derived in \eqref{gateaux} yield
   \begin{equation*}
	\partial_{f_i} \mathcal{F}^\alpha_{i}(0,0,\lambda)h_i=C_\alpha\gamma_i\left(1-\frac{\alpha}{2}\right)\int\!\!\!\!\!\!\!\!\!\; {}-{} \frac{h_i( x- y)\sin( y)d y}{\left(4\sin^2(\frac{ y}{2})\right)^{\frac{\alpha}{2}}}-C_\alpha \gamma_i\int\!\!\!\!\!\!\!\!\!\; {}-{} \frac{(h'_i( x)-h'_i( x- y))\cos( y)d y}{\left(4\sin^2(\frac{ y}{2})\right)^{\frac{\alpha}{2}}} .
\end{equation*}
and
\begin{equation*}
	\partial_{f_j} \mathcal{F}^\alpha_{i}(0,0,\lambda)h_i=0 .
\end{equation*}
  Now, It is well-known that  
  \begin{equation*}
      \partial_{f_i}  \mathcal{F}^\alpha_{i}(0,0,\lambda)(a^i_n\cos(n x))=a^i_n n\sigma_n  \sin (n  x),\quad \mbox{for}\quad n\geq 2 ,
  \end{equation*}
  see for instance \cite[Proposition 4.1]{Cas1}. Let $h_i\in X^{k+\alpha-1} $ be
    \begin{equation*}
        h_i( x)=\sum_{n=2}^{\infty}\left(a^i_n \cos (n  x)+d^i_n \sin (n  x)\right),
    \end{equation*}
Next, we compute how $\partial_{f_i} \mathcal{F}^\alpha_{i}(0,0,\lambda)(\sin(n x))$ acts on $-n\sigma_n \cos (n x)$ for each $n \ge 1$. By \eqref{4-1}, we have
	\begin{equation}\label{sin}
		C_\alpha\gamma_i\left(1-\frac{\alpha}{2}\right)\int\!\!\!\!\!\!\!\!\!\; {}-{} \frac{d^i_n \sin(n x-n y)\sin( y)d y}{\left(4\sin^2(\frac{ y}{2})\right)^{\frac{\alpha}{2}}}-C_\alpha \gamma_i n\int\!\!\!\!\!\!\!\!\!\; {}-{} \frac{(d^i_n\cos(n x)-d^i_n\cos(n x-n y))\cos( y)d y}{\left(4\sin^2(\frac{ y}{2})\right)^{\frac{\alpha}{2}}}.
	\end{equation}
	Using identity
    \begin{equation*}
        \cos\left(\frac{y}{2}\right)\left|\sin\left(\frac{y}{2}\right)\right|^{1-\alpha}=\frac{2}{2-\alpha}\partial_y\left(\left|\sin\left(\frac{y}{2}\right)\right|^{2-\alpha}\right),
    \end{equation*}
    and  the following identity (see \cite{Cas1})
    \begin{equation}\label{2-16}
        \int_0^\pi(\sin( y))^{2-\alpha}e^{\text{i} n y}d y=\frac{\pi e^{n\pi\text i}\Gamma(3-\alpha)}{2^{2-\alpha}\Gamma(2+n-\frac{\alpha}{2})\Gamma(2-n-\frac{\alpha}{2})}, \ \ \ \ \forall \, \alpha<3, \ \ \forall \, n\in\mathbb{R}.
    \end{equation}
	By applying integration by parts to the first term in \eqref{sin}, we obtain
	\begin{equation}\label{2-17}
		\begin{split}
			&C_\alpha \gamma_n\left(1-\frac{\alpha}{2}\right)\int\!\!\!\!\!\!\!\!\!\; {}-{} \frac{d^i_n\sin(n x-n y)\sin( y)d y}{\left(4\sin^2(\frac{ y}{2})\right)^{\frac{\alpha}{2}}}\\
			&=2^{1-\alpha}C_\alpha \gamma_n\left(1-\frac{\alpha}{2}\right)\int\!\!\!\!\!\!\!\!\!\; {}-{}d^i_n\sin(n x-n y)\cos\left(\frac{ y}{2}\right)\left|\sin\left(\frac{y}{2}\right)\right|^{1-\alpha}d y\\
			&=2^{1-\alpha}C_\alpha \gamma_n\left(1-\frac{\alpha}{2}\right)\frac{2n}{2-\alpha}d^i_n\int\!\!\!\!\!\!\!\!\!\; {}-{} \cos(n x-n y)\left|\sin\left(\frac{y}{2}\right)\right|^{2-\alpha}d y\\
			&=2^{1-\alpha}C_\alpha \gamma_n\left(1-\frac{\alpha}{2}\right)\frac{2n}{2-\alpha}d^i_n\cos(n x)\int\!\!\!\!\!\!\!\!\!\; {}-{} \cos(n y)\left|\sin\left(\frac{y}{2}\right)\right|^{2-\alpha}d y\\
			&=2^{1-\alpha}C_\alpha \gamma_n\left(1-\frac{\alpha}{2}\right)\frac{2n}{2-\alpha}\frac{d^i_n}{2\pi}\frac{\pi \cos(n\pi)\Gamma(3-\alpha)}{2^{1-\alpha}\Gamma(2+n-\frac{\alpha}{2})\Gamma(2-n-\frac{\alpha}{2})}\cos(n x).
		\end{split}
	\end{equation}
    For the second term in \eqref{sin} and the identity $\cos  y=1-2\sin^2(\frac{ y}{2})$, it holds
    \begin{equation*}
    	\begin{split}
    		&-C_\alpha \gamma_n n d^i_n\int\!\!\!\!\!\!\!\!\!\; {}-{} \frac{(\cos(n x)-\cos(n x-n y))\cos( y)d y}{\left(4\sin^2(\frac{ y}{2})\right)^{\frac{\alpha}{2}}}\\
    		&=-2^{-\alpha}C_\alpha \gamma_n n d^i_n\int\!\!\!\!\!\!\!\!\!\; {}-{} (\cos(n x)-\cos(n x-n y))\left|\sin\left(\frac{ y}{2}\right)\right|^{-\alpha}d y\\
    		&\ \ \ \ +2^{1-\alpha}C_\alpha \gamma_n n d^i_n\int\!\!\!\!\!\!\!\!\!\; {}-{} (\cos(n x)-\cos(n x-n y))\left|\sin\left(\frac{ y}{2}\right)\right|^{2-\alpha}d y.
    	\end{split}
    \end{equation*}
    In accordance with Lemma \ref{A-1}
     and the identity \eqref{2-16}, we can express the previous equation as
    \begin{equation}\label{2-18}
    	\begin{split}
    		& \ \ \ \ -\frac{C_\alpha}{2\pi } \gamma_n n d^i_n\frac{2\pi \Gamma(1-\alpha)}{\Gamma(\frac{\alpha}{2})\Gamma(1-\frac{\alpha}{2})}\left(\frac{\Gamma(\frac{\alpha}{2})}{\Gamma(1-\frac{\alpha}{2})}-\frac{\Gamma(j+\frac{\alpha}{2})}{\Gamma(1+j-\frac{\alpha}{2})}\right)\cos(n x)\\
    		&+2^{-1}\frac{C_\alpha}{2\pi } \gamma_n n d^i_n\frac{2\pi \Gamma(3-\alpha)}{\Gamma(\frac{\alpha}{2}-1)\Gamma(2-\frac{\alpha}{2})}\left(\frac{\Gamma(\frac{\alpha}{2}-1)}{\Gamma(2-\frac{\alpha}{2})}-\frac{\Gamma(j-1+\frac{\alpha}{2})}{\Gamma(2+j-\frac{\alpha}{2})}\right)\cos(n x)
    	\end{split}
    \end{equation}
  Summing it up, the $n$-th coefficient is exactly
    \begin{equation}\label{eqlin}
        \partial_{f_i} \mathcal{F}^\alpha_i(0, 0, \lambda)(d^i_n\sin (n x))=-d^i_n   n \sigma_n\cos (n x) .
    \end{equation}
    Then, we have
    \begin{equation*}
    \partial_{f_i} \mathcal{F}^\alpha_i(0, 0, \lambda)h_i= \sum_{n=2}^{\infty} n\sigma_n \left(a^i_n \sin (n x)-d^i_n \cos (n x)\right).
    \end{equation*}
To prove that $\partial_{f_i}\mathcal{F}^\alpha_i(0, 0, \lambda) h_i : X^{k+\alpha-1} \to Y_0^{k-1}$ is an isomorphism, we first note that the sequence $\{\sigma_n\}$, is both monotonically increasing and bounded below by a positive constant, see for instance Lemma \ref{A-1}. This ensures that the kernel of $\partial_{f_i} \mathcal{F}^\alpha_i(0, 0, \lambda)$ is trivial. 

Next, we aim to show that for any $p_i(x) \in Y_0^{k-1}$, there exists an $h_i(x) \in X^{k+\alpha-1}$ such that $\partial_{f_i} \mathcal{F}^\alpha_i(0, 0, \lambda) h_i = p_i$. According to the first part of the lemma, if $p_i$ can be expressed as
\[
p_i(x) = \sum\limits_{n=2}^\infty \tilde{a}_n^i \sin(n x) + \tilde{d}_n^i \cos(n x),
\]
then $h_i$ must satisfy
\[
h_i(x) = \sum\limits_{n=2}^\infty \left( a_n^i \sigma_n^{-1} n^{-1} \cos(n x) + d_n^i \sigma_n^{-1} n^{-1} \sin(n x) \right).
\]
By utilizing the asymptotic expansion of the Gamma function, we obtain $\sigma_n = O(n^{\alpha-1})$ for $\alpha \in (1, 2)$ (see Lemma \ref{A-1}). This leads to the following estimate
    \begin{equation*}
        \begin{aligned}
            \|h_i\|_{X^{k+\alpha-1}}&=\sum\limits_{n=2}^\infty ((\tilde{a}^i_n)^2+(\tilde{d}^i_n)^2)\sigma_n^{-2}n^{-2}n^{2k+2\alpha-2}\\
            &\leq  C\sum\limits_{n=2}^\infty ((\tilde{a}^i_n)^2+(\tilde{d}^i_n)^2)n^{-2(\alpha-1)}n^{-2}n^{2k+2\alpha-2}\\
            &\leq  C\sum\limits_{n=2}^\infty ((\tilde{a}^i_n)^2+(\tilde{d}^i_n)^2)n^{2k-2}\le C\|p_i\|_{Y_0^{k-1}},
        \end{aligned}
    \end{equation*}
which completes the proof of the first part.

Additionally, by \eqref{4-1}, we find that $\partial_{f_j} \mathcal{F}^\alpha_i(0, 0, \lambda) h_i = 0$ for $j \neq i$. Thus, we can write
\[
D_{f} \mathcal{F}^\alpha(0, 0, \lambda) = \operatorname{diag}\left( \partial_{f_1} \mathcal{F}^\alpha_1(0, 0, \lambda), \ldots, \partial_{f_N} \mathcal{F}^\alpha_N(0, 0, \lambda) \right),
\]
which implies that $D_{f} \mathcal{F}^\alpha(0, 0, \lambda)$ is an isomorphism from $\mathcal{X}^{k+\alpha-1}$ to $\mathcal{Y}_0^{k-1}$.  Thus, the proof is complete.
\end{proof}

\section{Existence of solutions}\label{section4}

\subsection{Integral identities for the stream function and the functional \texorpdfstring{$\mathcal{F}^\alpha$}{F\^{}alpha}}\label{sec:Gidentities}

In this subsection $\boldsymbol x,\boldsymbol y$ denote points of $\mathbb R^2$. We consider the stream function associated with the vortex patches \eqref{intial-vort},
\begin{equation}\label{eq:psi}
\psi_\varepsilon(\boldsymbol x)
=
-\frac1{\varepsilon^2}
\sum_{k=1}^N
\frac{\gamma_k}{b_k^2}
\int_{\mathcal D_k^\varepsilon}
K_\alpha(\boldsymbol x-\boldsymbol y)\,d\boldsymbol y,
\qquad \boldsymbol x\in\mathbb R^2,
\end{equation}
so that $\psi_\varepsilon=-(-\Delta)^{-1+\frac\alpha2}\theta_0^\varepsilon$ and $v^\varepsilon=\nabla^\perp\psi_\varepsilon$, in agreement with \eqref{1-1}. Here $K_\alpha$ is given by \eqref{eq:kernel}. The identities derived below explain the degeneracy of the functional defining the relative equilibria. We give a direct proof, based on the structure of the kernel $K_\alpha$.

\begin{lemma}\label{lem:idens}
For all $\varepsilon\in(-\varepsilon_0,\varepsilon_0)\setminus\{0\}$, the stream function \eqref{eq:psi} satisfies
    \begin{enumerate}[label=\rm(\roman*)]
    \item $\displaystyle\sum_{j=1}^N
\frac{\gamma_j}{\varepsilon^2b_j^2}
\int_{\partial \mathcal D_j^\varepsilon}
\psi_\varepsilon(\boldsymbol x)\,
\mathbf n(\boldsymbol x)\,
d\sigma(\boldsymbol x)
=0,$
\item $\displaystyle\sum_{j=1}^N
\frac{\gamma_j}{\varepsilon^2b_j^2}
\int_{\partial \mathcal D_j^\varepsilon}
\big(\boldsymbol x^\perp\!\cdot \mathbf n(\boldsymbol x)\big)\,
\psi_\varepsilon(\boldsymbol x)\,
d\sigma(\boldsymbol x)
=0.$
    \end{enumerate}
\end{lemma}

\begin{proof}
We first assume that $K_\alpha$ is replaced by a smooth radial kernel $K$. The general case follows by approximating $K_\alpha$ by $K_\alpha*\rho_\delta$, where $\rho_\delta$ is a radial mollifier. Since $K_\alpha\in L^1_{\rm loc}(\mathbb R^2)$, the corresponding stream functions converge uniformly on compact sets, and the identities only involve the values of $\psi_\varepsilon$ on the boundaries.

By the divergence theorem,
\[
\begin{split}
\sum_{j=1}^N\frac{\gamma_j}{\varepsilon^2b_j^2}
\int_{\partial \mathcal D_j^\varepsilon}
\psi_\varepsilon\,\mathbf n\,d\sigma
&=\sum_{j=1}^N\frac{\gamma_j}{\varepsilon^2b_j^2}\int_{\mathcal D_j^\varepsilon}\nabla\psi_\varepsilon(\boldsymbol x)\,d\boldsymbol x\\
&=-\sum_{j,k=1}^N\frac{\gamma_j\gamma_k}{\varepsilon^4b_j^2b_k^2}
\int_{\mathcal D_j^\varepsilon}
\int_{\mathcal D_k^\varepsilon}
\nabla K(\boldsymbol x-\boldsymbol y)\,d\boldsymbol y\,d\boldsymbol x.
\end{split}
\]
Since $K$ is radial, $\nabla K$ is odd, so the last double sum changes sign when the roles of $(j,\boldsymbol x)$ and $(k,\boldsymbol y)$ are interchanged. Hence it vanishes, which proves (i).

For (ii), since $\nabla\cdot \boldsymbol x^\perp=0$, the divergence theorem applied to the vector field $\boldsymbol x^\perp\psi_\varepsilon$ gives
\[
\begin{split}
\sum_{j=1}^N\frac{\gamma_j}{\varepsilon^2b_j^2}
\int_{\partial \mathcal D_j^\varepsilon}
(\boldsymbol x^\perp\!\cdot \mathbf n)\psi_\varepsilon\,d\sigma
&=\sum_{j=1}^N\frac{\gamma_j}{\varepsilon^2b_j^2}\int_{\mathcal D_j^\varepsilon}\boldsymbol x^\perp\cdot\nabla\psi_\varepsilon\,d\boldsymbol x\\
&=-\frac12\sum_{j,k=1}^N\frac{\gamma_j\gamma_k}{\varepsilon^4b_j^2b_k^2}\int_{\mathcal D_j^\varepsilon}\int_{\mathcal D_k^\varepsilon}(\boldsymbol x-\boldsymbol y)^\perp\cdot\nabla K(\boldsymbol x-\boldsymbol y)\,d\boldsymbol y\,d\boldsymbol x,
\end{split}
\]
where we symmetrized as before. Since $\nabla K(\boldsymbol z)$ is parallel to $\boldsymbol z$, the integrand vanishes identically, and (ii) follows.
\end{proof}

We next show that the functional $\mathcal F^\alpha$ can be expressed through the relative stream function
\begin{equation}\label{relative-stream-function}
\Psi_\varepsilon(\boldsymbol x) := \psi_\varepsilon(\boldsymbol x)-\frac{\Omega}{2}|\boldsymbol x|^2 - U x_1,\qquad \boldsymbol x=(x_1,x_2)\in\mathbb R^2,
\end{equation}
which satisfies $\nabla^\perp\Psi_\varepsilon=v^\varepsilon-\Omega\boldsymbol x^\perp-U\mathbf e_2$. Recall that by  \eqref{param}, $\partial\mathcal D_i^\varepsilon$ is parametrized by $z_i(x)=w_i+\varepsilon b_iR_i(x)(\cos x,\sin x)$.

\begin{lemma}\label{lem:psiG}
For all $(\varepsilon,f,\lambda)\in(-\varepsilon_0,\varepsilon_0)\times\mathcal B_X\times\Lambda$ with $\varepsilon\neq0$ and all $i=1,\dots,N$,
\begin{equation}\label{psiUG}
\partial_x\big[\Psi_\varepsilon(z_i(x))\big]=-\varepsilon b_i\,\mathcal F^\alpha_i(\varepsilon,f,\lambda)(x),\qquad x\in[0,2\pi).
\end{equation}
In particular, $\int_0^{2\pi}\mathcal F^\alpha_i(\varepsilon,f,\lambda)(x)\,dx=0$.
\end{lemma}

\begin{proof}
By the chain rule and the identity $\boldsymbol a\cdot\boldsymbol b=\boldsymbol a^\perp\cdot\boldsymbol b^\perp$,
\[
\partial_x\big[\Psi_\varepsilon(z_i(x))\big]=\nabla\Psi_\varepsilon(z_i)\cdot z_i'=\nabla^\perp\Psi_\varepsilon(z_i)\cdot (z_i')^\perp=\big(v^\varepsilon(z_i)-\Omega z_i^\perp-U\mathbf e_2\big)\cdot(z_i')^\perp .
\]
The computations of Section~\ref{section2} remain valid for $\varepsilon<0$ (the parametrization \eqref{param} is still counterclockwise, and $|z_i(x)-z_i(y)|^\alpha=|\varepsilon|^\alpha b_i^\alpha|R_i(x)(\cos x,\sin x)-R_i(y)(\cos y,\sin y)|^\alpha$). They give $v^\varepsilon(z_i)\cdot(z_i')^\perp=-\varepsilon b_i\big(\mathcal F^\alpha_{i,2}+\mathcal F^\alpha_{i,3}\big)$,
\begin{gather*}
-\Omega z_i^\perp\cdot(z_i')^\perp=-\Omega z_i\cdot z_i'=-\varepsilon b_i\Omega\big[\varepsilon b_iR_iR_i'+w_i\cdot(R_i'(\cos x,\sin x)+R_i(-\sin x,\cos x))\big],\\
-U\mathbf e_2\cdot(z_i')^\perp=-\varepsilon b_iU\big(R_i'\cos x-R_i\sin x\big).
\end{gather*}
Comparing with \eqref{f1}, we obtain \eqref{psiUG}. Since $x\mapsto\Psi_\varepsilon(z_i(x))$ is $2\pi$-periodic, integrating \eqref{psiUG} over a period gives $\int_0^{2\pi}\mathcal F^\alpha_i\,dx=0$. Note that this proof only uses the definitions of Section~\ref{section2}.
\end{proof}

We can now establish the integral identities which will be used in the proof of our main theorems. We denote by $|\mathcal O_i^\varepsilon|=\frac12\int_0^{2\pi}R_i(x)^2\,dx$ the area of the domain bounded by the curve $x\mapsto R_i(x)(\cos x,\sin x)$, so that $|\mathcal D_i^\varepsilon|=\varepsilon^2b_i^2|\mathcal O_i^\varepsilon|$.

\begin{lemma}\label{identities2}
Let \((\varepsilon,f,\lambda) \in (-\varepsilon_0, \varepsilon_0) \times \mathcal{B}_{X} \times \Lambda\).
 \begin{enumerate}[label=\rm(\roman*)]
  \item If $\Omega=0$, then
    \begin{align}
      \label{eqn:Gident:trans}
\sum_{i=1}^N\gamma_i\int_{0}^{2\pi}\mathcal{F}^\alpha_i(\varepsilon,f,\lambda)(x)\,R_i(x)\,(-\sin x,\cos x)\,dx
     =U\Big(\sum_{i=1}^N\gamma_i|\mathcal O_i^\varepsilon|\Big)\mathbf e_1 .
    \end{align}
In particular, $\displaystyle\sum_{i=1}^N\gamma_i\int_{0}^{2\pi}\mathcal F^\alpha_i(\varepsilon,f,\lambda)(x)R_i(x)\cos x\,dx=0$, and, if moreover $U=0$, also $\displaystyle\sum_{i=1}^N\gamma_i\int_{0}^{2\pi}\mathcal F^\alpha_i(\varepsilon,f,\lambda)(x)R_i(x)\sin x\,dx=0$.
  \item If $U=0$, then
    \begin{align}
      \label{eqn:Gident:rot}
  \sum_{i=1}^N\gamma_i\int_{0}^{2\pi}\mathcal{F}^\alpha_i(\varepsilon,f,\lambda)(x)\,R_i(x)\Big[w_i\cdot(\cos x,\sin x)+\frac{\varepsilon b_i}{2}R_i(x)\Big]dx=0.
    \end{align}
  \end{enumerate}
\end{lemma}
\begin{proof}
By continuity, it is sufficient to consider $\varepsilon \neq 0$. Along $\partial\mathcal D_j^\varepsilon$ we have $\mathbf n\,d\sigma=-(z_j'(x))^\perp dx$, since the parametrization is counterclockwise. Hence, integrating by parts and using Lemma~\ref{lem:psiG},
\begin{equation}\label{bdry-int}
\begin{split}
\int_{\partial\mathcal D_j^\varepsilon}\Psi_\varepsilon\,\mathbf n\,d\sigma&=-\int_0^{2\pi}\Psi_\varepsilon(z_j)(z_j')^\perp dx=\int_0^{2\pi}\partial_x\big[\Psi_\varepsilon(z_j)\big]z_j^\perp dx\\
&=-\varepsilon^2b_j^2\int_0^{2\pi}\mathcal F^\alpha_j\,R_j\,(-\sin x,\cos x)\, dx,
\end{split}
\end{equation}
where the contribution of $w_j^\perp$ vanishes because $\mathcal F^\alpha_j$ has zero mean.

(i) If $\Omega=0$, then $\Psi_\varepsilon=\psi_\varepsilon-Ux_1$, and by the divergence theorem and Lemma~\ref{lem:idens}(i),
\[
\begin{split}
\sum_{j=1}^N\frac{\gamma_j}{\varepsilon^2b_j^2}\int_{\partial\mathcal D_j^\varepsilon}\Psi_\varepsilon\,\mathbf n\,d\sigma&=-U\sum_{j=1}^N\frac{\gamma_j}{\varepsilon^2b_j^2}\int_{\mathcal D_j^\varepsilon}\nabla x_1\,d\boldsymbol x\\
&=-U\sum_{j=1}^N\frac{\gamma_j}{\varepsilon^2b_j^2}|\mathcal D_j^\varepsilon|\,\mathbf e_1=-U\sum_{j=1}^N\gamma_j|\mathcal O_j^\varepsilon|\,\mathbf e_1 .
\end{split}
\]
Combining this with \eqref{bdry-int} gives \eqref{eqn:Gident:trans}. Its second component gives the identity with $\cos x$ and, if $U=0$, its first component gives the identity with $\sin x$.

(ii) If $U=0$, then, since $\boldsymbol x^\perp\cdot\nabla|\boldsymbol x|^2=0$, Lemma~\ref{lem:idens}(ii) gives
\[
\sum_{j=1}^N\frac{\gamma_j}{\varepsilon^2b_j^2}\int_{\partial\mathcal D_j^\varepsilon}(\boldsymbol x^\perp\cdot\mathbf n)\Psi_\varepsilon\,d\sigma=\sum_{j=1}^N\frac{\gamma_j}{\varepsilon^2b_j^2}\int_{\mathcal D_j^\varepsilon}\boldsymbol x^\perp\cdot\nabla\Psi_\varepsilon\,d\boldsymbol x=\sum_{j=1}^N\frac{\gamma_j}{\varepsilon^2b_j^2}\int_{\mathcal D_j^\varepsilon}\boldsymbol x^\perp\cdot\nabla\psi_\varepsilon\,d\boldsymbol x=0 .
\]
On the other hand, $(\boldsymbol x^\perp\cdot\mathbf n)\,d\sigma=-z_j^\perp\cdot(z_j')^\perp dx=-z_j\cdot z_j'\,dx=-\frac12\partial_x|z_j|^2dx$, so that, integrating by parts and using Lemma~\ref{lem:psiG},
\[
\begin{split}
\int_{\partial\mathcal D_j^\varepsilon}(\boldsymbol x^\perp\cdot\mathbf n)\Psi_\varepsilon\,d\sigma&=\frac12\int_0^{2\pi}\partial_x\big[\Psi_\varepsilon(z_j)\big]|z_j|^2dx\\
&=-\varepsilon^2b_j^2\int_0^{2\pi}\mathcal F^\alpha_j\,R_j\Big[w_j\cdot(\cos x,\sin x)+\frac{\varepsilon b_j}{2}R_j\Big]dx,
\end{split}
\]
where we expanded $|z_j|^2=|w_j|^2+2\varepsilon b_jR_j\,w_j\cdot(\cos x,\sin x)+\varepsilon^2b_j^2R_j^2$ and used that $\mathcal F^\alpha_j$ has zero mean. Multiplying by $\gamma_j/(\varepsilon^2b_j^2)$ and summing over $j$ gives \eqref{eqn:Gident:rot}.
\end{proof}

\begin{remark}\label{rem:pv-identities}
For $(\varepsilon,f)=(0,0)$ we have $\mathcal F^\alpha_i(0,0,\lambda)=\mathcal P_i^\alpha(\lambda)\cdot(-\sin x,\cos x)$ by Proposition~\ref{equivalence}, and the identities of Lemma~\ref{identities2} reduce to the classical identities for point vortices $\displaystyle\sum_{i=1}^N\gamma_i\mathcal P_i^\alpha(\lambda)=U\big(\sum_i\gamma_i\big)\mathbf e_1$ if $\Omega=0$, and $\displaystyle\sum_{i=1}^N\gamma_i\,\mathcal P_i^\alpha(\lambda)\cdot w_i^\perp=0$ if $U=0$. They reflect the invariance of \eqref{ode-sys0} under translations and rotations. Note that the $\mathbf e_1$-component of \eqref{eqn:Gident:trans} does not vanish in general when $U\neq0$. Therefore only the $\mathbf e_2$-component can be used in the translating case.
\end{remark}

\subsection{Existence of Vortex Patch Equilibria for the gSQG equations}

In this subsection, we present a detailed formulation of Theorem \ref{thm:general}, employing a modified implicit function theorem as described in Lemma~\ref{abstract-lemma}, specifically for the gSQG equations with \(1<\alpha<2\). The stationary case is characterized by a greater level of degeneracy than the case of rigid motion. As a result, we will treat this case separately. The findings related to solutions corresponding to rigidly rotating or traveling vortex patches are articulated as follows.

\begin{theorem} \label{existence}
Let \(\alpha \in (1,2)\) and let \(\lambda^*\) be a non-degenerate rotating or traveling solution of the \(N\)-vortex problem \eqref{alg-sysP}, in the sense of Definition~\ref{def:non-deg}{\rm (i)}. The following statements hold:
\begin{enumerate}[label=\rm(\roman*)]
    \item There exist $\varepsilon_1 > 0$ and a unique $C^1$ function $(f, \lambda_1) : (-\varepsilon_1, \varepsilon_1) \to \mathcal{B}_X \times \mathbb{R}^{2N-1}$ such that
    \begin{equation} \label{sol_g-alpha}
    \mathcal{F}^\alpha \big(\varepsilon, f(\varepsilon), \lambda_1(\varepsilon), \lambda_2^*\big) = 0,
    \end{equation}
    where $\lambda_1(\varepsilon) = \lambda_1^* + o(\varepsilon)$ and, in $X^{k+\alpha-1}$,
    \[
    f_i(\varepsilon,x) =\varepsilon {b_i} \,\Xi_\alpha \sum_{j=1, j\neq i}^N \frac{\gamma_j}{\gamma_i} \frac{2\big((w_i-w_j)\cdot(\cos x,\sin x)\big)^2-|w_i-w_j|^2}{|w_i - w_j|^{\alpha + 4}}+ o(\varepsilon),
    \]
    with $\Xi_\alpha := \dfrac{(\alpha + 2) \Gamma(1 - \frac{\alpha}{2}) \Gamma(3 - \frac{\alpha}{2})}{4 \Gamma(2 - \alpha)}$, where $w_j$, $\gamma_j$ denote the components of $\lambda^*$.
 \item These solutions exhibit the following symmetries
\[f( - \varepsilon )(x) = -f(\varepsilon )(x+\pi),\qquad  \lambda_1( - \varepsilon ) = \lambda_1(\varepsilon ).\]
      \item For all $\varepsilon \in (-\varepsilon_1, \varepsilon_1) \setminus \{0\}$, the domains $\mathcal{O}_i^\varepsilon$, whose boundaries are parametrized by $x\mapsto R_i(x)(\cos x,\sin x)$ with $R_i(x) = 1 + \varepsilon |\varepsilon|^\alpha b_i^{1 + \alpha} f_i(\varepsilon,x)$, are strictly convex.
\end{enumerate}
\end{theorem}

\begin{proof}
By Propositions~\ref{p3-1} and \ref{lem2-3}, $\mathcal F^\alpha$ can be extended to a $C^1$ mapping $(-\varepsilon_0,\varepsilon_0)\times\mathcal B_X\times\Lambda\to\mathcal Y^{k-1}$. By Propositions~\ref{equivalence} and \ref{iso}, for any \( h \in \mathcal{X}^{k+\alpha-1} \) and \( \dot{\lambda}_1 \in \mathbb{R}^{2N-1} \), we have
\begin{equation}\label{diff-g-gen}
D_{(f,\lambda_1)}\mathcal{F}^\alpha(0,0,\lambda^*)( h, \dot\lambda_1)(x)=
 D_{f}\mathcal{F}^\alpha(0,0,\lambda^*)h(x)+
 \Big(\big(D_{\lambda_1}\mathcal{P}_i^\alpha(\lambda^*)\dot\lambda_1\big)\cdot(-\sin x,\cos x)\Big)_{i=1}^N,
\end{equation}
where \( D_f \mathcal{F}^\alpha(0, 0, \lambda^*) \) is an isomorphism from \( \mathcal{X}^{k+\alpha-1} \) to \( \mathcal{Y}_0^{k-1} \). Let
\[
\mathbb{Y} := \Big\{ x \mapsto \big(a_1^i \sin x + d_1^i \cos x\big)_{i=1}^N : (a_1^i, d_1^i)_{i=1}^N \in \mathbb{R}^{2N} \Big\},
\]
so that
\begin{equation}\label{www2}
\mathcal{Y}^{k-1}=\mathcal{Y}_0^{k-1} \oplus \mathbb{Y}.
\end{equation}
Since the map $(c_i)_{i}\mapsto(c_i\cdot(-\sin x,\cos x))_{i}$ is an isomorphism from $\mathbb R^{2N}$ onto $\mathbb Y$ and, by assumption, \( D_{\lambda_1} \mathcal{P}^\alpha(\lambda^*) \) is injective with a range of codimension one, the second operator in \eqref{diff-g-gen} is injective and its range has codimension one in $\mathbb Y$. Thus,
\begin{equation}\label{ker-ran-DGb}
\operatorname{codim} \operatorname{ran} D_{(f,\lambda_1)}\mathcal{F}^\alpha(0,0,\lambda^*) = 1 \quad \textnormal{and} \quad  \ker D_{(f,\lambda_1)}\mathcal{F}^\alpha(0,0,\lambda^*)=\{ 0\}.
\end{equation}
In the case of pure translation (i.e., \( \Omega^* = 0 \) and \( U^* \neq 0 \)), we define
 \begin{equation}\label{PHI12}
\Phi(\varepsilon,f, g,\lambda):=\sum_{i=1}^N\frac{\gamma_i}{\pi} \int_{0}^{2\pi}g_i(x)\,R_i(x)\cos x\,dx,
\end{equation}
and in the case of pure rotation (i.e., \( \Omega^* \neq 0 \) and \( U^* = 0 \)), we define
\begin{equation}\label{PHI22}
\Phi(\varepsilon,f, g,\lambda):=
   \sum_{i=1}^N\frac{\gamma_i}{\pi}\int_{0}^{2\pi}g_i(x)\,R_i(x)\Big[w_i\cdot(\cos x,\sin x)+\frac{\varepsilon b_i}{2}R_i(x)\Big]dx ,
\end{equation}
with \( g = (g_1, \ldots, g_N) \in \mathcal{Y}^{k-1} \) and $R_i=1+\varepsilon|\varepsilon|^\alpha b_i^{1+\alpha}f_i$. The mapping \( \Phi \colon (-\varepsilon_0, \varepsilon_0) \times \mathcal{B}_X \times \mathcal{Y}^{k-1} \times \Lambda \to \mathbb{R} \) is \( C^1 \) and linear in $g$, so that
\begin{align}
\label{condition-12b}
\Phi\big(\varepsilon, f, 0, \lambda\big) &= 0 .
\end{align}
Moreover, $\Omega$ (respectively $U$) is a component of $\lambda_2$ and vanishes at $\lambda_2^*$. Hence Lemma~\ref{identities2} gives, for all $(\varepsilon,f,\lambda_1)$ close to $(0,0,\lambda_1^*)$,
\begin{align}
\label{condition-22b}
\Phi\big(\varepsilon, f, \mathcal{F}^\alpha(\varepsilon, f, \lambda_1,\lambda_2^*), \lambda_1,\lambda_2^*\big) &= 0.
\end{align}

By differentiating equations \eqref{PHI12} with respect to \( g \) in the direction \( \tilde{g} = (\tilde{g}_1, \ldots, \tilde{g}_N) \in \mathcal{Y}^k \), where
\begin{equation}
\label{expang2}
\tilde{g}_i(x) = \sum_{n=1}^{\infty} \left(a^i_n \sin(n x) + d^i_n \cos(n x)\right), \quad i = 1, \ldots, N,
\end{equation}
we find, for \( \Omega = 0 \) and \( U \neq 0 \),
\begin{align*}
D_{g} \Phi(0, 0, 0, \lambda) \tilde{g} &= \sum_{i=1}^N \frac{\gamma_i}{\pi} \int_{0}^{2\pi} \tilde{g}_i(x) (\cos(x), \sin(x)) \, dx =\sum_{i=1}^N \gamma_i\, d^i_1
\end{align*}
In fact, by using \eqref{3-7} and subsequently computing the Gateaux derivatives of \eqref{PHI12} in the direction of \( \tilde{g}_i \), one obtains
\begin{equation}\label{eq:omega6}
    \begin{split}
       &  \partial_{g_i}\Phi (\varepsilon,f,g,\lambda)\tilde g_i= \frac{\gamma_i}{\pi}   \displaystyle\int_{0}^{2\pi}\big(1+\varepsilon\abs{\varepsilon}^\alpha b_i^{1+\alpha}  f_i(x)\big)\cos(x)\, \tilde{g}_i(x)dx \\
&
 =\frac{\gamma_i}{\pi} \displaystyle\int_{0}^{2\pi}\tilde{g}_i(x)  (\cos(x),\sin(x))\,dx +\varepsilon\mathcal{R}_{\Phi}\\
  &
 =\frac{\gamma_i}{\pi} \displaystyle\int_{0}^{2\pi}\left(a^i_n\sin(n x)+d^i_n \cos (n x)\right)  (\cos(x),\sin(x))\,dx=\gamma_i d^i_1 +\varepsilon\mathcal{R}_{\Phi}
    \end{split} 
\end{equation}
 Similarly, we can compute the Gateaux derivative on the case   $\Omega \neq 0$ and $U=0$, which is given by
 \begin{equation}\label{eq:omega8}
    \begin{split}
       &  \partial_{g_i}\Phi (\varepsilon,f,g,\lambda)\tilde g_i\\
       &
       \qquad=  \frac{\gamma_i}{\pi}  \displaystyle\int_{0}^{2\pi}\left[\frac{\varepsilon b_i}{2} (1+\varepsilon\abs{\varepsilon}^\alpha b_i^{1+\alpha}f_i(x))^2+ (1+\varepsilon\abs{\varepsilon}^\alpha b_i^{1+\alpha}f_i(x))w_i\cdot(\cos(x),\sin(x))\right]\, \tilde{g}_i(x)dx \\
&
  \qquad=\frac{\gamma_i}{\pi} \displaystyle\int_{0}^{2\pi}\tilde{g}_i(x) w_i \cdot(\cos(x),\sin(x))\,dx +\varepsilon\mathcal{R}_{\Phi}\\
 &
\qquad = \frac{\gamma_i}{\pi} \displaystyle\int_{0}^{2\pi}\left(a^i_n\sin(n x)+d^i_n \cos (n x)\right) w_i \cdot(\cos(x),\sin(x))\,dx=\gamma_i w_i\cdot(a^i_1,d^i_1) +\varepsilon\mathcal{R}_{\Phi}\\
 &
 \qquad =\gamma_i w_i\cdot (a^i_{1},d^i_{1})+\varepsilon\mathcal{R}_{\Phi} .
    \end{split} 
\end{equation}
Then, we have  for $\Omega \neq 0$ and $U=0$
 \begin{align*}
 D_{g}\Phi (0,0,0,\lambda)\tilde g &=\sum_{i=1}^N\frac{\gamma_i}{\pi}\int_{0}^{2\pi}\tilde g_i(x) w_i\cdot(\cos(x),\sin(x))\,dx= \sum_{i=1}^N \gamma_i\big(w_{i1}d^i_1+w_{i2}a_1^i\big).
\end{align*}
In the rotating case at least one of the distinct points $w_i$ is different from the origin. Since all $\gamma_i\neq0$, in either case
\begin{align}
\label{condition-3}
\operatorname{ran}\big[ D_{g} \Phi(0, 0, 0, \lambda^*) \big] &= \mathbb{R}.
\end{align}
We now apply Lemma~\ref{abstract-lemma} with $n=1$, $X=\mathbb R$, $Y=\mathcal X^{k+\alpha-1}\times\mathbb R^{2N-1}$, $Z=\mathcal Y^{k-1}$,
\[
G\big(\varepsilon,(f,\mu)\big):=\mathcal F^\alpha(\varepsilon,f,\lambda_1^*+\mu,\lambda_2^*),\qquad F\big(g,\varepsilon,(f,\mu)\big):=\Phi(\varepsilon,f,g,\lambda_1^*+\mu,\lambda_2^*).
\]
The assumptions \eqref{eqn:G0}--\eqref{eqn:ranF} follow from Proposition~\ref{equivalence}, \eqref{ker-ran-DGb}, \eqref{condition-12b}, \eqref{condition-22b} and \eqref{condition-3}. This proves the existence and the uniqueness in {\rm (i)}.

 Differentiating $\mathcal F^\alpha \big(\varepsilon, f(\varepsilon), \lambda_1(\varepsilon), \lambda_2^*\big) = 0$ with respect to \( \varepsilon \) at \( \varepsilon=0 \), we find
\begin{align}
\label{diff-g-gen-2}
D_{(f, \lambda_1)} \mathcal{F}^\alpha(0, 0, \lambda^*)\, \partial_\varepsilon \big( f(\varepsilon), \lambda_1(\varepsilon) \big) \Big|_{\varepsilon=0} &= -\partial_\varepsilon \mathcal{F}^\alpha \big( 0, 0, \lambda^* \big).
\end{align}
Since $\mathcal F^\alpha_{i,1}(\varepsilon,0,\lambda)$ and $\mathcal F^\alpha_{i,2}(\varepsilon,0,\lambda)$ do not depend on $\varepsilon$, only $\mathcal F^\alpha_{i,3}$ contributes to $\partial_\varepsilon \mathcal{F}^\alpha_i$. 

Considering \eqref{333}, for all \( \alpha \in (1, 2) \), we have
\begin{equation*}
    \begin{split}
         & {\resizebox{.99\hsize}{!}{$\mathcal{F}^\alpha_{i,3} =-\frac{\alpha C_\alpha }{2}\displaystyle\sum_{j \neq i} \frac{\gamma_j}{b_j}\displaystyle\fint \frac{  B_{i,j}\sin( x- y)d y }{\left(A_{i, j}\right)^{\frac{\alpha}{2}+1}}-\frac{\alpha (\alpha+2)C_\alpha }{4}\sum_{j \neq i}    \frac{\gamma_j}{b_j}\displaystyle\fint\int_0^1\frac{\varepsilon  \tau B^2_{i,j}\sin( x- y)d y d\tau }{\left(A_{i, j}+t\tau\varepsilon B_{i, j}( x,  y)\right)^{\frac{\alpha}{2}+2}}+\varepsilon \abs{\varepsilon}^\alpha\mathcal{R}_{i, 3}(\varepsilon ,f,\lambda)$}}\\
&
=-\frac{\alpha C_\alpha }{2}\sum_{j \neq i}  \frac{\gamma_j}{b_j} \displaystyle\fint \frac{  B_{i,j}\sin( x- y)d y }{\left(A_{i, j}\right)^{\frac{\alpha}{2}+1}}\\
&
 \qquad{\resizebox{.93\hsize}{!}{$\frac{-\alpha (\alpha+2)C_\alpha}{4} \displaystyle\sum_{j \neq i}    \frac{\gamma_j}{b_j}\displaystyle\fint\frac{\varepsilon  (w_i-w_j)^2(b_i(\cos x, \sin x)-b_j(\cos y, \sin y))^2\sin( x- y)d y }{\left(A_{i, j}\right)^{\frac{\alpha}{2}+2}}+\varepsilon \abs{\varepsilon}^\alpha\mathcal{R}_{i, 3}(\varepsilon ,f,\lambda)$}}\\
 &
 =-\frac{\alpha C_\alpha }{2}\sum_{j \neq i}   \frac{\gamma_j}{b_j}\displaystyle\fint \frac{  B_{i,j}\sin( x- y)d y }{\left(A_{i, j}\right)^{\frac{\alpha}{2}+1}}\\
 &
 \qquad{\resizebox{.93\hsize}{!}{$\frac{-\alpha (\alpha+2)C_\alpha }{4}\displaystyle\sum_{j \neq i}  \frac{\gamma_j}{b_j}  \displaystyle\fint\frac{\varepsilon  (w_i-w_j)^2(b_i^2+b_j^2-2b_ib_j\cos(x)\cos(y)-2b_ib_j\sin(x)\sin(y))\sin(x-y)d y }{\abs{w_i-w_j}^{\alpha+4}}+\varepsilon \abs{\varepsilon}^\alpha\mathcal{R}_{i, 3}(\varepsilon ,f,\lambda)$}}\\
  &
 ={\resizebox{.95\hsize}{!}{$\frac{\alpha C_\alpha}{2}\displaystyle\sum_{j \neq i} \gamma_j \frac{\left(w_i-w_j\right) \cdot(\sin  (x),-\cos  (x))}{\left|w_i-w_j\right|^{\alpha+2}}+\alpha (\alpha+2)C_\alpha  \displaystyle\sum_{j \neq i}\gamma_j  b_i \frac{\varepsilon  (w_i-w_j)^2 }{\abs{w_i-w_j}^{\alpha+4}}\sin(x)\cos(x)+\varepsilon \abs{\varepsilon}^\alpha\mathcal{R}_{i, 3}(\varepsilon ,f,\lambda)$}}
    \end{split}
\end{equation*}
Thus, we can conclude that
\begin{align}
\label{f0dif-eps}
\begin{split}
   & \partial_\varepsilon \mathcal{F}^\alpha_i(0, 0, \lambda^*)(x) = \\
    &
  \quad  \frac{\alpha(\alpha+2)C_\alpha}{2}\,b_i\sum_{j=1, j \neq i}^N \gamma_j \frac{\big((w_i-w_j)\cdot(\cos x,\sin x)\big)\times\big((w_i-w_j)\cdot(-\sin x,\cos x)\big)}{|w_i - w_j|^{\alpha + 4}} .
\end{split}
\end{align}
We stress that, at this order, the only contribution comes from the square of the scalar product $(w_i-w_j)\cdot\big(b_i(\cos x,\sin x)-b_j(\cos y,\sin y)\big)$. Since $\partial_\varepsilon \mathcal{F}^\alpha(0, 0, \lambda^*) \in \mathcal{Y}_0^{k-1}$, the decomposition \eqref{www2}, the injectivity of \( D_{\lambda_1} \mathcal{P}^\alpha(\lambda^*) \) and \eqref{diff-g-gen} give
\[
\partial_\varepsilon \lambda_1(\varepsilon) \big|_{\varepsilon=0} = 0 \qquad\text{and}\qquad D_f\mathcal F^\alpha(0,0,\lambda^*)\,\partial_\varepsilon f(\varepsilon)\big|_{\varepsilon=0}=-\partial_\varepsilon \mathcal{F}^\alpha(0, 0, \lambda^*).
\]

For $v=(v_1,v_2)\in\mathbb R^2$ we have
\begin{gather*}
2\big(v\cdot(\cos x,\sin x)\big)^2-|v|^2=(v_1^2-v_2^2)\cos2x+2v_1v_2\sin2x,\\
-2\big(v\cdot(\cos x,\sin x)\big)\big(v\cdot(-\sin x,\cos x)\big)=(v_1^2-v_2^2)\sin2x-2v_1v_2\cos2x .
\end{gather*}
By Proposition~\ref{iso}, $\partial_{f_i}\mathcal F^\alpha_i(0,0,\lambda^*)$ maps $\cos 2x$ to $2\gamma_i\sigma_2\sin2x$ and $\sin2x$ to $-2\gamma_i\sigma_2\cos 2x$, hence it maps $2\big(v\cdot(\cos x,\sin x)\big)^2-|v|^2$ to $-4\gamma_i\sigma_2\big(v\cdot(\cos x,\sin x)\big)\big(v\cdot(-\sin x,\cos x)\big)$. Taking $v=w_i-w_j$, we obtain
\[
\partial_\varepsilon f_i(\varepsilon) \big|_{\varepsilon=0}(x) = \frac{\alpha(\alpha+2) C_\alpha}{8\sigma_2}\, b_i\sum_{j=1, j \neq i}^N \frac{\gamma_j}{\gamma_i} \frac{2\big((w_i-w_j)\cdot(\cos x,\sin x)\big)^2-|w_i-w_j|^2}{|w_i - w_j|^{\alpha + 4}} .
\]
Finally, \eqref{sigma} gives $\sigma_2=2^{\alpha-1}\Gamma(2-\alpha)\Gamma(1+\frac\alpha2)\big/\big(\Gamma(1-\frac\alpha2)^2\Gamma(3-\frac\alpha2)\big)$, so that, by \eqref{eqn:kalpha2},
\begin{equation}
\label{gam1hatg}
\frac{\alpha C_\alpha}{\sigma_2} = \frac{2\,\Gamma\left(1 - \frac{\alpha}{2}\right) \Gamma\left(3 - \frac{\alpha}{2}\right)}{\Gamma(2 - \alpha)},\qquad\text{hence}\qquad \frac{\alpha(\alpha+2)C_\alpha}{8\sigma_2}=\Xi_\alpha .
\end{equation}
Since $(f,\lambda_1)$ is $C^1$, $f(0)=0$ and $\partial_\varepsilon\lambda_1(0)=0$, this completes the proof of {\rm (i)}. Note that $\Xi_0=1$, which is consistent with the Euler case $\alpha=0$ \cite{multipole}, and that the patches are elongated along the lines joining the vortices when $\gamma_i\gamma_j>0$.

 Let $\varepsilon\neq0$, $(f,\lambda)\in\mathcal B_X\times\Lambda$, and set $\tilde f_i(x):=-f_i(x+\pi)$. Then $1+(-\varepsilon)|\varepsilon|^\alpha b_i^{1+\alpha}\tilde f_i(x)=R_i(x+\pi)$ and
\[
w_i+(-\varepsilon)\,b_i\,R_i(x+\pi)\,(\cos x,\sin x)=w_i+\varepsilon b_iR_i(x+\pi)\,(\cos(x+\pi),\sin(x+\pi))=z_i(x+\pi).
\]
Thus the parameters $(\varepsilon,f,\lambda)$ and $(-\varepsilon,\tilde f,\lambda)$ describe the same domains $\mathcal D_i^\varepsilon$ and, since $\theta_0^\varepsilon$ only depends on $\varepsilon$ through $\varepsilon^2$, the same vorticity and the same relative stream function $\Psi_\varepsilon$. Lemma~\ref{lem:psiG} then yields
\[
\varepsilon b_i\,\mathcal F^\alpha_i(-\varepsilon,\tilde f,\lambda)(x)=\partial_x\big[\Psi_\varepsilon(z_i(x+\pi))\big]=-\varepsilon b_i\,\mathcal F^\alpha_i(\varepsilon,f,\lambda)(x+\pi),
\]
that is,
\begin{equation}\label{g-eps}
\mathcal F^\alpha_i(-\varepsilon,\tilde f,\lambda)(x)=-\mathcal F^\alpha_i(\varepsilon,f,\lambda)(x+\pi),
\end{equation}
and by continuity \eqref{g-eps} also holds for $\varepsilon=0$. The map $f\mapsto\tilde f$ is an isometry of $\mathcal X^{k+\alpha-1}$ that preserves $\mathcal B_X$. Hence, if $(f(\varepsilon),\lambda_1(\varepsilon))$ is the solution given by {\rm(i)}, then $(\widetilde{f(\varepsilon)},\lambda_1(\varepsilon))$ solves the problem at $-\varepsilon$, and the uniqueness in {\rm(i)} gives $f(-\varepsilon)=\widetilde{f(\varepsilon)}$ and $\lambda_1(-\varepsilon)=\lambda_1(\varepsilon)$. This proves {\rm(ii)}. In particular $\lambda_1$ is even, which again shows that $\lambda_1(\varepsilon)=\lambda_1^*+o(\varepsilon)$.

 We now prove {\rm (iii)}. The signed curvature at the point $x$ of the curve $x\mapsto R_i(x)(\cos x,\sin x)$ is
\begin{align*}
\kappa _{i}(x)&=\frac{R_i(x)^{2}+2R_i'(x)^{2}-R_i(x)R_i''(x)}{\left( R_i(x)^{2}+R_i'(x)^{2}\right) ^{\frac{3}{2}}}\\
&=\frac{\splitfrac{\big(1+\varepsilon|\varepsilon|^\alpha b_i^{1+\alpha}f_i(x)\big)^{2}+2\varepsilon^2|\varepsilon|^{2\alpha}b_i^{2+2\alpha}f_i'(x)^{2}}{-\varepsilon|\varepsilon|^\alpha b_i^{1+\alpha}f_i''(x)\big(1+\varepsilon|\varepsilon|^\alpha b_i^{1+\alpha}f_i(x)\big)}}{\Big( \big(1+\varepsilon|\varepsilon|^\alpha b_i^{1+\alpha}f_i(x)\big)^{2}+\varepsilon^2|\varepsilon|^{2\alpha}b_i^{2+2\alpha}f_i'(x)^{2}\Big) ^{\frac{3}{2}}}.
\end{align*}
Hence, we get $\kappa_i(x)=1+O(|\varepsilon|^{1+\alpha})$ uniformly in $x\in[0,2\pi)$. Hence $\kappa_i>0$ for $\varepsilon$ small, and the domains $\mathcal O_i^\varepsilon$ are strictly convex. This completes the proof of Theorem~\ref{existence}.
\end{proof}

Now, we consider the stationary case \( \Omega = U = 0 \), when $1<\alpha<2$. The following result holds.

\begin{theorem}\label{theorem-Phi-stationary}
Let \( \alpha \in (1,2) \), and let \( \lambda^* \) be a non-degenerate stationary solution of the \( N \)-vortex problem \eqref{alg-sysP}, in the sense of Definition~\ref{def:non-deg}{\rm (ii)}. Then the conclusions of Theorem~\ref{existence} still hold, except that \( \lambda_1(\varepsilon) \) now takes values in \( \mathbb{R}^{2N-3} \) instead of \( \mathbb{R}^{2N-1} \).
\end{theorem}

\begin{proof}
We only prove the existence and uniqueness. The asymptotic expansion, the symmetry and the convexity follow exactly as in Theorem~\ref{existence}.

From Proposition~\ref{iso}, the assumptions on \( D_{\lambda_1} \mathcal{P}^\alpha(\lambda^*) \), \eqref{diff-g-gen} and \eqref{www2}, we deduce that
\begin{equation}\label{ker-ran-DG2}
\operatorname{codim} \, \operatorname{ran} \, D_{(f, \lambda_1)} \mathcal{F}^\alpha(0, 0, \lambda^*) = 3 \quad \text{and} \quad \ker D_{(f, \lambda_1)} \mathcal{F}^\alpha(0, 0, \lambda^*) = \{ 0 \}.
\end{equation}
For all \( (\varepsilon, f, g,\lambda) \in (-\varepsilon_0, \varepsilon_0) \times \mathcal{B}_X\times\mathcal Y^{k-1} \times \Lambda \), we define
\begin{align*}
\widetilde\Phi(\varepsilon, f, g, \lambda) := \sum_{i=1}^N \frac{\gamma_i}{\pi}
\begin{pmatrix}
\displaystyle \int_{0}^{2\pi}g_i(x)\,R_i(x) \cos x\, dx \\[1ex]
\displaystyle\int_{0}^{2\pi}g_i(x)\,R_i(x)\sin x \,dx \\[1ex]
\displaystyle\int_{0}^{2\pi}g_i(x)\,R_i(x)\Big[w_i\cdot(\cos x,\sin x)+\frac{\varepsilon b_i}{2}R_i(x)\Big]dx
\end{pmatrix}.
\end{align*}
The map \( \widetilde\Phi \colon (-\varepsilon_0, \varepsilon_0) \times \mathcal{B}_X \times\mathcal{Y}^{k-1}  \times \Lambda \to \mathbb{R}^3 \) is continuously differentiable and satisfies \( \widetilde\Phi ( \varepsilon, f, 0, \lambda ) = 0\). Since $\Omega$ and $U$ are components of $\lambda_2$ and vanish at $\lambda_2^*$, Lemma~\ref{identities2} gives
\begin{align} \label{condition-22}
\widetilde\Phi \big( \varepsilon, f, \mathcal{F}^\alpha(\varepsilon, f, \lambda_1,\lambda_2^*), \lambda_1,\lambda_2^* \big) &= 0.
\end{align}
For \( \tilde{g}\in\mathcal Y^{k-1} \) expanded as in the proof of Theorem~\ref{existence}, we find
\begin{align*}
D_{g} \widetilde\Phi(0, 0, 0, \lambda^*) \tilde{g} =  \sum_{i=1}^N \gamma_i \begin{pmatrix} d^i_{1}\\  a^i_{1} \\ w_{i1}  d^i_{1} +w_{i2}\, a^i_{1} \end{pmatrix}.
\end{align*}
This linear map is onto $\mathbb R^3$. Otherwise there would be $(c_1,c_2,c_3)\neq0$ such that $c_1+c_3w_{i1}=c_2+c_3w_{i2}=0$ for all $i$ (recall $\gamma_i\neq0$). This forces $c_3\neq0$ and $w_i=-(c_1,c_2)/c_3$ for all $i$, which is impossible since the points $w_i$ are distinct and $N\geq2$. Thus
\begin{align} \label{condition-32}
\operatorname{ran} \left[ D_{g} \widetilde\Phi(0, 0, 0, \lambda^*) \right] &= \mathbb{R}^3,
\end{align}
and Lemma~\ref{abstract-lemma} with $n=3$, together with \eqref{ker-ran-DG2}--\eqref{condition-32}, gives the required result.
\end{proof}
\subsection{Existence of Vortex Patch Equilibria for the SQG equation}
Let $\mathcal{B}_{X}$ be an open neighborhood of zero in the space $X^{k+\log}$,
i.e.,
\begin{equation*}
\mathcal{B}_{X}:=\left\{ f\in \mathcal{X}^{k+\log}: \|f\|_{\mathcal{X}^{k+\log}(\mathbb{T})}< 1 .\right\}
\end{equation*}%
The following lemma establishes the \( C^1 \) regularity of the functional \( \mathcal{F}^1_i \). This result is analogous to Proposition \ref{p3-1} and Proposition \ref{lem2-3}, but in this instance, we focus on the case \(\alpha = 1\).

\begin{lemma}\label{lem2-7}
Let \(\alpha=1\) and let \(\lambda^*\) be a solution to \eqref{alg-sysP}. There exist \(\varepsilon_0 > 0\) and a small neighborhood \(\Lambda\) of \(\lambda^*\) such that \(\mathcal{F}^1\) can be extended to a \(C^1\) mapping
$
\left(-\varepsilon_0, \varepsilon_0\right) \times \mathcal{B}_X \times \Lambda \to \mathcal Y^{k-1}.
$
\end{lemma}

\begin{proof}
We begin by outlining the proof of continuity. Similar to the gSQG equations, it is easy to see that \( \mathcal{F}^1_{i,1} \) maps \( \left(-\varepsilon_0, \varepsilon_0\right) \times \mathcal{B}_X \times \Lambda \) to \( \mathcal{B}_Y \), with the expression
\[
\mathcal{F}^1_{i,1} = \big(\Omega w_i+U\mathbf e_1\big)\cdot (-\sin x,\cos x)+ \varepsilon |\varepsilon| \mathcal{R}_{i1}(\varepsilon, f_i),
\]
where \( \mathcal{R}_1(\varepsilon, f_i) \) is continuous. From this, we conclude that \( \mathcal{F}^1_{i,1} \) is continuous.

Since \( f_i \in X^{k+\log} \), and by applying the Taylor expansion \eqref{taylor} to \( \mathcal{F}^1_{i,2} \), we can address potential singularities at \( \varepsilon = 0 \). As noted in the proof of Proposition \ref{p3-1}, the most singular term is \( \mathcal{F}^1_{i22} \). Therefore, we must compute the \( \partial^{k-1} \) derivatives of \( \mathcal{F}^1_{i22} \).
\begin{equation*}
\begin{split}
& {\resizebox{.96\hsize}{!}{$\partial ^{k-1}\mathcal{F}^1_{i22}=C_\alpha\gamma _{i}\displaystyle\fint\frac{%
(\partial ^{k}f_{i}(y)-\partial ^{k}f_{i}(x))\cos (x-y)dy}{\left( A( x, y)+\varepsilon|\varepsilon|b_i^{2} B\left( f_{i}, x, y\right)\right)^{\frac{1}{2}}}  -C_\alpha\gamma _{i}\varepsilon |\varepsilon |b_{i}^{2}\displaystyle\fint\frac{\cos (x-y)}{\left( A( x, y)+\varepsilon|\varepsilon|b_i^{2} B\left( f_{i}, x, y\right)\right)^{\frac{3}{2}}} $}}\\
& \ \ \ \ {\resizebox{.96\hsize}{!}{$\times \left(\varepsilon |\varepsilon|
b_i^{2}(f_i(x)-f_i(y))(f'_i(x)-f'_i(y))+2((1+\varepsilon|\varepsilon|
b_i^{2}f_i(x))f'_i(y)+(1+\varepsilon|\varepsilon|
b_i^{2}f_i(y))f'_i(x))\sin^2(\frac{x-y}{2})\right)$}} \\
& \ \ \ \ \times (\partial ^{k-1}f_{i}(y)-\partial ^{k-1}f_{i}(x))dy+l.o.t,
\end{split}%
\end{equation*}%
By the Sobolev embedding theorem, we have that 
$
\| \partial^m f_i \|_{L^\infty} \leq C \| f_i \|_{X^{k+\log}} < \infty
$
for \( m = 0, 1, 2 \) and \( k \geq 3 \). Using the mean value theorem along with Hölder's inequality, we obtain
\[
\begin{aligned}
\left\| \partial^{k-1} \mathcal{F}^1_{i22} \right\|_{L^2} &\leq C \left\| \int \!\!\!\!\!\!\! \!\!\!\!\;{}-{} \frac{\partial^k f_i(x) - \partial^k f_i(y)}{|4 \sin \left(\frac{x - y}{2}\right)|} \, dy \right\|_{L^2} \\
&\quad + C \left\| \int \!\!\!\!\!\!\! \!\!\!\!\;{}-{} \frac{\partial^{k-1} f_i(x) - \partial^{k-1} f_i(y)}{|4 \sin \left(\frac{x - y}{2}\right)|} \, dy \right\|_{L^2} \\
&\leq C \| f_i \|_{X^{k+\log}} + C \| f_i \|_{X^{k+\log-1}} < \infty.
\end{aligned}
\]
Thus, we conclude that the range of \( \mathcal{F}^1_{i22} \) belongs to \( Y^{k-1} \).

By the results and notations established in Proposition \ref{p3-1}, specifically equation \eqref{2-7}, we only need to verify the continuity of the most singular term, \( \mathcal{F}^1_{i22} \). More specifically, for \( f_{i1}, f_{i2} \in X^{k+\log} \), we have 
\[
\| \mathcal{F}^1_{i22}(\varepsilon, f_{i1}) - \mathcal{F}^1_{i22}(\varepsilon, f_{i2}) \|_{Y^{k-1}} \leq C \| f_{i1} - f_{i2} \|_{X^{k+\log}}.
\]
Finally, following the same approach as in the proof of Proposition \ref{p3-1}, we obtain the continuity of the nonlinear functional \( \mathcal{F}^1_{i22} : (-\varepsilon_0, \varepsilon_0) \times \mathcal{B}_X \times \Lambda \to \mathcal{B}_Y \), and by invoking the Taylor formula \eqref{taylor}, we derive the following expression for \( \mathcal{F}^1_{i2} \)
\begin{equation}
{\resizebox{.94\hsize}{!}{$
\mathcal{F}^1_{i,2}=\frac{\gamma _{i}}{4}\displaystyle\fint   \frac{f_i( x- y)\sin( y)d y}{\left(\sin^2(\frac{ y}{2})\right)^{\frac{1}{2}}}-\frac{\gamma _{i}}{2}\displaystyle\fint \frac{(f'_i( x)-f'_i( x- y))\cos( y)d y}{\left(\sin^2(\frac{ y}{2})\right)^{\frac{1}{2}}}+\varepsilon\mathcal{R}_{i,2}(\varepsilon ,f_i),$}}  \label{2-21}
\end{equation}%
where $\mathcal{R}_{2}:\left(-\varepsilon_0, \varepsilon_0\right) \times \mathcal{B}_X\times \Lambda \rightarrow \mathcal{B}_Y$ is continuous.

Similarly, we can prove that \( \mathcal{F}^1_{i,3} : (-\varepsilon_0, \varepsilon_0) \times \mathcal{B}_X \times \Lambda \rightarrow \mathcal{B}_Y \) is continuous and can be written as:
\begin{equation}
    \mathcal{F}^1_{i,3} = \sum_{j \neq i} \frac{\gamma_j (w_i - w_j)}{2 |w_i - w_j|^3} \cdot (\sin(x), -\cos(x)) + \varepsilon \mathcal{R}_{i,3}(\varepsilon, f, \lambda),
\end{equation}
where \( \mathcal{R}_{i,3} : (-\varepsilon_0, \varepsilon_0) \times \mathcal{B}_X \times \Lambda \rightarrow \mathcal{B}_Y \) is also continuous. Hence, the proof of the continuity of the functional \( \mathcal{F}^1_i \) is complete.

Moreover, the continuity of \( \partial_{f_i} \mathcal{F}^1_i(\varepsilon, f, x, \lambda) h_i : (-\varepsilon_0, \varepsilon_0) \times \mathcal{B}_X \times \Lambda \rightarrow \mathcal Y^{k-1} \) and \( \partial_{f_j} \mathcal{F}^1_i(\varepsilon, f, x, \lambda) h_j : (-\varepsilon_0, \varepsilon_0) \times \mathcal{B}_X \times \Lambda \rightarrow \mathcal Y^{k-1} \) follows in the same manner as Proposition \ref{lem2-3}, by employing the same reasoning as in the case where \( 1 < \alpha < 2 \). Thus, the proof is omitted.
\end{proof}

As in \eqref{4-1}, setting $\varepsilon=0$ and $f_i \equiv 0$ for all $i=1,\ldots,N$, and using $C_1=1$, the Gateaux derivatives are given by
\begin{equation}\label{4-1b}
    \left\{
    \begin{aligned}
        &\partial_{f_i} \mathcal{F}^1_i(0, 0, \lambda) h_i =
        \dfrac{\gamma_i}{2} \fint\frac{h_i( x-  y)\sin( y)\,d y}{\left(4\sin^2(\frac{ y}{2})\right)^{\frac{1}{2}}}-  \gamma_i \fint \frac{(h'_i( x)-h'_i( x- y))\cos( y)\,d y}{\left(4\sin^2(\frac{ y}{2})\right)^{\frac{1}{2}}}, \\
        &
\partial_{f_j}\mathcal{F}^1_i(0, 0, \lambda) h_j = 0, \,\,\, j \not= i.
    \end{aligned}
    \right.
\end{equation}

\begin{proposition}\label{iso2}
    Let $\alpha=1$ and $h=(h_1,\ldots,h_N)\in \mathcal{X}^{k+\log} $, where
    \[h_i( x)=\sum_{n=2}^{\infty}\left(a^i_n \cos (n  x)+d^i_n \sin (n  x)\right).\]
    Then
    \begin{equation*}
            \partial_{f_i}  \mathcal{F}^1_i(0, 0, \lambda) h_i = \gamma_i\sum_{n=2}^{\infty} n \sigma_n  \left(a^i_n \sin (n  x)-d^i_n \cos (n  x)\right), \qquad
            \partial_{f_j}  \mathcal{F}^1_i(0, 0, \lambda) h_j = 0, \quad j \neq i,
    \end{equation*}
    where
\begin{equation}\label{sigma1}
\sigma _{n}=\frac{2}{\pi }\sum_{l=2}^{n}\frac{1}{2l-1},\qquad n\geq1 \quad(\text{so that } \sigma_1=0).
\end{equation}
  Furthermore,  the operator $\partial_{f_{i}} \mathcal{F}^1_i(0, 0, \lambda):X^{k+\log}\rightarrow Y^{k-1}_0$ is an isomorphism. Moreover,  the Gateaux derivative of $\mathcal{F}^1$ with respect to $f$ at $(0,0,\lambda)$ is given by
      \begin{align}\label{eq:linearization2}
    D_f \mathcal{F}^1(0, 0, \lambda)h(x)=&\sum_{n=2}^{\infty}n \sigma_n \begin{pmatrix}
      \gamma_1   \left(a^1_n \sin (n  x)-d^1_n \cos (n  x)\right) \\
      \vdots
      \\
       \gamma_N   \left(a^N_n \sin (n  x)-d^N_n \cos (n  x)\right)
    \end{pmatrix} .
  \end{align}
 Additionally, for any $\lambda \in \Lambda$, the linear operator 
  $D_{f}\mathcal{F}^1 (0,0,\lambda)\colon \mathcal{X}^{k+\log}\rightarrow \mathcal{Y}^{k-1}_0$ is also an isomorphism.
\end{proposition}

\begin{proof}
   The Gateaux derivative of $\mathcal{F}^1_i(0, 0, \lambda)$ at the
directions $h_{i}$ is given by
\begin{equation*}
\partial_{f_i} \mathcal{F}^1_i(0, 0, \lambda) h_i =
        \dfrac{\gamma_i}{2} \displaystyle\int\!\!\!\!\!\!\!\!\!\; {}-{}\frac{h_i( x-  y)\sin( y)d y}{\left(4\sin^2(\frac{ y}{2})\right)^{\frac{1}{2}}}-  \gamma_i \displaystyle\int\!\!\!\!\!\!\!\!\!\; {}-{} \frac{(h'_i( x)-h'_i( x- y))\cos( y)d y}{\left(4\sin^2(\frac{ y}{2})\right)^{\frac{1}{2}}},
\end{equation*}%
and
\begin{equation*}
\partial_{f_j} \mathcal{F}^1_i(0, 0, \lambda) h_i=0.
\end{equation*}%
Then, by applying the same computations as in Proposition \ref{iso}, we obtain
    \begin{equation*}
    \partial_{f_i} \mathcal{F}^1_i(0, 0, \lambda)h_i= \sum_{n=2}^{\infty} n\sigma_n \left(a^i_n \sin (n x)-d^i_n \cos (n x)\right).
    \end{equation*}    where
\begin{equation*}
\sigma _{n}=\frac{2}{\pi }\sum\limits_{l=2}^{n}\frac{1}{2l-1}.
\end{equation*}%
To demonstrate that \( \partial_{f_i} \mathcal{F}^1_i(0, 0, \lambda) h_i : X^{k+\log} \to Y_0^{k-1} \) is an isomorphism, we first observe that the sequence \( \{\sigma_n\} \) is monotonically increasing and bounded below by a positive constant, as outlined in Lemma \ref{A-1}. This observation implies that the kernel of \( \partial_{f_i} \mathcal{F}^1_i(0, 0, \lambda) \) is trivial.

Next, we aim to prove that for any \( p_i(x) \in Y_0^{k-1} \), there exists an \( h_i(x) \in X^{k+\log} \) such that \( \partial_{f_i} \mathcal{F}^1_i(0, 0, \lambda) h_i = p_i \). According to the first part of the lemma, if \( p_i \) can be written as
\[
p_i(x) = \sum\limits_{n=2}^\infty \tilde{a}_n^i \sin(n x) + \tilde{d}_n^i \cos(n x),
\]
then the function \( h_i \) must satisfy
\[
h_i(x) = \sum\limits_{n=2}^\infty \left( a_n^i \sigma_n^{-1} n^{-1} \cos(n x) + d_n^i \sigma_n^{-1} n^{-1} \sin(n x) \right).
\]
Using the asymptotic expansion of the Gamma function, we deduce that \( \sigma_n = O(\log(n)) \) for \( \alpha = 1 \) (as seen in Lemma \ref{A-1}). Consequently, we derive the following estimate
    \begin{equation*}
        \begin{aligned}
            \|h_i\|^2_{X^{k+\log}}&=\sum\limits_{n=2}^\infty ((\tilde{a}^i_n)^2+(\tilde{d}^i_n)^2)\sigma^{-2}_n n^{2k-2}(1+\log(n))^2\\
            &\leq  C\sum\limits_{n=2}^\infty ((\tilde{a}^i_n)^2+(\tilde{d}^i_n)^2)n^{2k-2}\left(\frac{1+\log(n)}{\log(n)}\right)^2\\
            &\leq  C\sum\limits_{n=2}^\infty ((\tilde{a}^i_n)^2+(\tilde{d}^i_n)^2)n^{2k-2}\le C\|p_i\|_{Y_0^{k-1}},
        \end{aligned}
    \end{equation*}
Furthermore, by \eqref{4-1b}, we observe that $\partial_{f_j} \mathcal{F}^1_i(0, 0, \lambda) h_i = 0$ for all $j \neq i$. Therefore, we can express
\[
D_{f} \mathcal{F}^1(0, 0, \lambda) = \operatorname{diag}\left( \partial_{f_1} \mathcal{F}^1_1(0, 0, \lambda), \ldots, \partial_{f_N} \mathcal{F}^1_N(0, 0, \lambda) \right),
\]
which indicates that \( D_{f} \mathcal{F}^1(0, 0, \lambda) \) is an isomorphism from \( \mathcal{X}^{k+\log} \) to \( \mathcal{Y}_0^{k-1} \). Thus, the proof is concluded.
\end{proof}

Now, we present a more detailed formulation of Theorem \ref{thm:general} for the SQG equation. For rigidly rotating or traveling vortex patches, the result is as follows.
 
\begin{theorem} \label{existenceb}
Consider \(\alpha = 1\), and let \(\lambda^*\) be a non-degenerate solution of the \(N\)-vortex problem \eqref{alg-sysP}, according to Definition~\ref{def:non-deg}{\rm (i)}, where at least one of \(\Omega\) or \(U\) is non-zero. Then, the subsequent statements are true
\begin{enumerate}[label=\rm(\roman*)]
    \item There exists a small $\varepsilon_1 > 0$ and a unique $C^1$ function $(f, \lambda_1) : (-\varepsilon_1, \varepsilon_1) \to \mathcal{B}_X \times \mathbb R^{2N-1}$ such that
    \begin{equation} \label{sol_g-alphab}
    \mathcal{F}^1 \big(\varepsilon, f(\varepsilon), \lambda_1(\varepsilon), \lambda_2^*\big) = 0,
    \end{equation}
    where $\lambda_1(\varepsilon) = \lambda_1^* + O(\varepsilon)$ and 
  \[
f_i(\varepsilon,x) = \frac{9\pi}{16}\,\varepsilon {b_i}  \sum_{j=1, j\neq i}^N \frac{\gamma_j}{\gamma_i} \frac{2\big((w_i-w_j)\cdot(\cos x,\sin x)\big)^2-|w_i-w_j|^2}{|w_i - w_j|^{5}} + o(\varepsilon) .
\]
    \item   These solutions exhibit the following symmetries
\[f( - \varepsilon )(x) = -f(\varepsilon )(x+\pi),\qquad  \lambda_1( - \varepsilon ) = \lambda_1(\varepsilon ).\]
      \item For all $\varepsilon \in (-\varepsilon_1, \varepsilon_1) \setminus \{0\}$, the domains $\mathcal{O}_i^\varepsilon$, with boundaries parametrized by $R_i(x) = 1 + \varepsilon |\varepsilon| b_i^{2} f_i(x) : \mathbb{T} \to \partial \mathcal{O}_i^\varepsilon$, exhibit strict convexity.
\end{enumerate}
\end{theorem}

\begin{proof}
The proof is identical to that of Theorem~\ref{existence}, using Lemma~\ref{lem2-7} and Proposition~\ref{iso2} instead of Propositions~\ref{p3-1}, \ref{lem2-3} and \ref{iso}. Lemmas~\ref{lem:psiG} and \ref{identities2} hold for $\alpha=1$. Thus, we have
\begin{align}
\label{f0dif-epsc}
\partial_\varepsilon \mathcal{F}^1_i(0, 0, \lambda^*)(x) &= \frac32\, b_i\sum_{j=1, j \neq i}^N \gamma_j  \frac{\big((w_i-w_j)\cdot(\cos x,\sin x)\big)\big((w_i-w_j)\cdot(-\sin x,\cos x)\big)}{|w_i - w_j|^{5}} ,
\end{align}
and, since $\sigma_2=\frac{2}{3\pi}$ by \eqref{sigma1}, we obtain
\[
\begin{split}
\partial_\varepsilon f_i(\varepsilon) \big|_{\varepsilon=0}(x) &= \frac{3}{8\sigma_2}\, b_i\sum_{j=1, j \neq i}^N \frac{\gamma_j}{\gamma_i} \frac{2\big((w_i-w_j)\cdot(\cos x,\sin x)\big)^2-|w_i-w_j|^2}{|w_i - w_j|^{5}}\\
&=\frac{9\pi}{16}\, b_i\sum_{j=1, j \neq i}^N \frac{\gamma_j}{\gamma_i} \frac{2\big((w_i-w_j)\cdot(\cos x,\sin x)\big)^2-|w_i-w_j|^2}{|w_i - w_j|^{5}},
\end{split}
\]
in agreement with $\Xi_1=\frac{9\pi}{16}$. The symmetry and the convexity are proved as in Theorem~\ref{existence}.
\end{proof}

We now focus on the stationary case where \( \Omega = U = 0 \).

\begin{theorem}\label{theorem-Phi-stationaryb}
Let \( \alpha = 1 \) and let \( \lambda^* \) be a non-degenerate stationary solution of the \( N \)-vortex problem \eqref{alg-sysP}, in the sense of Definition~\ref{def:non-deg}{\rm (ii)}. Then the conclusions of Theorem~\ref{existenceb} remain valid, except that \( \lambda_1(\varepsilon) \) now takes values in \( \mathbb{R}^{2N-3} \) instead of \( \mathbb{R}^{2N-1} \).
\end{theorem}

\begin{proof}
The proof is the same as that of Theorem~\ref{theorem-Phi-stationary}, using Proposition~\ref{iso2} instead of Proposition~\ref{iso}.
\end{proof}

\section{Applications of  vortex equilibria}\label{section5}

There exist several point vortex equilibria that fulfill the non-degeneracy condition of Definition~\ref{def:non-deg}. In this section we verify this condition for the asymmetric pairs and tripoles considered in \cite{multipole}, and we prove Theorem~\ref{thm:informal-pair} by applying Theorems~\ref{existence}, \ref{theorem-Phi-stationary}, \ref{existenceb} and \ref{theorem-Phi-stationaryb}. We first show that, for configurations which are symmetric with respect to the horizontal axis, the patches inherit this symmetry.

For $\boldsymbol x=(x_1,x_2)\in\mathbb R^2$ let $\boldsymbol x^\flat:=(x_1,-x_2)$. For $f\in\mathcal B_X$ and $\lambda$ we define $f^\flat_i(x):=f_i(-x)$ and $\lambda^\flat$ by replacing each $w_i$ by $w_i^\flat$, the other components of $\lambda$ being unchanged.

\begin{lemma}\label{lem:reflection}
For all $(\varepsilon,f,\lambda)\in(-\varepsilon_0,\varepsilon_0)\times\mathcal B_X\times\Lambda$ with $\lambda^\flat\in\Lambda$,
\begin{equation}\label{ref-sym}
\mathcal F^\alpha_i(\varepsilon,f^\flat,\lambda^\flat)(x)=-\mathcal F^\alpha_i(\varepsilon,f,\lambda)(-x),\qquad i=1,\dots,N.
\end{equation}
Consequently, if $(\lambda^*)^\flat=\lambda^*$, that is, if all the points $w_i^*$ lie on the horizontal axis, then the solutions given by Theorems~\ref{existence}, \ref{theorem-Phi-stationary}, \ref{existenceb} and \ref{theorem-Phi-stationaryb} satisfy $w_{i2}(\varepsilon)=0$ for every component $w_{i2}$ of $\lambda_1$, and $f_i(\varepsilon)(-x)=f_i(\varepsilon)(x)$. In particular, the domains $\mathcal O_i^\varepsilon$ and $\mathcal D_i^\varepsilon$ are symmetric with respect to the horizontal axis.
\end{lemma}

\begin{proof}
Let $\varepsilon\neq0$. The reflection $\boldsymbol x\mapsto\boldsymbol x^\flat$ maps the domain bounded by $z_i$ (with parameters $(\varepsilon,f,\lambda)$) onto the domain bounded by
\[
z_i^\flat(x):=\big(z_i(-x)\big)^\flat=w_i^\flat+\varepsilon b_iR_i(-x)\,(\cos x,\sin x),
\]
which is the boundary associated with $(\varepsilon,f^\flat,\lambda^\flat)$. Since $K_\alpha$ is radial, the stream function of the reflected vortex patches is $\psi_\varepsilon(\boldsymbol x^\flat)$. Moreover, $|\boldsymbol x^\flat|=|\boldsymbol x|$ and $(\boldsymbol x^\flat)_1=x_1$, so the corresponding relative stream function is $\Psi^\flat_\varepsilon(\boldsymbol x)=\Psi_\varepsilon(\boldsymbol x^\flat)$. By Lemma~\ref{lem:psiG},
\[
-\varepsilon b_i\,\mathcal F^\alpha_i(\varepsilon,f^\flat,\lambda^\flat)(x)=\partial_x\big[\Psi^\flat_\varepsilon(z_i^\flat(x))\big]=\partial_x\big[\Psi_\varepsilon(z_i(-x))\big]=\varepsilon b_i\,\mathcal F^\alpha_i(\varepsilon,f,\lambda)(-x),
\]
which proves \eqref{ref-sym} for $\varepsilon\neq0$; the case $\varepsilon=0$ follows by continuity. Now assume $(\lambda^*)^\flat=\lambda^*$. Then $(\lambda_2^*)^\flat=\lambda_2^*$, and the maps $f\mapsto f^\flat$ and $\lambda_1\mapsto\lambda_1^\flat$ are isometries fixing $0$ and $\lambda_1^*$, respectively. Hence, if $(f(\varepsilon),\lambda_1(\varepsilon))$ is the solution given by the above theorems, then $(f(\varepsilon)^\flat,\lambda_1(\varepsilon)^\flat)$ is also a solution in the uniqueness neighborhood, and the two coincide.
\end{proof}

\vspace{0.2cm}

\subsection{Asymmetric co-rotating pairs.}

The \(N\)-vortex problem \eqref{alg-sysP} presents various configurations, including both co-rotating and counter-rotating vortex pairs. A specific family of asymmetric co-rotating pairs is defined by
\begin{equation}\label{v-pair-rot}
  \lambda^*:=(w_{11}^*,w_{21}^*,w_{12}^*,w_{22}^*,\gamma_1^*,\gamma_2^*,\Omega^*, U^*)=\Big(d,-\mathtt{c} d,0,0,\mathtt{c}\gamma,\gamma,\frac{\gamma \widehat{C}_\alpha}{2d^{\alpha+2}(1+\mathtt{c})^{\alpha+1}},0\Big),
\end{equation}
where \(\gamma\in\mathbb R\setminus\{0\}\), \(d>0\) and \(0<|\mathtt{c}|<1\). One checks directly that $\mathcal P^\alpha(\lambda^*)=0$. Let \(N=2\) and
\[
{\mathcal{P}}^{\alpha} := \bigl(\mathcal{P}_1^\alpha\cdot\mathbf e_1 ,\mathcal{P}_1^\alpha\cdot\mathbf e_2, \mathcal{P}_2^\alpha\cdot\mathbf e_1, \mathcal{P}_2^\alpha\cdot\mathbf e_2\bigr).
\]
The differential of $\mathcal P^\alpha$ with respect to \(\lambda_1 = (w_{11}, w_{21}, w_{22})\) at \(\lambda^*\) is given by
\begin{equation*}
D_{\lambda_1}{\mathcal{P}}^{\alpha}(\lambda^*)\begin{pmatrix}
\dot{w}_{11} \\ \dot{w}_{21} \\ \dot{w}_{22}
\end{pmatrix} = \frac{\gamma \widehat{C}_\alpha}{2d^{\alpha+2}(1+\mathtt{c})^{\alpha+2}}
\begin{pmatrix}
2+\mathtt{c}+\alpha & -(\alpha+1) & 0 \\
0 & 0 & 1 \\
-\mathtt{c}(\alpha+1) & 1+\mathtt{c}(2+\alpha) & 0 \\
0 & 0 & 1
\end{pmatrix}
\begin{pmatrix}
\dot{w}_{11} \\ \dot{w}_{21} \\ \dot{w}_{22}
\end{pmatrix},
\end{equation*}
Removing the fourth row, we obtain a matrix whose determinant is
\[
-\Big(\frac{\gamma \widehat{C}_\alpha}{2d^{\alpha+2}(1+\mathtt{c})^{\alpha+2}}\Big)^3(\alpha+2)(1+\mathtt{c})^2,
\] which does not vanish since \(\mathtt{c} \neq -1\). Consequently, the matrix has rank 3, so its kernel is trivial and its range has codimension one. Moreover, $U$ is a component of $\lambda_2=(w_{12},\gamma_1,\gamma_2,\Omega,U)$ and $U^*=0$, so $\lambda^*$ is non-degenerate in the sense of Definition~\ref{def:non-deg}{\rm(i)}. Therefore, Theorems~\ref{existence} and \ref{existenceb} provide $\varepsilon_1>0$ and a unique \(C^1\) function
\[
(f,\lambda_1) = (f_1, f_2, w_{11}, w_{21}, w_{22}) \colon (-\varepsilon_1, \varepsilon_1) \to \mathcal{B}_X  \times \mathbb{R}^3
\]
satisfying
\begin{equation}\label{sol_g-alpha-pair}
\mathcal{F}^\alpha \big(\varepsilon, f(\varepsilon), \lambda_1(\varepsilon), \lambda_2^*\big) = 0,
\end{equation}
with $w_{11}(\varepsilon)=d+o(\varepsilon)$ and $w_{21}(\varepsilon)=-\mathtt cd+o(\varepsilon)$. Since $(\lambda^*)^\flat=\lambda^*$, Lemma~\ref{lem:reflection} gives $w_{22}(\varepsilon)=0$ and the symmetry of the patches. Together with the convexity in Theorems~\ref{existence} and \ref{existenceb}, this proves Theorem~\ref{thm:informal-pair}{\rm(i)}.

\vspace{0.2cm}

\subsection{Asymmetric traveling pairs}
Now, we consider the vortex pairs
\begin{equation}\label{v-pair-trans}
\lambda^*:=(w_{11}^*,w_{21}^*,w_{12}^*,w_{22}^*,\gamma_1^*,\gamma_2^*,\Omega^*,U^*)=\Big(d,-d,0,0,-\gamma,\gamma,0,\frac{\gamma \widehat{C}_\alpha}{2^{\alpha+2}d^{\alpha+1}}\Big),
\end{equation}
which translate steadily along the vertical axis.
For \(N=2\), the differential of
\(
{\mathcal{P}}^{\alpha}
\)
with respect to \(\lambda_1 = (w_{21}, w_{22}, \gamma_1)\) at $\lambda^*$ is
\begin{equation*}
D_{\lambda_1}{\mathcal{P}}^{\alpha}(\lambda^*)\begin{pmatrix}
\dot{w}_{21} \\ \dot{w}_{22} \\ \dot{\gamma}_1
\end{pmatrix} = \frac{\widehat{C}_\alpha}{2^{\alpha+3}d^{\alpha+2}}
\begin{pmatrix}
- \gamma(\alpha+1) & 0 & 0 \\
0 & \gamma & 0 \\
- \gamma(\alpha+1) & 0 & 2d \\
0 & \gamma & 0
\end{pmatrix}
\begin{pmatrix}
\dot{w}_{21} \\ \dot{w}_{22} \\ \dot{\gamma}_1
\end{pmatrix},
\end{equation*}
Eliminating the last row, we get a matrix with determinant
\[
-2\Big(\frac{\widehat{C}_\alpha}{2^{\alpha+3}d^{\alpha+2}}\Big)^3\gamma^2 d(\alpha+1)\neq0.
\] Hence, the matrix has rank 3, its kernel is trivial and its range has codimension one. Since $\Omega$ is a component of $\lambda_2=(w_{11},w_{12},\gamma_2,\Omega,U)$ and $\Omega^*=0$, the existence of a traveling vortex patch pair for the gSQG equations with $\alpha\in[1,2)$ follows from Theorems~\ref{existence} and \ref{existenceb}, and its symmetry from Lemma~\ref{lem:reflection}. This proves Theorem~\ref{thm:informal-pair}{\rm(ii)}.

\vspace{0.2cm}

\subsection{Stationary tripole}
We also consider asymmetric stationary tripoles of the form
\begin{equation}\label{stationary-tripole}
  \begin{aligned}
    \lambda^*&:=(w_{11}^*,w_{21}^*,w_{31}^*,w_{12}^*,w_{22}^*,w_{32}^*,\gamma_1^*,\gamma_2^*,\gamma_3^*,\Omega^*,U^*)\\ &=\Big(1,0,-\mathtt{a},0,0,0,\gamma ,-\gamma\big(\tfrac{\mathtt{a}}{\mathtt{a}+1}\big)^{\alpha+1},\gamma \mathtt{a}^{\alpha+1},0,0\Big),
  \end{aligned}
\end{equation}
where $\mathtt{a} \in (0,1)$. For \(N=3\), the differential of the mapping
\(
{\mathcal{P}}^{\alpha} := \bigl(\mathcal{P}_{1}^\alpha\cdot\mathbf e_1, \mathcal{P}_{1}^\alpha\cdot\mathbf e_2, \mathcal{P}_{2}^\alpha\cdot\mathbf e_1, \mathcal{P}_{2}^\alpha\cdot\mathbf e_2,\mathcal{P}_{3}^\alpha\cdot\mathbf e_1, \mathcal{P}_{3}^\alpha\cdot\mathbf e_2\bigr)
\)
with respect to \(\lambda_1 = (w_{31}, w_{32}, \gamma_2)\) at \(\lambda^*\) is
{\small
\setlength{\arraycolsep}{3.5pt}
\begin{equation*}
D_{\lambda_1}{\mathcal{P}}^{\alpha}(\lambda^*)\begin{pmatrix}
\dot{w}_{31} \\ \dot{w}_{32} \\ \dot{\gamma}_2
\end{pmatrix} = \frac{\gamma \widehat{C}_\alpha}{2}
\begin{pmatrix}
-(\alpha+1)\mathtt{a}^{\alpha+1}(\mathtt{a}+1)^{-\alpha-2} & 0 & -1/\gamma \\
0 & \mathtt{a}^{\alpha+1}(\mathtt{a}+1)^{-\alpha-2} & 0 \\
-(\alpha+1)\mathtt{a}^{-1} & 0 & 0 \\
0 & \mathtt{a}^{-1} & 0 \\
-(\alpha+1)\mathtt{a}^{-1}(\mathtt{a}+1)^{-\alpha-2} & 0 & \mathtt{a}^{-\alpha-1}/\gamma \\
0 & \mathtt{a}^{-1}(\mathtt{a}+1)^{-\alpha-2} & 0
\end{pmatrix}
\begin{pmatrix}
\dot{w}_{31} \\ \dot{w}_{32} \\ \dot{\gamma}_2
\end{pmatrix}.
\end{equation*}
}
This matrix has rank 3. Indeed, the second column vanishes on the odd rows, while the first and third columns vanish on the even rows. Moreover, the first and third columns are linearly independent, since the third one is nonzero and vanishes on the third row, whereas the first one does not. Since $\Omega$ and $U$ are components of $\lambda_2$, Theorems~\ref{theorem-Phi-stationary} and \ref{theorem-Phi-stationaryb} ensure the existence of a stationary vortex patch tripole for the SQG and gSQG equations, and Lemma~\ref{lem:reflection} gives $w_{32}(\varepsilon)=0$ and the symmetry of the patches. This proves Theorem~\ref{thm:informal-pair}{\rm(iii)}.

\appendix
\section*{ Appendix}
\section{ Auxiliary Results}

We list some auxiliary results for use in the preceding sections. For $0<\alpha<2$ and $n\in \mathbb{N}^+$, let
\begin{equation*}
    I_n(x)=\int_0^{2\pi}\frac{\sin(nx)-\sin(nx-ny)}{\left(\sin(\frac{y}{2})\right)^\alpha}dy,
\end{equation*} 
and
\begin{equation*}
    J_n(x)=\int_0^{2\pi}\frac{\cos(nx)-\cos(nx-ny)}{\left(\sin(\frac{y}{2})\right)^\alpha}dy.
\end{equation*}
Using the fundamental properties of the Euler gamma function, we can compute both $I_n(x)$ and $J_n(x)$ as trigonometric polynomials and analyze the asymptotic behavior of their coefficients. These formulas follow from the computations in \cite{Cas1} (see also \cite{HH15}). We summarize these findings in the following lemma.

\begin{lemma}\label{A-1}
For $0<\alpha<2$ and $n\in \mathbb{N}^+$, $I_n(x)$ and $J_n(x)$ satisfy
\begin{equation*}
    I_n(x) = \beta_n \sin(nx), \quad J_n(x) = \beta_n \cos(nx),
\end{equation*}
where $\beta_n=\displaystyle\int_0^{2\pi}(1-\cos(ny))\sin(\frac y2)^{-\alpha}dy$ and, if $\alpha \neq 1$,
\begin{equation*}
    \beta_n = 2^\alpha \frac{2\pi \Gamma(1-\alpha)}{\Gamma(\frac{\alpha}{2})\Gamma(1-\frac{\alpha}{2})} \left( \frac{\Gamma(\frac{\alpha}{2})}{\Gamma(1-\frac{\alpha}{2})} - \frac{\Gamma(n+\frac{\alpha}{2})}{\Gamma(n+1-\frac{\alpha}{2})} \right),
\end{equation*}
and if $\alpha = 1$,
\begin{equation*}
    \beta_n = \sum_{l=1}^n \frac{8}{2l-1}.
\end{equation*}
Moreover, the sequence $\{\beta_n\}$ is increasing with respect to $n$ and exhibits the following asymptotic behavior for large $n$: there are constants $0<c_1\leq c_2$ such that, for $n\geq2$, $c_1\log n\leq\beta_n\leq c_2\log n$ if $\alpha=1$, and $c_1n^{\alpha-1}\leq\beta_n\leq c_2n^{\alpha-1}$ if $\alpha>1$.
\end{lemma}

We finally recall the following mild generalization of the implicit function theorem, proved in \cite[Lemma 2.6]{multipole}, which is the key tool in the proofs of Theorems~\ref{existence}, \ref{theorem-Phi-stationary}, \ref{existenceb} and \ref{theorem-Phi-stationaryb}.

\begin{lemma}\label{abstract-lemma}
Let $G \colon U \times V \to Z$ and $F \colon Z \times U \times V \to \mathbb{R}^n$ be $C^1$ mappings satisfying the following conditions
\begin{align}
    \label{eqn:G0}
    G(0,0) &= 0, \\
    \label{eqn:F0}
    F(0,x,y) &= 0, \\
    \label{eqn:FG}
    F(G(x,y),x,y) &= 0,
\end{align}
for all $(x,y) \in U \times V$. Here, $X, Y, Z$ are Banach spaces, and $U \subset X$ and $V \subset Y$ are open sets containing the origin. If the linearizations of these mappings at the origin satisfy the following
\begin{align}
    \label{eqn:kerG}
    \ker D_y G(0,0) &= \{0\}, \\
    \label{eqn:ranG}
    \operatorname{codim} \operatorname{ran} D_y G(0,0) &= n, \\
    \label{eqn:ranF}
    \operatorname{ran} D_z F(0,0,0) &= \mathbb{R}^n,
\end{align}
then there exists a neighborhood $\tilde{U} \times \tilde{V}$ of the origin in $U \times V$ and a $C^1$ mapping $g \colon \tilde{U} \to \tilde{V}$ such that
\begin{align}
    g(0) &= 0, \notag \\
    \label{eqn:g}
    G(x,g(x)) &= 0 \quad \text{for all } x \in \tilde{U}.
\end{align}
Moreover, every solution of $G(x,y) = 0$ in $\tilde{U} \times \tilde{V}$ is of the form $(x,g(x))$, and the operator $D_x g(0)$ is uniquely determined by the equation
\begin{align*}
    D_x G(0,0) + D_y G(0,0) D_x g(0) = 0,
\end{align*}
which is obtained by implicitly differentiating \eqref{eqn:g}.
  \end{lemma}


\vspace{0.3cm}

\addcontentsline{toc}{section}{References}


\phantom{s} \thispagestyle{empty}



\end{document}